\documentclass[11pt,reqno]{amsart}
\usepackage{mathrsfs}
\usepackage[american]{babel}
\usepackage{todonotes}
\usepackage{comment}
\usepackage{amsmath}
\usepackage{dsfont,mathtools,amssymb}
\usepackage{appendix}
\usepackage{color}
\usepackage{nicematrix}
\usepackage{tikz}
\usetikzlibrary{decorations.pathreplacing,fit,backgrounds,positioning}
\usepackage[hidelinks]{hyperref}
\mathtoolsset{showonlyrefs,showmanualtags}
\usepackage[shortlabels]{enumitem}
\allowdisplaybreaks

\usepackage[margin=1.2in]{geometry}
\usepackage{array,longtable}

\theoremstyle{plain}
\newtheorem{theorem}{Theorem}[section]
\newtheorem{proposition}[theorem]{Proposition}
\newtheorem{lemma}[theorem]{Lemma}
\newtheorem{corollary}[theorem]{Corollary}
\theoremstyle{remark}
\newtheorem{remark}[theorem]{Remark}
\theoremstyle{definition}

\numberwithin{equation}{section}
\numberwithin{figure}{section}

\renewcommand{\complement}{{\mathrm c}}

\renewcommand{\P}{\mathds{P}}

\newcommand{\R}{\mathbb{R}}

\newcommand{\C}{\mathbb{C}}

\newcommand{\dd}{{\rm d}}
\newcommand{\fstop}{\, \text{.}}
\newcommand{\comma}{\; \text{,}\;\;}

\newcommand{\tonde}[1]{\left(#1\right)}
\newcommand{\ttonde}[1]{\big(#1\big)}
\newcommand{\tttonde}[1]{(#1)}

\newcommand{\abs}[1]{\left\lvert#1\right\rvert}

\newcommand{\emparg}{{\,\cdot\,}}
\newcommand{\emp}{\varnothing}
\newcommand{\eqdef}{\coloneqq}

\newcommand{\car}{\mathds{1}}
\newcommand{\scalar}[2]{\left\langle #1 , #2 \right\rangle}

\newcommand{\ttscalar}[2]{\langle #1 , #2 \rangle}
\newcommand{\norm}[1]{\left\lVert#1\right\rVert}

\newcommand{\ttnorm}[1]{\|#1\|}
\newcommand{\set}[1]{\left\{#1\right\}}
\newcommand{\tset}[1]{\big\{#1\big\}}
\newcommand{\ttset}[1]{\{#1\}}
 
\newcommand{\eps}{\varepsilon} 

\renewcommand{\ker}{{\rm ker}}
\newcommand{\im}{{\rm ran}}
\newcommand{\spec}{{\rm spec}} 
\newcommand{\spanop}{{\rm span}}
\newcommand{\interior}{{\rm int}}
\newcommand{\supp}{{\rm supp}}
\newcommand{\var}{{\rm Var}}
\newcommand{\gap}{{\rm gap}}
\newcommand{\sym}{{{\rm \scriptscriptstyle sym}}}
\newcommand{\per}{{\scriptscriptstyle\otimes}}
\newcommand{\tp}{{\scriptscriptstyle+}}
\newcommand{\tn}{{\scriptscriptstyle[n]}}
\newcommand{\tm}{{\scriptscriptstyle[m]}}
\newcommand{\tups}{{\scriptscriptstyle\Upsilon}}
\newcommand{\tneq}{{\scriptscriptstyle \neq}}
\newcommand{\tvarpi}{{\scriptscriptstyle\varpi}}
\newcommand{\tKMP}{{\rm\scriptscriptstyle KMP}}
\newcommand{\tiem}{{\rm\scriptscriptstyle IEM}}
\newcommand{\thp}{{\rm\scriptscriptstyle HP}}
\newcommand{\tavg}{{\rm\scriptscriptstyle AVG}}
\newcommand{\tbep}{{\rm\scriptscriptstyle BEP}}

\renewcommand{\l}{\lambda}
\newcommand{\cC}{\ensuremath{\mathcal C}} 
\newcommand{\cD}{\ensuremath{\mathcal D}} 
\newcommand{\cE}{\ensuremath{\mathcal E}} 
 
\newcommand{\cG}{\ensuremath{\mathcal G}}

\newcommand{\cL}{\ensuremath{\mathcal L}} 
\newcommand{\cM}{\ensuremath{\mathcal M}}

\newcommand{\cS}{\ensuremath{\mathcal S}}

\title[Aldous' phenomena in stochastic exchange models]{Aldous' spectral gap phenomena\\ in stochastic exchange models}
	\author{Pietro Caputo, Matteo Quattropani,  and Federico Sau}
\address{Pietro Caputo\\ Università degli Studi Roma Tre}
\email{pietro.caputo@uniroma3.it}
\address{Matteo Quattropani\\  Università degli Studi Roma Tre}
\email{matteo.quattropani@uniroma3.it}
\address{Federico Sau\\ Università degli Studi di Milano}
\email{federico.sau@unimi.it}

\begin{document}
	\maketitle
	\thispagestyle{empty}

\begin{abstract}
	We consider a broad class of exchange dynamics with arbitrary weighted graph or hypergraph update structures, which includes the Kipnis–Marchioro–Presutti (KMP) model and the energies of Kac's walk on the sphere. These are conservative continuous-spin systems whose reversible measures are Dirichlet distributions. We prove that their spectral gap is always attained by a polynomial of degree at most two in the energy variables. Equivalently, through an intertwining with the associated discrete particle systems, the dominant mode is always represented by either a one-particle or a two-particle observable. We interpret this as a manifestation of an Aldous-type spectral gap phenomenon. In particular, the resulting two-particle spectral gap identity settles a recent conjecture of Alon and Puder.
	
	We also characterize sharply when one particle suffices and when a genuinely two-particle mode dominates. Under a common rescaling of the Dirichlet parameters, segment-like geometries are precisely those for which the gap is always of one-particle type, mean-field geometries are precisely those for which it is always of two-particle type, and all other geometries exhibit a nontrivial transition.
	
	The proof relies on the analysis of the so-called hidden model, a dual representation of the original process. The mechanism underlying the degree-two reduction is remarkably robust and extends to a much broader class of stochastic exchange models and their associated particle systems, including the harmonic process, the immediate exchange model, averaging-type processes, and nonreversible variants.
	
	Finally, we analyze boundary-driven versions of the KMP model, in which the bulk exchange dynamics interacts with reservoirs, generally resulting in nonreversible processes. In contrast to the conservative setting, we prove that the spectral gap of a boundary-driven model is always of one-particle type.
\end{abstract}
\tableofcontents	
\section{Introduction, models and main results}\label{sec:intro}
Stochastic exchange models are a broad family of Markov processes in which a continuous quantity, such as energy, wealth, or mass, is randomly redistributed among agents interacting through the edges of a graph or hypergraph. These models play a prominent role across a variety of fields, ranging from statistical physics to economics and the social sciences. Important examples include the Kipnis--Marchioro--Presutti (KMP) model \cite{kipnis_heat_1982}, Kac's walk on the sphere \cite{kac_foundations_1956,diaconis_saloff-coste_bounds_2000,carlen_carvalho_loss_determination_2003,caputo_mixing_2019}, the averaging process \cite{aldous_lecture_2012}, the harmonic process \cite{frassek2021exact}, and the immediate exchange model \cite{heinsalu_patriarca_kinetic_2014}.

A fundamental feature of these dynamics is their intertwining relation with discrete particle systems, which allows one to describe the action of the generator on polynomial observables in terms of interacting random walks on the underlying (hyper)graph. In this work we prove that the spectral gap of these models is always attained by a polynomial observable of degree at most two in the energy variables. In terms of the associated particle systems, this means that the dominant relaxation mode is always described by either a one-particle or a two-particle observable, a property that we refer to as an Aldous-type phenomenon.

The name originates from a conjecture of David Aldous concerning a different dynamics, namely the interchange process on graphs, asserting that the spectral gap is always realized by a one-particle eigenfunction. This conjecture was later proved for arbitrary weighted graphs in \cite{caputo_proof_2010}. An analogous statement was then conjectured for interchange processes on arbitrary weighted hypergraphs; see, for instance, \cite[Conjecture 1.7]{bristiel_caputo_entropy_2021}. In that generality, the conjecture remains open; see \cite{alon_kozma_puder_aldous_2025,bristiel_caputo_entropy_2021} for some partial progress. See also \cite{cesi_remarks_2016,alon2026aldous,levhari2026aldous} and references therein  for some further related results.

In the setting of stochastic exchange models, however, it is known that one-particle eigenfunctions do not suffice in general, and that genuinely two-particle observables are sometimes required; see, e.g., \cite[Remark 4]{caputo_mixing_2019}. It remained, however, to understand 
whether two particles are always sufficient to characterize the spectral gap. Related phenomena and questions have recently been investigated for a broad class of stochastic exchange models in \cite{kim_quattropani_sau_spectral_2025}, for the symmetric inclusion process and the related diffusion \cite{kim_sau_spectral_2023,kim_sau_one_2024}, and for random walks on the unitary group generated by weighted hypergraphs in \cite{alon2026aldous}. These works provide substantial progress and proved the validity of Aldous-type phenomena in several special cases, but the question of whether two particles always suffice remains unresolved.

 The present work settles this issue in the affirmative for stochastic exchange models, answering several questions raised in \cite{kim_quattropani_sau_spectral_2025}. It also proves a conjecture of Alon and Puder stated in \cite[Conjecture 1.15]{alon2026aldous}, yet, not addressing the more general statement  concerning random walks on the unitary group formulated in \cite[Conjecture 1.14]{alon2026aldous}.

\subsection{KMP models: Aldous' spectral gap phenomenon}\label{sec:KMP-intro}
Let  $n\ge 2$, $[n]\eqdef \{1,\ldots, n\}$, and
consider the simplex	
\begin{equation}\textstyle
		\Omega\eqdef \tset{\eta \in [0,1]^n:\sum_{x\in [n]}\eta_x=1}\fstop
\end{equation}
In the original KMP model,  $(\eta_x)_{x\in [n]}$ represent the energies in a  one-dimensional chain of oscillators, and at rate one independently each pair $(\eta_x,\eta_{x+1})$, $x=1,\dots,n-1$, undergoes exchange updates 
\begin{equation}\label{eq:2_up}
(\eta_x,\eta_{x+1}) \longmapsto (U(\eta_x+\eta_{x+1}), (1-U)(\eta_x+\eta_{x+1}))\comma
\end{equation}
where $U$ is a uniform random variable in $[0,1]$. In the absence of external reservoirs, the total energy $\sum_{x=1}^n \eta_x=1$ is conserved---hence, we sometimes speak of the \textit{conservative} ${\rm KMP}$ model---and the equilibrium measure is simply the uniform distribution over $\Omega$. Convergence to stationarity for this model was extensively analyzed in \cite{caputo_mixing_2019}, which, among other things, proved that the spectral gap is given by   
$\lambda=1 - \cos(\pi/n)$ with the linear eigenfunction
\begin{equation}\label{eq:gap1_seg}
\eta\longmapsto \sum_{y=1}^{n-1}\sin(\pi y/n) \sum_{x=1}^y\left(\eta_x -\tfrac1n\right)\,.
\end{equation}
If one replaces the graph segment $\{1,\dots,n\}$ by the complete graph over $n$ vertices, computations as in \cite{carlen_carvalho_loss_determination_2003,caputo_kac2008} show that the spectral gap is now given by $\lambda=(n+1)/3$  with the quadratic eigenfunction
\begin{equation}\label{eq:gap3_cg}
\eta\longmapsto \sum_{x=1}^{n}\left(\eta_x^2-  \tfrac2{n(n+1)}\right) \fstop
\end{equation}

In this work, we consider a generalized version of the 
KMP model where the underlying graph is a completely arbitrary weighted hypergraph, and the redistribution variable $U$ in \eqref{eq:2_up} is replaced by a general inhomogeneous Dirichlet distribution.
While there are no closed formulas for the spectral gap and the associated eigenfunction in this generality, we shall prove that, for the conservative KMP model, the dominant mode is always a polynomial of degree at most two. In contrast, we show that polynomials of degree one always suffice for its non-conservative variant; see Section \ref{sec:reservoirs-intro}.

\subsubsection{Setting and first main result}
Fix positive site weights $\alpha=(\alpha_x)_{x\in [n]}$, and nonnegative block weights $w=(w_B)_{B\subseteq [n]}$. The  \textit{$(\alpha,w)$-${\rm KMP}$ model} is the process where each block $B\subseteq[n]$ is updated independently with rate $w_B$, and when an update occurs at $B$, the variables $\eta^B=(\eta_x)_{x\in B}$ undergo the transition  
\begin{equation}\label{eq:B_up}
	\textstyle
\eta^B \longmapsto  \eta(B) U\comma\qquad
		U=(U_x)_{x\in B}\sim {\rm Dir}(\alpha^B)\comma \eta(B)\eqdef\sum_{x\in B}\eta_x\fstop
	\end{equation}
Here, ${\rm Dir}(\alpha^B)$, with $\alpha^B=(\alpha_x)_{x\in B}$, denotes the Dirichlet distribution on the probability simplex over $B$, with density proportional to $\prod_{x\in B}u_x^{\alpha_x-1}$, or equivalently the law of $(V_x/V(B))_{x\in B}$, where the $V_x$ are independent ${\rm Gamma}(\alpha_x,1)$ random variables, with density proportional to $v^{\alpha_x-1}e^{-v}$ on $(0,\infty)$, and $V(B)\eqdef\sum_{x\in B}V_x$. The Dirichlet distribution is the multivariate analogue of the Beta distribution; in particular, if $B=\{x,y\}$, then $U_x\sim{\rm Beta}(\alpha_x,\alpha_y)$ and $U_y=1-U_x\sim {\rm Beta}(\alpha_y,\alpha_x)$.

	Note that when $B=\{x,x+1\}$ and $\alpha_x=\alpha_{x+1}= 1$, \eqref{eq:B_up} coincides with the exchange update in \eqref{eq:2_up}. 
 Note also that if $\mu= {\rm Dir}(\alpha)$ denotes the Dirichlet distribution on $\Omega$ with parameters $\alpha$, then the update at $B$ replaces the energy variables $\eta^B=(\eta_x)_{x\in B}$ by a sample from their conditional distribution under $\mu$ given the variables $\eta^{B^\complement}=(\eta_{y})_{y\notin B}$.  
 In other words, 
$(\alpha,w)$-${\rm KMP}$ is a block heat-bath dynamics with respect to $\mu$ and, more precisely, the Markov process with state space $\Omega\subseteq \R^n$ with infinitesimal generator
	\begin{equation}\label{eq:gen-KMP}
		\cL f(\eta)=\sum_{B\subseteq [n]} w_B\left(\mu_B f(\eta)-f(\eta)\right)\comma\qquad \eta\in \Omega\comma
	\end{equation}
	where $ f$ is any bounded measurable function on $\Omega$, and 
	$\mu_B f(\eta)\eqdef \mu(f\mid \eta^{B^\complement})$
is the conditional expectation with respect to $\mu$ given $\eta^{B^\complement}$. Being a finite linear combination of orthogonal projections, $\cL$ is a bounded self-adjoint operator on the (real) Hilbert space $L^2(\mu)=L^2(\Omega,\mu)$. The spectral gap is the second largest eigenvalue of $-\cL$, which is given by 
\begin{equation}\label{eq:gapac}
\gap(\cL)=\inf_{\substack{f\in L^2(\mu)\\
		f\neq 0, \,\mu(f)=0}}\frac{\ttscalar{f}{-\cL f}_\mu}{\ttscalar{f}{f}_\mu}\fstop
	\end{equation}
	where $\ttscalar{f}{g}_\mu=\mu(fg)$ denotes the scalar product in  $L^2(\mu)$. Equivalently, $\gap(\cL)$ is the largest constant $\l\ge 0$ such that, for every $f\in L^2(\mu)$,
	\begin{equation}\label{eq:gap_decay}
\|e^{t\cL}f - \mu(f)\|_\mu\le e^{-\l t}\|f - \mu(f)\|_\mu\,,\qquad t\ge 0\fstop
	\end{equation}
	where $(e^{t\cL})_{t\ge 0}$ denotes the semigroup generated by $\cL$, and $\|\emparg \|_\mu$ is the norm in $L^2(\mu)$. 
	It is possible to verify---and also follows from Theorem \ref{th:gap-1-2-quantitative} or Remark \ref{rem:positivity-gap} below---that
	\begin{equation}\label{eq:equiv-gap>0}
		\gap(\mathcal L)>0
		\qquad\Longleftrightarrow\qquad
		w\ \text{induces a connected hypergraph}\fstop
	\end{equation}
Here, we refer to the standard notion of hypergraph connectedness; see \eqref{eq:connected-hypergraph}.

Let $\mathscr P_k$  denote the set of all polynomials of degree at most $k\ge1$  in the variables $\eta\in {\Omega}$, that is, the linear span of all monomials $\eta_{x_1}\cdots\eta_{x_j}$, with $x_i\in[n]$ and $j\le k$, while $\mathscr P_0$ denotes the space of constant functions. A simple computation shows that
\begin{equation}\label{eq:poly-invariance}
\cL \mathscr P_k\subseteq \mathscr P_k\comma\qquad k\ge 0\comma
\end{equation} that is, $\mathscr P_k$ is left invariant by $\cL$. As each $\mathscr P_k$ is a finite dimensional vector space, the invariance in \eqref{eq:poly-invariance},  the density of polynomials in $L^2(\mu)$, and the self-adjointness of $\cL$ guarantee that
the spectrum of $\cL$  is pure point, and that all eigenfunctions are polynomials. Moreover, 
 \begin{equation}\label{eq:gap-inf}
		\gap(\cL) = \inf_{k\ge 1}\gap(\cL|_{\mathscr P_k})\comma
	\end{equation}
where $ \gap(\cL|_{\mathscr P_k})$ is obtained as in \eqref{eq:gapac}
with the restriction that $f\in\mathscr P_k$. 
Our first main result states that one may actually restrict to $k=2$. 
\begin{theorem}[KMP]\label{th:main-KMP}
	For all positive site weights $\alpha$ and all nonnegative block weights $w$, the spectral gap of the $(\alpha,w)$-${\rm KMP}$ model satisfies
	\begin{equation}
		\gap(\cL)=\gap(\cL|_{\mathscr P_2})\fstop
	\end{equation}
\end{theorem}

\subsubsection{Particle systems and second main result}\label{sec:KMP-particle-intro}
As a starting point, we observe that the invariance property \eqref{eq:poly-invariance} admits a sharper formulation: the action of $\cL$ on polynomials transfers to a particle system's generator acting on their coefficients. We now recall some basic facts, postponing all the details to Section \ref{sec:prelim-particle-systems}.

Fix $k\ge1$. Since $\sum_{x\in[n]}\eta_x=1$ on $\Omega$, every element of $\mathscr P_k$ can be represented in homogeneous degree $k$ form. Thus, any polynomial in $\mathscr P_k$ takes the form
\begin{equation}\label{eq:intro-widehat-psi}
	\widehat\psi(\eta)
	\eqdef
	\sum_{x_1,\ldots,x_k\in[n]}
	\psi(x_1,\ldots,x_k)\,
	\eta_{x_1}\cdots \eta_{x_k}
	\comma
	\qquad \eta\in\Omega\comma
\end{equation}
for some $\psi :[n]^k\to \R$.
Then, as shown in Proposition \ref{pr:particle-system-intertwining}, for every $\psi\in\R^{[n]^k}$,
\begin{equation}\label{eq:intro-intertwining}
	\cL\widehat\psi
	=
	\widehat{L_k\psi}\comma
\end{equation}
where $L_k$ denotes the generator of a labeled $k$-particle system on $[n]^k$.
The dynamics is as follows. Start with $k$ labeled particles at positions $x_1,\ldots,x_k\in[n]$. Whenever a block $B$ is selected, at rate $w_B$, sample
$U\sim{\rm Dir}(\alpha^B)$. 
Each particle currently sitting in $B$ is then resampled inside $B$, conditionally independently given $U$, according to
\begin{equation}
	x_i\in B
	\longmapsto
	y\in B
	\quad\text{with probability }U_y\comma
\end{equation}
whereas particles outside $B$ are kept fixed. 
After averaging over the block updates and over the Dirichlet vectors, one obtains the Markov generator $L_k$; see \eqref{eq:gen-particle} below for the explicit expression. We call the resulting process on $[n]^k$ the ${\rm KMP}$ labeled-particle system or, more precisely, the discrete $(k,\alpha,w)$-KMP model.

When $k=1$, the dynamics reduces to a random walk. Indeed, writing
\begin{equation}\label{eq:RW-rates}
	\alpha(B)
	\eqdef
	\sum_{z\in B}\alpha_z
	\comma B\subseteq[n]\comma
	\qquad
	r_{xy}(\alpha,w)
	=
	\alpha_y
	\sum_{B\ni x,y}
	\frac{w_B}{\alpha(B)}
	\comma
	 x\neq y\comma
\end{equation}
the identity
$
	\mathbb E_{U\sim{\rm Dir}(\alpha^B)}[U_y]
	=
	\frac{\alpha_y}{\alpha(B)}
	$, 
	$ y\in B$,
shows that the single particle jumps on $[n]$ with rates given in \eqref{eq:RW-rates}. For $k\ge2$, the particles genuinely interact: several may jump during the same block update, and their transition probabilities are sampled from the same random vector $U$.

Just as $\cL$ in \eqref{eq:gen-KMP} is a block heat-bath generator with respect to $\mu={\rm Dir}(\alpha)$, the $k$-particle generator $L_k$ admits an analogous interpretation. To identify the corresponding reference measure, let $\eta\sim\mu$ and, conditionally on $\eta$, let $X_1,\ldots,X_k$ be independent with common law $\eta$ on $[n]$. Their joint law $\mu_k$ on $[n]^k$ is given, for $x_1,\ldots,x_k\in[n]$, by
\begin{equation}\label{eq:mu-k}
	\mu_k(x_1,\ldots,x_k)
	\eqdef
	\mu(\eta_{x_1}\cdots\eta_{x_k})
	=
	\frac{\Gamma(\alpha_0)}{\Gamma(\alpha_0+k)}
	\prod_{x\in[n]}
	\frac{\Gamma(\alpha_x+\mathfrak n_x(x_1,\ldots,x_k))}
	{\Gamma(\alpha_x)}
	\comma
\end{equation}
where $\alpha_0\eqdef\alpha([n])=\sum_{x\in[n]}\alpha_x$ and
$\mathfrak n_x(x_1,\ldots,x_k)\eqdef\sum_{i=1}^k\car_{\set{x_i=x}}$. Here, $\Gamma(a)$, $a>0$, is the gamma function, satisfying $\Gamma(a+1)=a\Gamma(a)$. The last identity in \eqref{eq:mu-k} follows by integrating $\eta_{x_1}\cdots\eta_{x_k}$ against the Dirichlet density and using, for $\beta_x>0$,
\begin{equation}
	\int_\Omega
	\prod_{x\in[n]}\eta_x^{\beta_x-1}\,\dd\eta
	=
	\frac{\prod_{x\in[n]}\Gamma(\beta_x)}
	{\Gamma(\sum_{x\in[n]}\beta_x)}
	\comma
\end{equation}
where $\dd\eta$ denotes the Lebesgue measure on $\Omega$ in the coordinates $(\eta_1,\ldots,\eta_{n-1})\in[0,1]^{n-1}$, with $\eta_n=1-\sum_{x=1}^{n-1}\eta_x$. 
That $L_k$ is indeed the block heat-bath generator with respect to $\mu_k$ follows from \eqref{eq:mu-k} and the corresponding heat-bath property of $\cL$. Namely, conditionally on the particles outside a block $B$, the conditional resampling in \eqref{eq:B_up} of $\eta^B$ under $\mu$ and the identity $\mu_k(x_1,\ldots,x_k)=\mu(\eta_{x_1}\cdots\eta_{x_k})$ imply that the particles lying in $B$ are redistributed independently according to a common random probability vector $U\sim{\rm Dir}(\alpha^B)$. This is exactly the $B$-update defining $L_k$. Hence, $\mu_k$ is reversible for $L_k$, and $L_k$ is self-adjoint on $L^2(\mu_k)=L^2([n]^k,\mu_k)$. Consequently, $L_k$ has all real eigenvalues. In analogy with \eqref{eq:gap-inf}, we define $\gap(L_k)$ as the second smallest eigenvalue of $-L_k$.

As the representation in \eqref{eq:intro-widehat-psi} depends only on the symmetric part of $\psi:[n]^k\to \R$, the intertwining relation in \eqref{eq:intro-intertwining} is fully informative only for symmetric functions $\psi$ in
\begin{equation}\label{eq:coeff-space-sym}
\mathscr H_k
	\eqdef
	\tset{
	\psi
	:
	\psi(x_{\varsigma(1)},\ldots,x_{\varsigma(k)})
	=
	\psi(x_1,\ldots,x_k)\ \text{for all}\  \varsigma\in\mathfrak S_k,  x_1,\ldots, x_k\in [n]
}\comma
\end{equation}
where $\mathfrak S_k$ denotes the symmetric group on $[k]$.
Since $L_k$ clearly commutes with permutations of the labels, this subspace is $L_k$-invariant, and
\begin{equation}\label{eq:L-k-sym}
	L_k^\sym
	\eqdef
	L_k|_{\mathscr H_k}	
\end{equation}
corresponds to the generator of the unlabeled version of the same particle system. Equivalently, $L_k^\sym$ describes the $L_k$-dynamics in which labels are forgotten and only the particle-occupation numbers at each site are retained.

As we shall prove in Proposition \ref{pr:particle-system-intertwining}, the intertwining \eqref{eq:intro-intertwining} shows that $\cL|_{\mathscr P_k}$ and $L_k^\sym$ are isospectral, see \eqref{eq:isospectral-basic}; thus, 
\begin{equation}\label{eq:intro-gap-particle-identification}
	\gap(\cL|_{\mathscr P_k})
	=
	\gap(L_k^\sym)
	\comma
	\qquad k\ge1\fstop
\end{equation}
Hence, Theorem \ref{th:main-KMP} can be equivalently rewritten as the following chain of identities:
\begin{equation}\label{eq:intro-chain}
	\gap(L_k^\sym)
	=
	\gap(\cL|_{\mathscr P_k})
	=
	\gap(\cL|_{\mathscr P_2})
	=
	\gap(L_2^\sym)
	\comma
	\qquad k\ge2\fstop
\end{equation}
 In other words, for arbitrary positive site weights $\alpha$ and nonnegative block weights $w$, the unlabeled particle systems satisfy a two-particle spectral gap identity. In the special case $\alpha\equiv1$, this is precisely the particle-system identity predicted by \cite[Conjecture 1.15]{alon2026aldous}; hence, Theorem \ref{th:main-KMP} settles that conjecture.

We now turn to the labeled dynamics. The operator $L_k^\sym$ is obtained by restricting $L_k$ to the subspace of symmetric observables; this restriction is proper exactly when $k\ge2$. Consequently, the spectrum of $L_k^\sym$ is contained in that of $L_k$, and, in particular,
\begin{equation}\label{eq:intro-labeled-unlabeled-ineq}
	\gap(L_k)
	\le
	\gap(L_k^\sym)
	\comma
	\qquad k\ge1\fstop
\end{equation}
Our second main result shows that this inequality always saturates to an identity.

\begin{theorem}[Discrete KMP]\label{th:disc-ell-KMP}
	For all positive site weights $\alpha$ and all nonnegative block weights $w$, the spectral gap of the discrete $(k,\alpha,w)$-${\rm KMP}$ model satisfies
	\begin{equation}
		\gap(L_k)
		=
		\gap(L_k^\sym)
		\comma
		\qquad k\ge1\fstop
	\end{equation}
	Together with \eqref{eq:intro-chain}, this yields
	\begin{equation}\label{eq:intro-labeled-chain}
		\gap(L_k)
		=
		\gap(L_2^\sym)
		\comma
		\qquad k\ge2\fstop
	\end{equation}
\end{theorem}
\begin{remark}
 The identity in \eqref{eq:intro-labeled-chain} in Theorem \ref{th:disc-ell-KMP} is a strengthening of Theorem \ref{th:main-KMP}: the identity in \eqref{eq:intro-chain} identifies the gap over symmetric observables as being governed by the two-particle system,	 while \eqref{eq:intro-labeled-chain} further proves that non-symmetric observables contribute no smaller eigenvalues.
 \end{remark}

The spectral gap identity in \eqref{eq:intro-labeled-chain} shows that at most two particles always suffice to determine the ${\rm KMP}$ spectral gap. As we shall see, both this phenomenon and the mechanism behind its proof are remarkably robust, extending to a broad class of stochastic exchange models and their associated particle systems; see Sections \ref{sec:SEM-intro} and \ref{sec:SEM}. 

For ${\rm KMP}$ models, however, these ideas can be sharpened considerably: according to the weights $\alpha$ and the geometry encoded by $w$, they determine whether the gap is already attained by a degree-one polynomial or instead requires a genuinely degree-two polynomial, thereby revealing a sharp dichotomy between one-particle and two-particle domination. We develop this characterization in the next subsection.

\subsection{KMP models: one- \emph{vs.}\ two-particle gaps}\label{sec:comparison-intro}
 Given the weights
$\alpha=(\alpha_x)_{x\in[n]}$, 
$w=(w_B)_{B\subseteq[n]}$, and the corresponding $(\alpha,w)$-${\rm KMP}$
generator $\cL$, here we write (cf.\ \eqref{eq:intro-gap-particle-identification})
\begin{equation}\label{eq:def-gap-k-alpha-w}
	\gap_k(\alpha,w)
	\eqdef
	\gap(\cL|_{\mathscr P_k})
	\comma
	\qquad k\ge1\comma
\end{equation} 
Note that, since $\mathscr P_k\subseteq \mathscr P_{k+1}$,  one has
$\gap_k(\alpha,w)\ge \gap_{k+1}(\alpha,w)$.
To isolate the source of possible strictness in this inequality, we separate the pure degree-$k$  component in $L^2(\mu)$ from the lower-degree one, and set
\begin{equation}\label{eq:P-k-star}
	\mathscr P_{k,\ast}
	\eqdef
	\mathscr P_k\cap\mathscr P_{k-1}^{\perp_\mu}
	\comma
	\qquad k\ge1\fstop
\end{equation}
Self-adjointness of $\cL$, together with the invariance of each $\mathscr P_k$, implies $\cL\mathscr P_{k,\ast}\subseteq\mathscr P_{k,\ast}$. Consequently, \eqref{eq:gap-inf} refines to
\begin{equation}\label{eq:gap-inf-star-intro}
	\gap(\cL)
	=
	\inf_{k\ge1}\gap_{k,\ast}(\alpha,w)\comma
\end{equation}
where 	$\gap_{k,\ast}(\alpha,w)
\eqdef
\gap(\cL|_{\mathscr P_{k,\ast}})$ is defined as in \eqref{eq:gapac} with $f\in \mathscr P_{k,\ast}$.
Since $\gap_{1,\ast}(\alpha,w)=\gap_1(\alpha,w)$ and \begin{equation}\gap_k(\alpha,w)=\min\{\gap_{k-1}(\alpha,w),\gap_{k,\ast}(\alpha,w)\}\comma\qquad k\ge 2\comma
\end{equation}  Theorem
\ref{th:main-KMP} reduces the infimum in \eqref{eq:gap-inf-star-intro} to
\begin{equation}\label{eq:intro-gap-star}
	\gap(\cL)
	=
	\min\{
	\gap_1(\alpha,w),
	\gap_{2,\ast}(\alpha,w)
	\}\fstop
\end{equation}
Thus, the gap is always either of one-particle type or of genuinely
two-particle type. 

\subsubsection{Quantitative comparison}
Our first result gives a quantitative comparison between the two quantities on the right-hand side of \eqref{eq:intro-gap-star}.

\begin{theorem}\label{th:gap-1-2-quantitative}
For all positive site weights $\alpha$ and all nonnegative block weights $w$,
	\begin{equation}\label{eq:gap-1-2-quantitative}
		\frac{\alpha_{w,{\rm min}}}{1+\alpha_{w,{\rm min}}}
		\left(1+\frac1{\alpha_0}\right)
		\gap_1(\alpha,w)
		\le
		\gap_{2,\ast}(\alpha,w)
		\le
		2\gap_1(\alpha,w)
		\comma
	\end{equation}
	where $\alpha_{w,{\rm min}}\eqdef \min_{B\,:\, w_B>0}\alpha(B)$. Further, if $\gap_1(\alpha,w)>0$, the second inequality is strict.
\end{theorem}

As shown in Remarks \ref{rem:gap-lower-sharp} and \ref{rem:lambda-infty-sharp} below, both bounds are sharp: suitable choices of $\alpha$ and $w$ attain either one. Moreover, both multiplicative factors in \eqref{eq:gap-1-2-quantitative} are strictly positive. Since $\gap_1(\alpha,w)>0$ if and only if $w$ induces a connected hypergraph, \eqref{eq:intro-gap-star} and \eqref{eq:gap-1-2-quantitative} give a quantitative version of \eqref{eq:equiv-gap>0}.

Since the prefactor in the lower bound is smaller than $1$, whereas that in the upper bound equals $2$, Theorem \ref{th:gap-1-2-quantitative} does not determine whether, for any given $\alpha$ and $w$,
\begin{equation}\label{eq:alternatives}
	\gap_1(\alpha,w)\le\gap_{2,\ast}(\alpha,w)
	\qquad\text{or}\qquad
	\gap_1(\alpha,w)>\gap_{2,\ast}(\alpha,w)\fstop
\end{equation}
Recall that both alternatives occur already in the segment and complete-graph examples from \eqref{eq:gap1_seg} and \eqref{eq:gap3_cg}, corresponding, respectively, to
\begin{equation}\alpha\equiv 1\comma w_B=\textstyle{\sum_{x=1}^{n-1}}\,\car_{B=\{x,x+1\}}\qquad\text{and}\qquad \alpha\equiv 1\comma w_B=\car_{|B|=2}\fstop
\end{equation}
Indeed, the segment realizes the first alternative in \eqref{eq:alternatives}, whereas the complete graph realizes the second. In what follows, we classify the possible relations between $\gap_1(\alpha,w)$ and $\gap_{2,\ast}(\alpha,w)$ for general weighted hypergraphs. Remarkably, the segment and the complete graph---and their weighted-hypergraph analogues---emerge as the only two distinct geometries for which one of these quantities uniformly dominates the other.

\subsubsection{A phase transition}
First, observe that a full-block update only produces a global shift by $w_{[n]}\ge 0$ in the spectrum of  $-\cL$ associated to nonconstant eigenfunctions; in particular, 
\begin{equation}\label{eq:global-shift-gap-k-ast}
	\gap_{k,\ast}(\alpha,w)= \gap_{k,\ast}(\alpha,w^{\scriptscriptstyle \circ})+w_{[n]}\comma\quad\text{with}\  w_B^{\scriptscriptstyle \circ}=\car_{B\neq [n]}\,w_B\comma\qquad k\ge 1\fstop
\end{equation} Therefore, for the sole purpose of comparing $\gap_1(\alpha,w)$ and $\gap_{2,\ast}(\alpha,w)$, we may assume
\begin{equation}\label{eq:w-n}
	w_{[n]}=0
\end{equation}
with no loss of generality. Next, let $w$ induce a connected hypergraph,   namely,  \begin{equation}\label{eq:connected-hypergraph}
	\forall\,  N\subseteq[n], N\neq \emp, [n]\comma\quad \exists\, B\subseteq[n]:  w_B>0\ \text{and}\ B\cap N\neq \emp\neq  B\cap ([n]\setminus N)\fstop
\end{equation}
It is immediate to verify that \eqref{eq:w-n} and \eqref{eq:connected-hypergraph} together imply $n\ge 3$.

When comparing the one- and two-particle gaps, it will be crucial to vary the overall scale of the site weights while preserving their relative proportions. Thus, for fixed $\alpha$ and $w$, we consider the one-parameter family $\tau\alpha$, $\tau>0$. The key point is that this scaling affects the two gaps in fundamentally different ways. Under \eqref{eq:w-n}--\eqref{eq:connected-hypergraph}, we shall prove that
\begin{equation}\label{eq:gap-2-monotonicity}
	(0,\infty)\ni\tau\longmapsto \gap_{2,\ast}(\tau\alpha,w)
\end{equation}
is strictly increasing for every $\alpha=(\alpha_x)_{x\in[n]}$.  By contrast, the one-particle dynamics is invariant under this rescaling: since the jump rates \eqref{eq:RW-rates} depend on $\alpha$ only through ratios, the map $\tau\longmapsto\gap_1(\tau\alpha,w)$ is constant; see Remark~\ref{rem:monotonicity-gap-1}.
The directional monotonicity of \eqref{eq:gap-2-monotonicity}, established in Proposition~\ref{pr:monotonicity}, is one of the crucial inputs of our argument, and motivates us to analyze the
limits as $\tau\to 0 $ and $\infty$, for which we show in \eqref{eq:gap-lower-infty} and \eqref{eq:first-limit} that	
\begin{equation}\label{eq:limits}
	\lim_{\tau\to 0}\gap_{2,\ast}(\tau\alpha,w)\le \gap_1(\alpha,w)\le \lim_{\tau\to \infty}\gap_{2,\ast}(\tau\alpha,w)\fstop
\end{equation}
Altogether, this guarantees the existence of a sharp phase transition separating those instances in which either the one-particle or two-particle gap dominates. 
\begin{theorem}[Phase transition]\label{th:phase-transition}
	For all positive site weights $\alpha$ and all nonnegative block weights $w$ satisfying \eqref{eq:w-n}--\eqref{eq:connected-hypergraph}, there exists a unique $\tau_c(\alpha,w)\in [0,\infty]$ such that
	\begin{align}
		\gap_1(\alpha,w)=\gap_1(\tau\alpha,w)<\gap_{2,\ast}(\tau\alpha,w)\comma \qquad &\text{if}\ \tau>\tau_c(\alpha,w)\comma
		\label{eq:regime-high}\\
		\label{eq:regime-low}
		\gap_1(\alpha,w)=\gap_1(\tau\alpha,w)>\gap_{2,\ast}(\tau\alpha,w)\comma\qquad &\text{if}\  \tau<\tau_c(\alpha,w)\fstop
	\end{align}
\end{theorem}

The statement above leaves open whether either extremal value $\tau_c(\alpha,w)\in\{0,\infty\}$ can occur; we now give a complete geometric characterization of all three possible regimes.

\subsubsection{Characterization of the extremal geometries}
The case $\tau_c(\alpha,w)=0$ is the regime in which the KMP spectral gap is always determined by $\gap_1(\tau\alpha,w)$, whereas $\tau_c(\alpha,w)=\infty$ corresponds to the opposite regime, in which it is always determined by $\gap_{2,\ast}(\tau\alpha,w)$. We shall show that these two extremal cases correspond, respectively, to segment-like and mean-field geometries, while all remaining weighted hypergraphs exhibit a genuine phase transition, with $\tau_c(\alpha,w)\in(0,\infty)$. This classification applies to general weighted hypergraphs and goes substantially beyond the previously known cases mentioned above, as well as the extension conjectured in \cite[Remark 3]{caputo_mixing_2019}.

For clarity, we first state the characterization for weighted graphs (Theorem \ref{th:dichotomy-graph}), where the relevant geometries take a particularly transparent form, and then turn to the general hypergraph setting (Theorem \ref{th:dichotomy-hypergraph}).

\subsubsection*{Graph case.} 
Here, we focus on graphs,  namely, on weights $w$	 supported on size-two blocks:
\begin{equation}\label{eq:graph-case}
	w_B=0\comma \qquad \text{whenever}\ |B|\neq2\fstop
\end{equation}
For connected graphs (i.e., $w$ satisfying \eqref{eq:connected-hypergraph}, \eqref{eq:graph-case}), \eqref{eq:w-n} is equivalent to requiring $n\ge 3$.
\begin{theorem}[Graph-version of Theorem \ref{th:dichotomy-hypergraph}]\label{th:dichotomy-graph} For all positive site weights $\alpha$ and for all nonnegative block weights $w$ satisfying \eqref{eq:w-n}--\eqref{eq:connected-hypergraph} and \eqref{eq:graph-case}, one has:
	\begin{enumerate}[(a)] 
		\item \label{it:graph-dichotomy-segment} $\tau_c(\alpha,w)=0$ $\Longleftrightarrow$ $w$ induces a segment, i.e.,  up to a relabeling of $[n]$, 
		\begin{equation}\label{eq:segment-condition-graph}
			w_B>0\qquad \text{if and only if}\qquad B=\{x,x+1\}\comma x =1,\ldots, n-1\fstop
		\end{equation}
		\item \label{it:graph-dichotomy-mean-field} $\tau_c(\alpha,w)=\infty$ $\Longleftrightarrow$ the weights $(\alpha,w)$ satisfy the mean-field condition:
		\begin{equation}\label{eq:mean-field-condition}
			\frac{w_{\{x,y\}}}{\alpha_x+\alpha_y}\qquad \text{does not depend on}\ x,y \in [n]\comma x \neq y\fstop
		\end{equation}
	\end{enumerate} 
	Hence, $\tau_c(\alpha,w)\in (0,\infty)$ in all remaining cases.
\end{theorem}
\begin{remark}[Segment \textit{vs.}\ non-segment]\label{rem:segment-non-segment}
	Since the segment condition \eqref{eq:segment-condition-graph} depends only on the graph weights $w$, whereas the mean-field condition \eqref{eq:mean-field-condition} imposes a specific relation between $\alpha$ and $w$, fixing $w$ as in Theorem \ref{th:dichotomy-graph} yields the following dichotomy:
	\begin{enumerate}[(i)]
		\item Either $w$ induces a segment as in \eqref{eq:segment-condition-graph}, and 
		\begin{equation}
			\gap_1(\alpha,w)<\gap_{2,\ast}(\alpha,w)\comma\quad\text{for every}\ \alpha \in (0,\infty)^n\fstop
		\end{equation}
		\item Or $w$ does not induce a segment, and there exist $\alpha^+, \alpha^-\in (0,\infty)^n$ such that
		\begin{equation}
			\gap_1(\alpha^+,w)<\gap_{2,\ast}(\alpha^+,w)\comma\qquad \gap_1(\alpha^-,w)>\gap_{2,\ast}(\alpha^-,w)\fstop
		\end{equation}
	\end{enumerate}
\end{remark}

\subsubsection*{Hypergraph case.}
For general hypergraphs satisfying only \eqref{eq:w-n}--\eqref{eq:connected-hypergraph}, an analogous picture remains valid after passing to a suitable partition, as we explain below; then, 
the extremal regimes $\tau_c(\alpha,w)\in\{0,\infty\}$ are characterized, respectively, by hypergraph analogues of the segment and mean-field graphs.

Given block weights $w=(w_B)_{B\subseteq[n]}$, we say that a partition
\begin{equation}\label{eq:partition}
	[n]=N_1\sqcup\cdots\sqcup N_m\comma
	\qquad 2\le m\le n\comma
\end{equation}
is $w$-\textit{compatible} if every block $B\subseteq[n]$ with $w_B>0$ is either contained in a single atom $N_j$ or is a union of atoms. Thus, sites belonging to the same atom are always updated together whenever the update involves sites outside that atom.

The discrete partition
$
	[n]=\{1\}\sqcup\cdots\sqcup\{n\}
$
is always $w$-compatible. Moreover, compatibility depends only on the support of the block weights, and is preserved when positive-weight blocks are removed:  if $w_B\ge w'_B$ for every $B\subseteq[n]$, then every $w$-compatible partition is also $w'$-compatible. Further, the join of two $w$-compatible partitions is again $w$-compatible, provided that it has at least two atoms. 

If $w$ induces a connected \textit{graph} \eqref{eq:graph-case}, the discrete partition is in fact the only $w$-compatible one. Indeed, any non-discrete partition contains an atom $N\subseteq[n]$ with $2\le |N|<n$, and graph-connectedness yields an edge $B=\{x,y\}$ with $w_B>0$ such that
$
	B\cap N\neq\emp
	\neq
	B\cap([n]\setminus N)$,
which is incompatible with \eqref{eq:partition}. For general hypergraphs, nontrivial compatible partitions may instead occur. For instance, if
\begin{equation}
	w_B
	=
	\car_{B=\{1,\ldots,n'\}}
	+
	\car_{B=\{n',\ldots,n\}}\comma
	\qquad
	1<n'<n\comma
\end{equation}
with $n\ge3$, then the partition
$
	[n]
	=
	\{1,\ldots,n'-1\}
	\sqcup
	\{n'\}
	\sqcup
	\{n'+1,\ldots,n\}
$
is $w$-compatible.

Given a $w$-compatible partition as in \eqref{eq:partition}, we define
$\widetilde w=(\widetilde w_B)_{B\subseteq[n]}$ by discarding the blocks contained in a single atom and retaining those which are unions of at least two atoms:
\begin{equation}\label{eq:weights-tilde}
	\widetilde w_B
	\eqdef
	\begin{dcases}
		w_B
		&\text{if }B=N_{j_1}\sqcup\cdots\sqcup N_{j_\ell}\comma\ell\ge2\comma\\
		0
		&\text{otherwise}\fstop
	\end{dcases}
\end{equation}
Under \eqref{eq:w-n}--\eqref{eq:connected-hypergraph}, the hypergraph induced by $\widetilde w$ is still connected and satisfies $\widetilde w_{[n]}=0$.

Finally, a \textit{coarsest partition} of $w$ is a $w$-compatible partition with the smallest number of atoms.  Under \eqref{eq:w-n}--\eqref{eq:connected-hypergraph}, the coarsest partition is unique and has $m\ge3$. Indeed, by connectedness, a $w$-compatible partition with $m=2$ would force any positive-weight block crossing the two atoms to be $[n]$, contradicting $w_{[n]}=0$; uniqueness then follows by taking the join of two coarsest partitions.

\begin{theorem}[Hypergraph-version of Theorem \ref{th:dichotomy-graph}]\label{th:dichotomy-hypergraph}
	For all positive site weights $\alpha$ and for all nonnegative block weights $w$ satisfying \eqref{eq:w-n}--\eqref{eq:connected-hypergraph}, one has:
	\begin{enumerate}[(a)]
		\item \label{it:hypergraph-dichotomy-segment} $\tau_c(\alpha,w)=0$ $\Longleftrightarrow$ $w$ is segment-like, i.e., by taking its coarsest partition \eqref{eq:partition}--\eqref{eq:weights-tilde} and an appropriate relabeling of~$[m]$, 
		\begin{equation}
			\widetilde w_B>0 \qquad \text{implies}\qquad B=N_i\sqcup N_{i+1}\sqcup \cdots \sqcup N_{j-1}\sqcup N_{j}\comma 1\le i< j\le m\fstop
		\end{equation}
		\item\label{it:hypergraph-dichotomy-mean-field} $\tau_c(\alpha,w)=\infty$ $\Longleftrightarrow$ $(\alpha,w)$ is mean-field, i.e., for its coarsest partition \eqref{eq:partition}--\eqref{eq:weights-tilde},	
		\begin{equation}
			c_{ij}=c_{ij}(\alpha,w)\eqdef \sum_{B\supseteq N_i\sqcup N_j} \frac{\widetilde w_B}{\alpha(B)}\qquad\text{does not depend on}\ i,j\in [m]\comma i\neq j\fstop
		\end{equation}
	\end{enumerate}
	Hence, $\tau_c(\alpha,w)\in (0,\infty)$ in all remaining cases.
\end{theorem}

	\begin{remark}[From hypergraphs to graphs] As already mentioned,		for graph weights $w$ (i.e., satisfying \eqref{eq:graph-case}) on $n\ge 3$ sites, connectedness \eqref{eq:connected-hypergraph} forces the discrete partition $[n]=\{1\}\sqcup \cdots \sqcup \{n\}$ to be the coarsest one. Hence, Theorem \ref{th:dichotomy-graph} indeed  follows from Theorem \ref{th:dichotomy-hypergraph}.
\end{remark}

\subsection{Related models and extensions} We next place the main spectral gap identity for the KMP model, established in Theorem \ref{th:main-KMP}, in a broader context.
 We first discuss analogous Aldous-type spectral gap phenomena for Kac's walk on the sphere and heat-bath dynamics for cone measures, systems whose structure is richer than that of the ${\rm KMP}$ model. We then establish a degree-one spectral gap identity for the non-conservative ${\rm KMP}$ model and, finally, show that our main theorem extends to a broad class of stochastic exchange models.

\subsubsection{Kac's walk and generalizations}
{Kac's walk on the sphere} is the continuous-time Markov process on
\begin{equation}
\mathbb S\eqdef\set{v\in\R^n:\textstyle \sum_{x\in[n]} v_i^2=1}\subseteq \R^n\comma
\end{equation}
defined as follows. Each unordered pair of coordinates $\{x,y\}$ is updated independently at rate $1$. When an update occurs at the pair $\{x,y\}$, the vector $(v_x,v_y)$ is replaced by a new pair $(v_x',v'_y)$, chosen uniformly on the circle
\begin{equation}
\{(v_x',v_y')\in\R^2:(v_x')^2+(v_y')^2=v_x^2+v_y^2\}\fstop
\end{equation}
The process is reversible with respect to the uniform distribution on $\mathbb S$. 
Introduced by Kac in \cite{kac_foundations_1956} as a mean-field toy model for energy-preserving collisions in kinetic theory, this process has since been studied extensively.
It was shown in \cite{carlen_carvalho_loss_determination_2003} that the spectral gap equals $\l=(n+2)/4$ with eigenfunction 
\begin{equation}
	\label{eq:eigKac}
	v\longmapsto \sum_{i=1}^n\left(v_i^4 - \tfrac{3}{n(n+2)}\right)\,,
\end{equation}
a polynomial of degree 4 in the velocity variables $v=(v_x)_{x\in [n]}$.
We also refer to \cite{caputo_kac2008} for an alternative computation of the spectral gap, and to \cite{maslen_eigenvalues_2003} for a characterization of all eigenvalues of this process. 

In this work, we consider a generalization of Kac's walk in two directions. First, the complete-graph geometry is replaced by a weighted hypergraph, allowing updates on arbitrary blocks and with arbitrary rate. Second, the Euclidean sphere $\mathbb S\subseteq \R^n$ is replaced by the inhomogeneous $\ell_p^n$-type sphere
\begin{equation}
	\mathbb S_p
	\eqdef
	\set{
		\zeta\in\R^n:
		\textstyle \sum_{x\in[n]}|\zeta_x|^{p_x}=1}
		\subseteq \R^n\comma
\end{equation}
where $p=(p_x)_{x\in[n]}$ is an arbitrary vector with positive entries. 

We equip $\mathbb S_p$ with the probability distribution $\kappa$ given as follows. Let $Z=(Z_x)_{x\in [n]}$ be independent real random variables with densities proportional to $z\mapsto \exp(-|z|^{p_x})$, and define
\begin{equation}
	\Phi_p(Z)\eqdef\sum_{y\in [n]}|Z_y|^{p_y}\comma\qquad \zeta_x\eqdef \frac{Z_x}{\Phi_p(Z)^{1/p_x}}\comma x \in [n]\fstop
\end{equation}
Therefore, $\zeta\in \mathbb S_p$ by construction, and its law $\kappa$ is referred to as
the \textit{cone measure} on $\mathbb S_p$; see, e.g., \cite{caputo_salez_entropy_2024} or \cite{barthe_probabilistic_2005} for a discussion  of the homogeneous cases $p\equiv {\rm const.}>0$. We note that when $p\equiv 2$ one recovers the uniform distribution over $\mathbb S$.

For nonnegative block weights $w=(w_B)_{B\subseteq[n]}$, define the $(p,w)$-Gibbs sampler, or heat-bath dynamics, for the cone measure,  as the Markov 
process on $\mathbb S_p$ with infinitesimal generator
\begin{equation}\label{eq:gen-cone}
	\cG h(\zeta)
	\eqdef
	\sum_{B\subseteq[n]} w_B
	\left(
	\kappa_{B}h(\zeta)-h(\zeta)
	\right) \comma
	\qquad
	\kappa_{B}h(\zeta)\eqdef \kappa(h\mid \zeta^{B^c})\fstop
\end{equation}
with $\zeta\in \mathbb S_p$ and $h:\mathbb S_p\to \R$ bounded and measurable.
That is, with rate $w_B$, the block variables $\zeta^B=(\zeta_x)_{x\in B}$ are resampled according to the conditional distribution $\kappa$ given the variables $\zeta^{B^c}=(\zeta_x)_{x\notin B}$. The generator $\cG$ is, thus, a bounded self-adjoint operator in $L^2(\kappa)=L^2(\mathbb S_p,\kappa)$. Observe that the case $p\equiv 2$ and $w_B=\car_{|B|=2}$ corresponds to Kac's walk on the sphere. 

For general $p=(p_x)_{x\in [n]}$ and $w=(w_B)_{B\subseteq[n]}$, this dynamics admits a natural projection, under which it reduces to a KMP dynamics. More precisely, define, for all $x\in [n]$,
\begin{equation}
	\eta_x\eqdef |\zeta_x|^{p_x}\comma\qquad 
	\sigma_x\eqdef {\rm sign}(\zeta_x)\fstop
\end{equation}
In this context, we refer to $\eta=(\eta_x)_{x\in [n]}$ and $\sigma=(\sigma_x)_{x\in [n]}$ as the energy and spin/angle variables of $\zeta$, respectively.
On the one hand, when  $\zeta\sim\kappa$, the corresponding energy variables $\eta$ belong to $\Omega$ and satisfy
\begin{equation}\label{eq:dirp}
	\eta=(\eta_x)_{x\in[n]}\sim {\rm Dir}(\alpha)\comma\qquad
	\sigma=(\sigma_x)_{x\in[n]}\ \text{i.i.d.\ uniform on }\{\pm1\}\comma\qquad
	\eta\perp\!\!\!\perp\sigma\,,
\end{equation}
where $\alpha_x=1/p_x$, $x\in [n]$, and where $\eta\perp\!\!\!\perp\sigma$ indicates that $\eta$ and $\sigma$ are independent.
Indeed, the random variables $|Z_x|^{p_x}$ are independent ${\rm Gamma}(\alpha_x,1)$; thus, their normalization by $\Phi_p(Z)$ gives the law
${\rm Dir}(\alpha)$.
On the other hand, under the map $\zeta\mapsto\eta$, the $(p,w)$-Gibbs sampler in \eqref{eq:gen-cone} projects exactly onto the ${\rm KMP}$ model in \eqref{eq:B_up} with parameters $(\alpha,w)$. Consequently, for this choice of parameters,
$\spec(\mathcal L)\subseteq \spec(\mathcal G)$.
Thus, the spectral gap  of the Gibbs sampler, referred to as $\gap(\cG)$,  is at most the spectral gap of $\cL$, namely $\gap(\cL)=\gap(\cL|_{\mathscr P_2})$ as in Theorem \ref{th:main-KMP}. We shall establish that they coincide.   
\begin{theorem}[Gibbs sampler for cone measure]\label{th:GS-p}
	For all positive site weights $p$ and all nonnegative block weights $w$, the spectral gap of the $(p,w)$-Gibbs sampler on $\mathbb S_p$ satisfies
	\begin{equation}
		\gap(\cG)=\gap(\cL|_{\mathscr P_2})\comma
	\end{equation}
	where $\cL$ is the generator of the $(\alpha,w)$-${\rm KMP}$ model with $\alpha_x=1/p_x$, $x\in [n]$. Moreover,  it is attained by a polynomial eigenfunction of degree at most $2$ in the energy variables $\eta_x=|\zeta_x|^{p_x}$.
\end{theorem}
Combined with Theorem \ref{th:dichotomy-graph}\ref{it:graph-dichotomy-segment} and Remark \ref{rem:segment-non-segment}, this shows also that, in the graph case \eqref{eq:graph-case}, the spectral gap of the $(p,w)$-Gibbs sampler on $\mathbb S_p$
is attained by a linear function in the energy variables $\eta_x=|\zeta_x|^{p_x}$  for every choice of positive weights $p$ if and only if $w$ is a weighted segment. To the best of our knowledge, this and analogous claims which may be deduced by combining Theorems \ref{th:gap-1-2-quantitative}, \ref{th:dichotomy-graph}, \ref{th:dichotomy-hypergraph} and \ref{th:GS-p} are new even in the special case of Kac's walk corresponding to $p\equiv 2$, for which previously only the homogeneous complete graph case \eqref{eq:eigKac} was understood. 

\subsubsection{Boundary-driven ${\rm KMP}$}\label{sec:reservoirs-intro}

Consider the ${\rm KMP}$ model in which, in addition to the conservative exchange dynamics \eqref{eq:B_up}--\eqref{eq:gen-KMP}, the system exchanges energy with external reservoirs. 
Fix parameters $\omega_x\ge 0$ and $\beta_x,\rho_x>0$, for $x\in[n]$. At rate $\omega_x$, the energy at $x$ is updated according to
\begin{equation}\label{eq:eta-x-post-reservoir}
	\eta_x\longmapsto U_x\tonde{\eta_x+\xi_x}\comma
\end{equation}
while all other coordinates remain unchanged; at each such update, the pair $(U_x,\xi_x)$ is sampled independently of the current configuration and of all previous updates, with
\begin{equation}\label{eq:xi-x-u-x}
	(U_x,\xi_x)
	\sim
	{\rm Beta}(\alpha_x,\beta_x)
	\otimes
	{\rm Gamma}(\beta_x,\rho_x)
	\fstop
\end{equation}
Here, ${\rm Gamma}(\beta,\rho)$ denotes the gamma distribution whose density is proportional to
$s\mapsto s^{\beta-1}e^{-s/\rho}$ on $[0,\infty)$. We write $\mathbf E_x$ for expectation with respect to the product law in \eqref{eq:xi-x-u-x}.

Letting $\eta^{x,U_x,\xi_x}$ denote the configuration obtained from $\eta$ through \eqref{eq:eta-x-post-reservoir}, the reservoir generator is
\begin{equation}\label{eq:gen-res}
	\mathcal L_\partial f(\eta)
	=
	\sum_{x\in[n]}\omega_x\,
	\tttonde{
		\mathbf E_x\big[
		f(\eta^{x,U_x,\xi_x})
		\big]
		-f(\eta)
	}
	\comma
	\qquad
	\eta\in\Omega_{\rm bd}\eqdef[0,\infty)^n
	\comma
\end{equation}
where $f$ is a bounded measurable function on $\Omega_{\rm bd}$. The boundary-driven KMP generator is
\begin{equation}\label{eq:gen-bd}
	\mathcal L_{\rm bd}
	=
	\mathcal L+\mathcal L_\partial
	\fstop
\end{equation}
Here, $\cL$ denotes the conservative ${\rm KMP}$ generator associated with some block weights $w=(w_B)_{B\subseteq[n]}$ and site weights $\alpha=(\alpha_x)_{x\in[n]}$,  introduced on $\Omega$ in \eqref{eq:gen-KMP} and naturally extended to $\Omega_{\rm bd}$ through the same update rules.

Assume that the hypergraph induced by $w$ is connected and that $\omega\not\equiv0$. In this general setting, it is possible to show that the boundary-driven process admits a unique invariant measure. However, we omit a detailed proof of this fact, as it will not be used in the proof of Theorem~\ref{th:gap-res} below. In the special case in which $\rho_x\equiv\rho_\ast$ is constant on $\supp(\omega)$, one can check that the process admits an explicit, product, reversible measure, given by 
\begin{equation}\label{eq:nu-rhoa}
	\nu
	=
	\underset{x\in[n]}\bigotimes
	{\rm Gamma}(\alpha_x,\rho_\ast)
	\fstop
\end{equation}
Thus, whenever $\rho_x\equiv \rho_\ast$, $\cL_{\rm bd}$ is self-adjoint on $L^2(\nu)=L^2(\Omega_{\rm bd},\nu)$.
If instead $\rho$ is non-constant on $\supp(\omega)$, the process is non-reversible, and its invariant measure---usually referred to as its nonequilibrium steady state---is generally neither in product form nor explicit. 
Several important questions concerning this measure remain open, even on the segment. For further details, we refer to \cite{de_masi_ferrari_gabrielli_hidden_2023,giardina_redig_tol_intertwining_2024}, where the choice $\beta_x=\alpha_x$ is considered, and to the original KMP work \cite{kipnis_heat_1982}, corresponding to the full-refresh update with $\beta_x=\infty$. We return to these cases and further possible generalizations in Remark \ref{rem:general-reservoir} below.

For every $k\ge0$, let $\mathscr Q_k$ denote the finite-dimensional space of polynomials of degree at most $k$ in the variables $\eta\in\Omega_{\rm bd}$; when restricting to $\Omega\subseteq\Omega_{\rm bd}$, $\mathscr Q_k$ returns $\mathscr P_k$ from Section \ref{sec:KMP-intro}. It follows directly from the update rules that
\begin{equation}\label{eq:Qk-invariance-bd}
	\cL_{\rm bd}\mathscr Q_k
	\subseteq
	\mathscr Q_k
	\comma\qquad k\ge0
	\fstop
\end{equation}
Indeed, the conservative updates preserve polynomial degree (cf.\ \eqref{eq:poly-invariance}), while a reservoir update \eqref{eq:eta-x-post-reservoir}--\eqref{eq:xi-x-u-x} maps a polynomial of degree at most $k$ into another polynomial of either equal or lower degree. Hence, the restriction of $\cL_{\rm bd}$ to $\mathscr Q_k$ is a finite-dimensional operator and admits polynomial eigenfunctions.

When $\mathcal L_{\rm bd}$ is self-adjoint, the invariance in \eqref{eq:Qk-invariance-bd} also yields the invariance of the orthogonal polynomial layers
\begin{equation}
\mathscr Q_{k,\ast}\eqdef	\mathscr Q_k\cap\mathscr Q_{k-1}^{\perp_\nu}
	\comma\qquad k\ge1
	\fstop
\end{equation}
Together with the density of polynomials in $L^2(\nu)$, this allows one to fully recover the $L^2(\nu)$ spectrum from polynomial eigenfunctions.
This argument breaks down in nonequilibrium. Indeed, although $\mathscr Q_k$ remains invariant, the orthogonal layer $\mathscr Q_{k,\ast}$ need not be invariant. Moreover, the restriction of $\cL_{\rm bd}$ to $\mathscr Q_k$ need not be self-adjoint or even diagonalizable, and its eigenvalues need not exhaust the spectrum of $\cL_{\rm bd}$ on $L^2(\nu)$.

Regardless of this issue, one may always consider the polynomial spectrum 
$
	\spec(-\cL_{\rm bd}|_{\mathscr Q})
$,
 $\mathscr Q\eqdef \cup_{k\ge 0}\,\mathscr Q_k$, defined as the collection of all eigenvalues of $\cL_{\rm bd}$ on $\mathscr Q$, counted with multiplicity.
We shall prove that this spectrum is real, nonnegative and independent of $\rho$.  Accordingly, we define the polynomial spectral gap by
\begin{equation}\label{eq:def-polynomial-gap-bd}
	\gap(\cL_{\rm bd})
	\eqdef
	\inf_{k\ge 1}\gap(\cL_{\rm bd}|_{\mathscr Q_k})\comma
\end{equation}
where  $\gap(\cL_{\rm bd}|_{\mathscr Q_k})$ denotes the second smallest eigenvalue of $-\cL_{\rm bd}|_{\mathscr Q_k}$.
In the reversible case, $\gap(\cL_{\rm bd})$ coincides with the usual $L^2(\nu)$-spectral gap.
Our main finding is that, in contrast with the conservative setting of Theorem \ref{th:main-KMP}, this gap is always attained by a polynomial of degree one.

\begin{theorem}[Boundary-driven KMP]\label{th:gap-res}
	For all positive weights
	$\alpha, 
	\beta,
	\rho$ 
	and all non-negative weights
	$
	w$ and $\omega$, 
one has
	\begin{equation}\label{eq:gap-res-intro}
		\gap(\cL_{\rm bd})
		=
		\gap(
		\cL_{\rm bd}|_{\mathscr Q_1})
		\fstop
	\end{equation}
	In particular, $\gap(\cL_{\rm bd})$ does not depend on $\rho$. 
\end{theorem}

As in the conservative case, this result admits a particle-system interpretation: the relevant intertwining now involves the particle dynamics from Section \ref{sec:KMP-particle-intro} with an additional killing of particles at reservoir sites. We refer to Section \ref{sec:KMP-reservoirs} for further details.

\subsubsection{Stochastic exchange models}\label{sec:SEM-intro}

The arguments used to prove Theorems \ref{th:main-KMP} and \ref{th:disc-ell-KMP} for the conservative ${\rm KMP}$ model extend to a considerably broader class of processes, which we refer to as \emph{stochastic exchange models}; see also \cite{kim_quattropani_sau_spectral_2025,caputo2025universal,casanova2025partially}. Roughly speaking, these are Markov processes on $\Omega$ whose updates redistribute the energy through multiplication with random stochastic matrices.

More specifically, let $\Upsilon$ be a Borel measure on the space of row-stochastic matrices $M=(M_{xy})_{x,y\in [n]}\in[0,1]^{n\times n}$. The dynamics on $\Omega$ is formally described, for $f:\Omega\to \R$, by
\begin{equation}\label{eq:gen-SEM-intro}
	\cL^\tups f(\eta)
	\eqdef
	\int \Upsilon(\dd M)
	\tonde{
		f(\eta M)-f(\eta)
		}
	\comma\qquad \eta\in \Omega\fstop
\end{equation}
Here, $\eta M$ denotes row-vector--matrix multiplication; in particular, $\eta M\in\Omega$ whenever $\eta\in\Omega$.
We defer the precise construction of stochastic exchange models and further properties to Section \ref{sec:SEM}. For now, let us simply note that the ${\rm KMP}$ generator in \eqref{eq:gen-KMP} corresponds to (cf.~\eqref{eq:B_up})
\begin{equation}\label{eq:Upsilon-KMP}
	\Upsilon = \Upsilon^\tKMP
	=
	\textstyle{\sum_{B\subseteq[n]}}\,
	w_B\,{\rm Law}(K_B^U)
	\comma\qquad
	U=(U_x)_{x\in B}\sim{\rm Dir}(\alpha^B)
	\comma
\end{equation}
where $K_B^U\in \R^{n\times n}$ is the random row-stochastic matrix  such that, for every $\eta \in \Omega$,
\begin{equation}\label{eq:K_B^U}
	(\eta K_B^U)_x = \eta(B)U_x\comma x \in B\comma\qquad (\eta K_B^U)_x = \eta_x\comma x \notin B\fstop
\end{equation}  Observe that the $B\times B$ block of $K_B^U\in \R^{n\times n}$ has identical rows, each equal to $U=(U_x)_{x\in B}$.

As further examples, we anticipate that suitable variants of the ${\rm KMP}$ redistribution law in \eqref{eq:Upsilon-KMP} yield:
\begin{itemize}
	\item the \textit{averaging process} considered in \cite{aldous_lecture_2012,quattropani2021mixing}, obtained by taking the redistribution vector to be deterministic, with $U_x=\frac{\alpha_x}{\alpha(B)}$ for $x\in B$;
	\item a block version of the \textit{immediate exchange model} introduced in \cite{heinsalu_patriarca_kinetic_2014}, obtained by taking the rows of the $B\times B$ update block to be {independent}---rather than identical---Dirichlet vectors with suitable parameters; see Section \ref{sec:IEM}.
\end{itemize}	
 In all these examples, $\Upsilon$ is finite, and the corresponding generator $\cL^\tups$ is bounded on $\cC(\Omega)$, the space of continuous functions on $\Omega$ endowed with the uniform norm. There are, however, relevant models described by an infinite measure $\Upsilon$, most notably the \textit{harmonic process} \cite{frassek2021exact}; see Section \ref{sec:HP}.  Each of the examples just mentioned admits a Dirichlet distribution as a reversible measure; for the averaging process, this degenerates to a Dirac mass.

 To encompass all the above examples, we always impose, as in \cite{kim_quattropani_sau_spectral_2025}, the condition
 \begin{equation}\label{eq:SEM-assumption-intro}
\textstyle 	\int \Upsilon(\dd M)
 	M_{xy}
 	<\infty\comma\qquad \text{for every}\ x,y\in [n]\comma x\neq y\fstop
 \end{equation}
This guarantees that \eqref{eq:gen-SEM-intro} generates a Feller process on $\Omega$. By compactness, an invariant probability measure always exists (Remark \ref{rem:SEM-invariant-measures}), while Theorem~\ref{th:SEM} and Proposition~\ref{pr:SEM-uniqueness} provide explicit equivalent
criteria for its uniqueness.

One key common feature of all these stochastic exchange models is that, under \eqref{eq:SEM-assumption-intro}, the linearity of the updates in \eqref{eq:gen-SEM-intro} guarantees
\begin{equation}\label{eq:SEM-invariance-intro}
	\cL^\tups\mathscr P_k
	\subseteq
	\mathscr P_k
	\comma\qquad
	k\ge0\comma
\end{equation}
where $\mathscr P_k$ is the same space introduced in Section \ref{sec:KMP-intro}, namely, the space of polynomials of degree at most $k$ in the variables $\eta\in\Omega$. 
As in the preceding discussion on the boundary-driven ${\rm KMP}$ model, the eigenvalues obtained by diagonalizing $\cL^\tups$ on the finite-dimensional spaces $\mathscr P_k$ need not be in one-to-one correspondence with the spectrum of any $L^2$ realization of the generator. Nevertheless, we may define the polynomial spectral gap by
\begin{equation}\label{eq:SEM-gap-intro}
	\gap(\cL^\tups)
	\eqdef
	\inf_{k\ge1}
	\gap(\cL^\tups|_{\mathscr P_k})
	\comma
\end{equation}
where $\gap(\cL^\tups|_{\mathscr P_k})$ is the smallest real part among the eigenvalues of $-\cL^\tups|_{\mathscr P_k}$ after removing one copy of the zero eigenvalue corresponding to constant functions, with algebraic multiplicities taken into account.
With these definitions, the degree-two reduction established for the ${\rm KMP}$ model in Theorem \ref{th:main-KMP} in fact holds throughout the whole class of stochastic exchange models. 
\begin{theorem}[Stochastic exchange models]\label{th:SEM}
	For every Borel measure $\Upsilon$ satisfying \eqref{eq:SEM-assumption-intro},
	\begin{equation}\label{eq:SEM-gap-identity}
		\gap(\cL^\tups)
		=
		\gap(
		\cL^\tups|_{\mathscr P_2}
		)
		\fstop
	\end{equation}
	Moreover, $\cL^\tups$ admits a \textit{unique} invariant measure precisely when $\gap(\cL^\tups)>0$.
\end{theorem}
We stated the above result for restrictions to polynomial observables. If, in addition, the process is reversible with respect to a probability measure which is absolutely continuous with respect to the uniform measure on $\Omega$, then distinct polynomial functions on $\Omega$ define distinct elements of the corresponding $L^2$-space. Since polynomials are dense in this space, self-adjointness yields a complete orthonormal basis of polynomial eigenfunctions. Consequently, $\gap(\cL^\tups)$ coincides with the usual $L^2$-spectral gap. This applies, in particular, whenever the reversible measure is a nondegenerate Dirichlet distribution, thus, to the immediate exchange model and harmonic process; see Corollaries \ref{cor:IEM} and \ref{cor:HP}.

We provide further motivation for this general spectral gap reduction in the next section, where we outline the key steps of its proof.

\subsection{Hidden models,  proof strategy, and further remarks}\label{sec:proof-strategy}
As for the ${\rm KMP}$ model, every stochastic exchange model as introduced in Section \ref{sec:SEM-intro} admits a natural particle representation. Consequently, the spectral problems for polynomial eigenfunctions underlying Theorem \ref{th:main-KMP} and its extension in Theorem \ref{th:SEM} can be equivalently reformulated in terms of eigenpairs of the associated finite-particle systems; see \eqref{eq:SEM-gap-particle-interpretation}. The most effective viewpoint for proving the degree-two reductions in those theorems, however, comes from a different representation: the \textit{hidden} (\textit{temperature} or \textit{parameter}) \textit{model}, first identified for ${\rm KMP}$ in \cite{de_masi_ferrari_gabrielli_hidden_2023}; see also \cite{giardina_redig_tol_intertwining_2024,kim_quattropani_sau_spectral_2025} for subsequent developments and extensions.

\subsubsection{Hidden models}\label{sec:hidden-model-intro}
To introduce the hidden model in an intuitive way, first observe that the ${\rm KMP}$ model---as any other stochastic exchange model with generator $\cL^\tups$ as in \eqref{eq:gen-SEM-intro}---is built from elementary random matrix updates of the form
\begin{equation}\label{eq:eta-M}
	\eta\longmapsto \eta M\comma
\end{equation}
where $\eta\in\Omega\subseteq\R^n$, viewed as a row vector, is a probability distribution on $[n]$, and $M\in\R^{n\times n}$ is row-stochastic.  For instance, the ${\rm KMP}$ update in \eqref{eq:2_up}, restricted to the coordinates $x,y=x+1\in[n]$ (as the others remain unchanged), can be written as
\begin{equation}\label{eq:KMP-update-segment-matrix}
	(\eta_x,\eta_y)
	\longmapsto
	(\eta_x,\eta_y)
	\begin{pmatrix}
		U & 1-U\\
		U & 1-U
	\end{pmatrix} = \begin{pmatrix}
		U(\eta_x+\eta_y),(1-U)(\eta_x+\eta_y)
	\end{pmatrix}\fstop
\end{equation}
Hence, at any time $t>0$, after $\ell=\ell(t)$ updates \eqref{eq:eta-M} with matrices $M_1,\ldots,M_\ell$, the energy configuration is
$\eta M_1\cdots M_\ell$.
Crucially, the update times and matrices are sampled independently of the current energy configuration. The same sequence may therefore be read in reverse chronological order and applied, by left multiplication, to a column vector, still yielding a Markov dynamics.

This backward evolution is precisely the \textit{hidden model}: its configurations are column vectors $\theta\in\R^n$, and each update takes the form
\begin{equation}
	\theta\longmapsto M\theta\comma
\end{equation}
so that the same row-stochastic matrix $M$ now acts on the hidden-model configuration from the left.
In particular, the matrix representation of the ${\rm KMP}$ update in \eqref{eq:KMP-update-segment-matrix} gives the corresponding $\theta$-variable update
\begin{equation}\label{eq:hidden-update-segment}
	\begin{pmatrix}
		\theta_x\\
		\theta_y
	\end{pmatrix}
	\longmapsto
	\begin{pmatrix}
		U & 1-U\\
		U & 1-U
	\end{pmatrix}
	\begin{pmatrix}
		\theta_x\\
		\theta_y
	\end{pmatrix}
	=
	\begin{pmatrix}
		U\theta_x+(1-U)\theta_y\\
		U\theta_x+(1-U)\theta_y
	\end{pmatrix}\fstop
\end{equation}
Observe that, whereas the forward update in \eqref{eq:KMP-update-segment-matrix} redistributes the conserved total energy $\eta_x+\eta_y$ between the two sites, the hidden-model update in \eqref{eq:hidden-update-segment} replaces $\theta_x$ and $\theta_y$ by the same random convex combination and, in general, does not preserve their sum. The same description extends to the ${\rm KMP}$ hidden model on a weighted hypergraph: whenever a block is updated, all its coordinates are replaced by a common random convex combination. 
Thus, for the ${\rm KMP}$ hidden-model dynamics,
\begin{enumerate}[(i)]
	\item\label{item:hidden-flat} every flat configuration $s\mathbf1$, $s\in\R$, is invariant;
	\item\label{item:hidden-max-principle} the minimum coordinate is nondecreasing and the maximum coordinate is nonincreasing, so every hypercube $[a,b]^n$, $-\infty< a<b< \infty$,	 is invariant;
	\item\label{item:hidden-nonconservative} the total mass need not be conserved.
\end{enumerate}

The properties in \ref{item:hidden-flat}--\ref{item:hidden-nonconservative} are not specific to the hidden model of the ${\rm KMP}$, but are shared by the hidden model of any stochastic exchange model from Section \ref{sec:SEM-intro}. Indeed, every row-stochastic matrix $M$ satisfies $M\mathbf1=\mathbf1$ and each coordinate of $M\theta$ is a convex combination of the coordinates of $\theta$, yielding respectively the invariance of flat configurations in \ref{item:hidden-flat} and the maximum principle in \ref{item:hidden-max-principle}. Moreover, in general,
\begin{equation}
	\textstyle
	\sum_{x\in[n]}(M\theta)_x
	\neq
	\sum_{x\in[n]}\theta_x\comma
\end{equation}
so the hidden-model dynamics is typically non-conservative, i.e., \ref{item:hidden-nonconservative}.

 Under an appropriate connectivity assumption, the hidden-model dynamics converges, in the long run, to a flat configuration with a possibly random height. Although the degeneracy of these steady states may at first appear as a complication, it is precisely what makes the hidden-model representation effective for the spectral gap problem.

\subsubsection{Eigenvalue problem for hidden models} 
The original ${\rm KMP}$ model, its particle representation, and its hidden model all encode the same polynomial eigenvalue problem. Indeed, for given weights $\alpha=(\alpha_x)_{x\in [n]}$ and $w=(w_B)_{B\subseteq[n]}$, $k\ge1$, and $\psi\in\mathscr H_k$ as defined in \eqref{eq:coeff-space-sym}, next to the $\eta$-variable polynomial $\widehat\psi\in\mathscr P_k$ introduced in \eqref{eq:intro-widehat-psi}, we also define
\begin{equation}\label{eq:intro-tilde-psi}
	\widetilde\psi(\theta)
	\eqdef
	\sum_{x_1,\ldots,x_k\in[n]}
	\mu_k(x_1,\ldots,x_k)\,
	\psi(x_1,\ldots,x_k)\,
	\theta_{x_1}\cdots\theta_{x_k}\comma
	\qquad \theta\in\R^n\comma
\end{equation}
where $\mu_k$ is the reversible measure of the $k$-particle system; see \eqref{eq:mu-k}. Comparing $\widetilde\psi$ in \eqref{eq:intro-tilde-psi} with $\widehat\psi$ in \eqref{eq:intro-widehat-psi}, the two polynomials differ only by the additional coefficient $\mu_k(x_1,\ldots,x_k)$ attached to each monomial.

Recall that $\cL$ denotes the generator of the $(\alpha,w)$-${\rm KMP}$ model, while $L_k$ is the generator of the associated $k$-particle system. As seen in \eqref{eq:intro-intertwining}, the map $\psi\mapsto\widehat\psi$ on $\mathscr H_k$ intertwines these two generators; the map $\psi \mapsto \widetilde \psi$ in \eqref{eq:intro-tilde-psi} yields the corresponding relation with the hidden-model dynamics. More precisely, if $\mathscr L$ denotes the generator of the hidden model, then
\begin{equation}\label{eq:intro-forward-backward-intertwining}
	\cL\widehat\psi
	=
	\widehat{L_k\psi}
	\qquad\text{and}\qquad
	\mathscr L\widetilde\psi
	=
	\widetilde{L_k\psi}\comma\qquad \text{for every}\ \psi \in \mathscr H_k\fstop
\end{equation}
In view of these intertwining relations, we shall prove that, for every $\lambda\ge 0$ and $\psi \in \mathscr H_k$,
\begin{equation}\label{eq:eigenvalue-problems-three}
	\cL\widehat\psi=-\lambda\widehat\psi
	\qquad\Longleftrightarrow\qquad
	L_k\psi=-\lambda\psi
	\qquad\Longleftrightarrow\qquad
	\mathscr L\widetilde\psi=-\lambda\widetilde\psi\fstop
\end{equation}
Since $\psi$, $\widehat\psi$, and $\widetilde\psi$ may only vanish simultaneously, the eigenvalue problem may be formulated interchangeably for the ${\rm KMP}$ generator, the $k$-particle system on symmetric observables, or the hidden model.  

In view of the degeneracy of the hidden model steady states, the last formulation in \eqref{eq:eigenvalue-problems-three} has the advantage of translating the problem of identifying the optimal $L^2(\mu)$ convergence rate of the ${\rm KMP}$ dynamics toward its Dirichlet equilibrium, as in \eqref{eq:gap_decay}, into the more tractable task of quantifying the decay of spatial fluctuations in the hidden model.

\subsubsection{Spatial variance as a Lyapunov function}\label{sec:variance-intro}

A natural first candidate for measuring spatial fluctuations of the hidden-model profile is the variance with respect to the probability measure induced by $\alpha=(\alpha_x)_{x\in[n]}$. More precisely, for $x\in[n]$ and $\theta\in\R^n$, set
\begin{equation}\label{eq:var-pi-intro}
	\textstyle
	\pi_x\eqdef\frac{\alpha_x}{\alpha_0}\comma
	\qquad
	\pi(\theta)\eqdef\sum_{x\in[n]}\pi_x\,\theta_x\comma
	\qquad
	\var_\pi(\theta)
	\eqdef
	\sum_{x\in[n]}\pi_x\tonde{\theta_x-\pi(\theta)}^2\fstop
\end{equation}
Since $\theta\mapsto \var_\pi(\theta)$ is nonnegative and vanishes precisely on flat configurations, consider its optimal exponential contraction rate
\begin{equation}\label{eq:def-delta}
	\delta(\alpha,w)
	\eqdef
	\sup\ttset{
		\delta\ge0:
		e^{t\mathscr L}\var_\pi(\theta)
		\le
		e^{-\delta t}\var_\pi(\theta)
		\ \text{for all}\ t\ge0,\, \theta\in\R^n
		}\comma
\end{equation}
where, here and in what follows, $(e^{t\mathscr L})_{t\ge0}$ denotes the hidden-model semigroup.
Following the strategy introduced in \cite{kim_quattropani_sau_spectral_2025}, and here extended to the weighted-hypergraph setting, one obtains a quantitative lower bound on
$
	\inf_{k\ge2}\gap_{k,\ast}(\alpha,w)
$
in two separate steps, as we now briefly sketch; for further details, we refer to Sections~\ref{sec:proof-main-KMP} and \ref{sec:tau-infty}.

First, one estimates the decay of the variance through an infinitesimal contraction computation. The resulting bound involves only $\gap_1(\alpha,w)$---and hence spectral information carried by linear observables of the energy variables---together with the factor 
\begin{equation}
	\gamma(\alpha,w)\eqdef \frac{\alpha_{w,{\rm min}}}{1+\alpha_{w,{\rm min}}}\tonde{1+\frac1{\alpha_0}}\in(0,1)
\end{equation} appearing on the left-hand side of \eqref{eq:gap-1-2-quantitative}. In fact,  a direct computation \cite{kim_quattropani_sau_spectral_2025} yields
\begin{equation}\label{eq:var-der-intro}
	\mathscr L\var_\pi(\theta)
	\le
	-\gamma(\alpha,w)\,\cE_1(\theta)
	\le
	-\gamma(\alpha,w)\,\gap_1(\alpha,w)\,\var_\pi(\theta)\comma\qquad \theta \in \R^n\comma
\end{equation}
where $\cE_1$ denotes the Dirichlet form of the one-particle dynamics with rates \eqref{eq:RW-rates} (see \eqref{eq:dirichlet-form-1}), while the second inequality follows by the variational characterization of $\gap_1(\alpha,w)$. Gr\"onwall's inequality then gives
\begin{equation}\label{eq:intro-delta-lower-bound}
	\delta(\alpha,w)
	\ge
	\gamma(\alpha,w)\,\gap_1(\alpha,w)\fstop
\end{equation}

Second, the quadratic vanishing of hidden-model eigenfunctions at flat configurations allows the decay estimate for the variance to be converted into the uniform bound
\begin{equation}\label{eq:gap-k-delta}
	\gap_{k,\ast}(\alpha,w)\ge\delta(\alpha,w)\comma\qquad k\ge 2\fstop
\end{equation}
Indeed, if $g=\widetilde\psi\neq 0$ is a hidden-model eigenfunction associated to $\lambda=\gap_{k,\ast}(\alpha,w)$,  $k\ge2$, then
\begin{equation}\label{eq:g-var}
	|g(\theta)|
	\le
	C\var_\pi(\theta)\comma
	\qquad \theta\in[a,b]^n\comma
\end{equation}
for every $-\infty<a<b<\infty$ and some $C=C(g,k,\alpha,a,b)>0$.
Hence,  $\mathscr Lg=-\lambda g$, positivity of the hidden-model semigroup $(e^{t\mathscr L})_{t\ge 0}$, the maximum principle in \ref{item:hidden-max-principle}, \eqref{eq:g-var}, and \eqref{eq:def-delta} give, for every $t\ge 0$ and $\theta \in [a,b]^n$,
\begin{equation}\label{eq:chain-1}
	e^{-\lambda t}|g(\theta)|
	=
	|e^{t\mathscr L}g(\theta)|
	\le
	e^{t\mathscr L}|g|(\theta)
	\le C e^{t\mathscr L}\var_\pi(\theta)\le
	Ce^{-\delta(\alpha,w) t}\var_\pi(\theta)\fstop
\end{equation}
As $g\neq 0$ on $[a,b]^n$, letting $t\to\infty$ in \eqref{eq:chain-1} proves \eqref{eq:gap-k-delta}.

Altogether, \eqref{eq:intro-delta-lower-bound} and \eqref{eq:gap-k-delta} yield the lower bound
\begin{equation}
	\inf_{k\ge 2}\gap_{k,\ast}(\alpha,w)\ge \gamma(\alpha,w)\,\gap_1(\alpha,w)\comma
\end{equation}
which, while quantitative, does not yet identify the lowest-degree gap.

\subsubsection{A Perron-Frobenius theorem for nonnegative quadratic functions}\label{sec:PF-intro} The preceding argument does not allow, in general, to prove that the optimal variance-contraction rate $\delta(\alpha,w)$ in \eqref{eq:def-delta} \textit{always} coincides with $\gap_{2,\ast}(\alpha,w)$. A closer inspection, though, shows that this identity does hold in \textit{some} instances, namely, in homogeneous mean-field settings, such as
\begin{equation}\label{eq:mean-field}
\textstyle\alpha\equiv{\rm const.}\comma	w_B=\car_{{|B|=\ell}}
	\comma\qquad \text{for some}\ \ell=2,\ldots,n-1\fstop
\end{equation}
Indeed, in this case all the inequalities in \eqref{eq:var-der-intro} become identities (Remark \ref{rem:gap-lower-sharp}), yielding
\begin{equation}
	\mathscr L\var_\pi
	=
	-\gamma(\alpha,w)\,\gap_1(\alpha,w)\,\var_\pi \fstop
\end{equation}
Thus, in this context, $\var_\pi$ is itself an eigenfunction of $-\mathscr L$ with eigenvalue $\gamma(\alpha,w)\,\gap_1(\alpha,w)$, which necessarily coincides with $\delta(\alpha,w)$ in \eqref{eq:def-delta}. Moreover, $\var_\pi$ is a nonzero, quadratic polynomial in the $\theta$-variables that vanishes precisely on constant configurations. Leveraging on the correspondence in \eqref{eq:eigenvalue-problems-three}, one then proves that $\var_\pi$ is a genuine degree-two eigenfunction, so that $\gap_{2,\ast}(\alpha,w)\le\gamma(\alpha,w)\,\gap_1(\alpha,w)$. Combined with \eqref{eq:gap-k-delta} for $k=2$, this forces
\begin{equation}
	\gamma(\alpha,w)\,\gap_1(\alpha,w)
	=
	\delta(\alpha,w)
	=
	\gap_{2,\ast}(\alpha,w)\fstop
\end{equation}

Hence, in the homogeneous mean-field setting \eqref{eq:mean-field}, one obtains not only
\begin{equation}
	\inf_{k\ge2}\gap_{k,\ast}(\alpha,w)
	=
	\gap_{2,\ast}(\alpha,w)\comma
\end{equation}
but also the stronger structural fact that $\var_\pi$ is a nonnegative quadratic eigenfunction of $-\mathscr L$ associated with $\gap_{2,\ast}(\alpha,w)$ and vanishing precisely on flat configurations. These are exactly the properties used in the preceding argument: nonnegativity and quadratic vanishing yield the pointwise domination in \eqref{eq:g-var}, while the eigenvalue equation provides the sharp decay rate $\gap_{2,\ast}(\alpha,w)$, rather than merely a computable lower bound.

For general geometries, $\var_\pi$ is typically no longer an eigenfunction. The central finding of our approach is that the properties above can nevertheless always be recovered simultaneously for any ${\rm KMP}$ hidden-model dynamics: for all positive site weights $\alpha$ and nonnegative block weights $w$, under suitable irreducibility assumptions there exists a nonzero quadratic function $g_\ast\ge0$, strictly positive away from flat configurations, such that
\begin{equation}\label{eq:intro-positive-quadratic-eigenfunction}
	\mathscr Lg_\ast
	=
	-\gap_{2,\ast}(\alpha,w)\,g_\ast\fstop
\end{equation}
Thus, $g_\ast$ may replace $\var_\pi$ as a Lyapunov function in \eqref{eq:chain-1}, while capturing the optimal rate; cf.\ \eqref{eq:chain-2}.

Let us emphasize that the existence of such a one-signed eigenfunction is a distinctive feature of the absorbing hidden-model dynamics. Every nonconstant eigenfunction of the original ${\rm KMP}$ generator associated with a nonzero eigenvalue is orthogonal to constants in $L^2(\mu)$ and must therefore change sign. In the hidden-model variables, instead, in view of the degenerate form of its steady state, the corresponding eigenfunction may be nonnegative, vanishing on the absorbing family of flat configurations.

Ultimately, the key step is therefore to establish the existence of such a distinguished eigenfunction $g_\ast$.
 Once the appropriate functional setting has been identified, this existence result follows rather directly from Perron--Frobenius theorems applied on suitable cones of nonnegative quadratic hidden-model observables. Although the underlying tool is classical, its application to degree-two polynomials in the hidden-model variables appears to be both novel and rather unexpected. It is also remarkably powerful: the resulting eigenfunction encodes substantial structural information, which we exploit systematically in the proofs of the characterization theorems in Section \ref{sec:comparison-intro}. Finally, this same mechanism proves to be quite universal, extending beyond the ${\rm KMP}$ model to all stochastic exchange models.

\subsubsection{Further remarks and results}

In fact, although the proof sketch above captures the essence of our approach, a Perron--Frobenius theorem as described there---and, in fact, the stronger version established in Proposition \ref{pr:PF-strong}---plays a key role only in the finer comparison results from Section \ref{sec:comparison-intro}. For the proofs of Theorems \ref{th:main-KMP}--\ref{th:disc-ell-KMP} and of their stochastic-exchange-model counterpart, Theorem \ref{th:SEM}, it suffices instead to exhibit a suitable quadratic nonnegative and nondegenerate---in the sense of \eqref{eq:cone-interior}---polynomial, not necessarily an eigenfunction of the generator. A convenient choice for such a polynomial is the variance functional $\var_\pi$ from \eqref{eq:var-pi-intro}.	

This more general principle also yields refined information on higher-degree spectral gaps. Indeed, at every even degree $\ell\ge2$, the same role is played by the homogeneous polynomial $\var_\pi^{\ell/2}$, yielding $\gap_{k,\ast}(\alpha,w)\ge\gap_{k-1,\ast}(\alpha,w)$ for odd $k\ge3$ and $\gap_{k,\ast}(\alpha,w)\ge\gap_{k-2,\ast}(\alpha,w)$ for even $k\ge4$. See Remarks \ref{rem:new-ordering} and \ref{rem:even-degree-cones}, as well as Remark \ref{rem:SEM-higher-degree} for the analogous statement for stochastic exchange models. In particular, when $\alpha\equiv1$, the monotonicity along even degrees in  \eqref{eq:even-chain} proves another conjecture raised by Alon and Puder in \cite{alon2026aldous}, namely, part of \cite[Conjecture 4.16]{alon2026aldous}.

\subsubsection{Open problems}
As already discussed at the beginning of Section~\ref{sec:intro}, Aldous-type spectral gap phenomena have by now been identified both for classical interacting particle systems and for stochastic dynamics on groups, creating links between statistical and quantum physics, combinatorics, group theory, and representation theory; see, e.g., \cite{cesi_remarks_2016,hermon_version_2019,quattropani2021mixing,bristiel_caputo_entropy_2021,kim_sau_spectral_2023,alon_kozma_puder_aldous_2025,alon2026aldous,levhari2026aldous,greaves2026aldous,zhu2026random} and references therein. All previously known results and conjectures associated with the symmetric group and its variants can be interpreted, in the corresponding particle or energy variables, as degree-one spectral gap reductions. A prominent open problem in this direction is Caputo's conjecture for the interchange process on arbitrary weighted hypergraphs; see Remark~\ref{rem:IP}. Another related important open problem is the general degree-two reduction conjectured by Alon and Puder for the unitary group, see \cite[Conjecture 1.7]{alon2026aldous}. 
Our results provide the first proof of general degree-two spectral gap reductions on arbitrary geometries for a broad class of stochastic exchange dynamics.
 Finding further occurrences of this phenomenon, and identifying a common mechanism behind them, remains an interesting open problem.

A natural test case is the \textit{Brownian energy process} (${\rm BEP}$), another conservative stochastic dynamics of energies on a weighted $n$-site graph, parametrized by site weights $\alpha$ and edge weights $w$ as the ${\rm KMP}$ model, but evolving diffusively rather than through random exchange updates. On $\Omega$, it is reversible with respect to $\mu={\rm Dir}(\alpha)$ and has generator
\begin{equation}
	\cL^{\tbep} f(\eta)
	=
	\frac12\sum_{x,y\in[n]}w_{\set{x,y}}
	\set{
		-\tonde{\alpha_y\eta_x-\alpha_x\eta_y}
		\tonde{\partial_{\eta_x}-\partial_{\eta_y}}
		+
		\eta_x\eta_y
		\tonde{\partial_{\eta_x}-\partial_{\eta_y}}^2
	}f(\eta)\fstop
\end{equation}
As for the stochastic exchange models above, $\cL^{\tbep}$ preserves polynomial degree and is intertwined with a particle system, known as the \textit{symmetric inclusion process} (${\rm SIP}$). The results of \cite{kim_sau_spectral_2023,kim_sau_one_2024} establish one-particle domination when $\min_{x\in[n]}\alpha_x\ge1$, show that it may fail outside this regime, and prove an asymptotic two-particle reduction under the rescaling $\alpha\mapsto\tau\alpha$ as $\tau\to0$. Whether, for arbitrary weighted graphs and positive site weights, the spectral gap is always attained at degree at most two remains open.

Our hidden-model approach does not seem to extend directly to this problem. The formal hidden-model operator would read
\begin{equation}
	\mathscr L^{\tbep}g(\theta)
	=
	\frac12\sum_{x,y\in[n]}w_{\set{x,y}}
	\set{
		-\tonde{\theta_x-\theta_y}
		\tonde{\alpha_y\partial_{\theta_x}-\alpha_x\partial_{\theta_y}}
		+
		\tonde{\theta_x-\theta_y}^2
		\partial_{\theta_x}\partial_{\theta_y}
	}g(\theta)\fstop
\end{equation}
but this operator is not Markovian. Indeed, if $w_{\set{x,y}}>0$ and $c\neq0$, then $g(\theta)=(\theta_x-\theta_y-c)^2$ attains its minimum whenever $\theta_x-\theta_y=c$, while $\mathscr L^{\tbep}g(\theta)=-2w_{\set{x,y}}c^2<0$ there, violating the positive minimum principle. Thus, even in this closely related diffusion setting, a proof of the conjectural degree-two reduction appears to require a different mechanism.

\subsection*{Organization of the paper} The rest of the paper is organized as follows.
Section~\ref{sec:KMP-main} is devoted to the proof of Theorems \ref{th:main-KMP}--\ref{th:disc-ell-KMP} and Theorem \ref{th:GS-p}. Section~\ref{sec:PF} develops the Perron--Frobenius theory used in the comparison between one- and two-particle gaps. Section~\ref{sec:hypergraph-reduction} reduces this comparison to minimal hypergraphs, while Section~\ref{sec:comparison} completes the proofs of the results stated in Section~\ref{sec:comparison-intro}. 
Section~\ref{sec:KMP-reservoirs} treats the boundary-driven ${\rm KMP}$ model and proves Theorem \ref{th:gap-res}. Finally, Section~\ref{sec:SEM} develops the general theory of stochastic exchange models, proves Theorem \ref{th:SEM}, and discusses several examples. The appendix contains the refinement needed in the analysis of minimal hypergraphs.
Sections~\ref{sec:Gibbs-sampler} and~\ref{sec:KMP-reservoirs} may each be read independently of the remaining sections. 

\section{Aldous' phenomenon in ${\rm KMP}$}\label{sec:KMP-main}
	The main goal of this section is to prove Theorems \ref{th:main-KMP}, \ref{th:disc-ell-KMP} and \ref{th:GS-p}. We start by presenting all definitions, properties and intertwining relations involving the ${\rm KMP}$ model, the corresponding particle system, and its hidden model. We then turn to the proof of Theorems \ref{th:main-KMP} and \ref{th:disc-ell-KMP} in Section \ref{sec:proof-main-KMP}. Finally, in Section~\ref{sec:Gibbs-sampler} we show how a mild extension of these results leads to the statement for Gibbs samplers on cone measures from Theorem \ref{th:GS-p}.

\subsection{Energies, particles,  hidden models, and intertwining relations}
\label{sec:prelim-particle-systems}
We now make precise the three equivalent formulations of the polynomial spectral problem
announced in \eqref{eq:eigenvalue-problems-three}. For this purpose, we first describe the ${\rm KMP}$ dynamics, its labeled-particle
representation, and the corresponding hidden model. Then, we identify the polynomial spaces on
which the three generators are intertwined. The positive site weights $\alpha$ and nonnegative block weights $w$ are fixed throughout the section, without imposing any additional assumption. For every integer $k\ge 1$, we shall use the shorthand notation
\begin{equation}\label{eq:coeff-space}
	\mathscr H_k^\per\eqdef \R^{[n]^k}=\ttset{\psi:[n]^k\to \R}\comma
\end{equation}
and write $L^2(\mu_k)$ when endowing $\mathscr H_k^\per$ with the $\mu_k$-inner product. Further, recall that  $\mathscr H_k\subseteq \mathscr H_k^\per$, defined 	in \eqref{eq:coeff-space-sym}, is its subspace of symmetric functions.

\subsubsection{KMP, particle, and hidden models}

Recall from \eqref{eq:Upsilon-KMP}--\eqref{eq:K_B^U} that a $B$-update of the ${\rm KMP}$ model is represented by the right multiplication
$
\eta\longmapsto\eta K_B^U
$,
where $U=(U_y)_{y\in B}\sim{\rm Dir}(\alpha^B)$. We extend the definition of $K_B^U$ from \eqref{eq:K_B^U} as follows: if $u=(U_x)_{x\in [n]}$ is a probability vector with $u(B)=\sum_{x\in B}u_x>0$, let $K_B^u\in\R^{n\times n}$ be the row-stochastic matrix
\begin{equation}\label{eq:K_B^U2}
	(K_B^u)_{xy}
	=
	\begin{dcases}
		\frac{u_y}{u(B)}
		&\text{if}\ x,y\in B\comma\\
		1
		&\text{if}\ x=y\notin B\comma\\
		0
		&\text{otherwise}\fstop
	\end{dcases}
\end{equation}
In particular, with a slight abuse of notation, we shall identify probability vectors on $B$ with their extensions by zero to $[n]$; for such vectors, $u(B)=1$ and \eqref{eq:K_B^U2} agrees with \eqref{eq:K_B^U}.

For each $B\subseteq[n]$, let $\mathbf E_B$ denote expectation with respect to $U\sim{\rm Dir}(\alpha^B)$. Thus, the ${\rm KMP}$ generator in \eqref{eq:gen-KMP} may be rewritten as
\begin{equation}\label{eq:gen-KMP2}	
	\cL f(\eta)
	=
	\sum_{B\subseteq[n]}w_B\,
	\mathbf E_B[
	f(\eta K_B^U)-f(\eta)]
	\comma 
	\qquad \eta\in \Omega\comma f\in\cC(\Omega)\fstop
\end{equation}

The same random matrices induce the labeled-particle dynamics introduced in Section \ref{sec:KMP-particle-intro}. Fix $k\ge 1$. Its generator on $\mathscr H_k^\per$, the space given in \eqref{eq:coeff-space}, is
\begin{equation}\label{eq:gen-particle}
	L_k
	\eqdef
	\sum_{B\subseteq[n]}w_B
	\left(
	\Pi_{B,k}-I
	\right)
	\comma
\end{equation}
where $I$ is the identity operator on $\mathscr H_k^\per$, while $\Pi_{B,k}$ is given by
\begin{equation}\label{eq:Pi-B-k-matrix}
	\Pi_{B,k}\psi
	\eqdef
	\mathbf E_B[
	(K_B^U)^{\otimes k}\psi
	]
	\comma\qquad
	\psi\in\mathscr H_k^\per\fstop
\end{equation}
Here, $(K_B^U)^{\otimes k}$ is the tensor Markov operator corresponding, conditionally on $U$, to the independent motion of the $k$ labels according to the rows of $K_B^U$: for every $x_1,\ldots, x_k\in [n]$,
\begin{equation}
	(K_B^U)^{\otimes k}\psi(x_1,\ldots,x_k)= \textstyle\sum_{y_1,\ldots,y_k\in [n]}(K_B^U)_{x_1y_1}\cdots (K_B^U)_{x_ky_k}\psi(y_1,\ldots,y_k)\comma\qquad \psi \in \mathscr H_k^\per\fstop
\end{equation}
In words, $\Pi_{B,k}$ is the heat-bath operator associated to the measure $\mu_k$ introduced in \eqref{eq:mu-k}, resampling the positions of the particles lying in $B$, conditionally on those outside $B$. Hence, $\Pi_{B,k}$ is an orthogonal projection on $L^2(\mu_k)$, and $L_k$ is self-adjoint. Moreover, $L_k$ commutes with permutations of the labels, so $\mathscr H_k$ from \eqref{eq:coeff-space-sym} is invariant. Recall from \eqref{eq:L-k-sym} that $L_k^\sym = L_k|_{\mathscr H_k}$ denotes the corresponding restriction of $L_k$.

The ${\rm KMP}$ hidden model is generated by the same random matrices in \eqref{eq:K_B^U2}, now acting by left multiplication $\theta\longmapsto K_B^U\theta$ on column vectors $\theta\in\R^n$. Consequently, at a $B$-update with the random vector $U\sim {\rm Dir}(\alpha^B)$, all variables $\theta^B=(\theta_x)_{x\in B}$ are replaced by the common random convex combination
\begin{equation}
\theta_x\longmapsto	\sum_{y\in B}U_y\theta_y
	\comma\qquad x \in B\comma
\end{equation}
while those outside $B$ remain unchanged. Hence, its generator acts on continuous functions $g:\R^n\to \R$ as
\begin{equation}\label{eq:gen-hidden-KMP}
	\mathscr L g(\theta)
	\eqdef
	\sum_{B\subseteq[n]}w_B\,
	\mathbf E_B[
	g(K_B^U\theta)-g(\theta)
	]
	\comma
	\qquad \theta \in \R^n\fstop
\end{equation}
Since $K_B^U\mathbf1=\mathbf1$, every flat configuration $s\mathbf1$, $s\in\R$, is fixed by the hidden-model dynamics. Moreover, each coordinate of $K_B^U\theta$ is a convex combination of the coordinates of $\theta$, yielding the maximum principle described in Section \ref{sec:hidden-model-intro}. Consequently, for every $-\infty<a<b<\infty$, the restriction $\mathscr L|_{\cC([a,b]^n)}$ is well defined and bounded, and generates a Feller semigroup on $\cC([a,b]^n)$.

For the labeled-particle formulation, it is useful to introduce \textit{synchronized} $k$-copy versions of the ${\rm KMP}$ energy and hidden-model dynamics. By synchronized, we mean that all $k$ copies undergo each update simultaneously, using the same block $B$ and the same redistribution vector $U$.
 Their generators act on continuous functions, respectively, as
\begin{align}
	\cL_k^\per F(\eta^{(1)},\ldots,\eta^{(k)})
	&\eqdef
	\sum_{B\subseteq[n]}w_B\,\mathbf E_B[
	F(\eta^{(1)}K_B^U,\ldots,\eta^{(k)}K_B^U)
	-F(\eta^{(1)},\ldots,\eta^{(k)})]
	\comma\label{eq:gen-KMP-tensor}\\
	\mathscr L_k^\per G(\theta^{(1)},\ldots,\theta^{(k)})
	&\eqdef
	\sum_{B\subseteq[n]}w_B\,\mathbf E_B[
	G(K_B^U\theta^{(1)},\ldots,K_B^U\theta^{(k)})
	-G(\theta^{(1)},\ldots,\theta^{(k)})
	]\comma
	\label{eq:gen-hidden-tensor}
\end{align}
where $\eta^{(1)},\ldots,\eta^{(k)}\in\Omega$ and $\theta^{(1)},\ldots,\theta^{(k)}\in\R^n$. 
When acting on continuous bounded functions, these operators are linear, bounded, and generate a Markov semigroup.  

\subsubsection{Polynomials and intertwinings}\label{sec:polynomials}

For $k\ge 1$ and $\psi\in\mathscr H_k^\per$, define
\begin{align}
	\widehat\psi^\per(\eta^{(1)},\ldots,\eta^{(k)})
	&\eqdef
	\sum_{x_1,\ldots,x_k\in[n]}
	\psi(x_1,\ldots,x_k)\,
	\eta_{x_1}^{(1)}\cdots\eta_{x_k}^{(k)}
	\comma
	\label{eq:hat-tensor}\\
	\widetilde\psi^\per(\theta^{(1)}\ldots\theta^{(k)})
	&\eqdef
	\sum_{x_1,\ldots,x_k\in[n]}
	\mu_k(x_1,\ldots,x_k)\,
	\psi(x_1,\ldots,x_k)\,
	\theta_{x_1}^{(1)}\cdots\theta_{x_k}^{(k)}
	\comma
	\label{eq:tilde-tensor}
\end{align}
where $\eta^{(1)},\ldots,\eta^{(k)}\in\Omega$ and
$\theta^{(1)},\ldots,\theta^{(k)}\in\R^n$. We write  their diagonal restrictions as
\begin{equation}\label{eq:hat-tilde-diagonal}
	\widehat\psi(\eta)
	\eqdef
	\widehat\psi^\per(\eta,\ldots,\eta)
	\comma
	\qquad
	\widetilde\psi(\theta)
	\eqdef
	\widetilde\psi^\per(\theta,\ldots,\theta)\fstop
\end{equation}
Observe that these last two definitions agree with \eqref{eq:intro-widehat-psi} and
\eqref{eq:intro-tilde-psi}, respectively. 
\begin{remark}
	Since diagonal evaluation is invariant under permutations of the $k$ arguments, both restrictions depend only on the symmetrized part of $\psi$. More precisely, writing
	\begin{equation}
\textstyle		\psi_\sym(x_1,\ldots,x_k)
		\eqdef
		\frac1{k!}
		\sum_{\varsigma\in\mathfrak S_k}
		\psi(x_{\varsigma(1)},\ldots,x_{\varsigma(k)})\comma
	\end{equation}
	one has
	$
		\widehat\psi
		=
		\widehat{\psi_\sym}
		$ and $	\widetilde\psi
		=
		\widetilde{\psi_\sym}$.	
\end{remark}	
Remark that both \eqref{eq:hat-tensor}--\eqref{eq:tilde-tensor} admit complementary probabilistic interpretations. Namely,
	\eqref{eq:hat-tensor} is the expectation of $\psi\in\mathscr H_k^\per$ with respect to the product probability measure
	$\eta^{(1)}\otimes\cdots\otimes\eta^{(k)}$ on $[n]^k$, whereas
	\eqref{eq:tilde-tensor} is the $L^2(\mu_k)$ scalar product of $\psi$ with the tensor-product function
	$\theta^{(1)}\otimes\cdots\otimes\theta^{(k)}$. More precisely,
	\begin{align}
		\widehat\psi^\per(\eta^{(1)},\ldots,\eta^{(k)})
		&=
		(\eta^{(1)}\otimes\cdots\otimes\eta^{(k)})(\psi)
		\comma
		\label{eq:hat-representation-expectation}\\
		\widetilde\psi^\per(\theta^{(1)},\ldots,\theta^{(k)})
		&=
		\ttscalar{\psi}
		{\theta^{(1)}\otimes\cdots\otimes\theta^{(k)}}_{\mu_k}
		\fstop
		\label{eq:tilde-representation2}
	\end{align}

Recall from Section \ref{sec:KMP-intro} that $\mathscr P_k$ is the space of polynomials on $\Omega$ of degree at
most $k$, with $\mathscr P_0$ consisting of the constant functions. As hidden-model polynomial spaces, set
\begin{equation}\label{eq:R-k}
	\mathscr R_k
	\eqdef
	\{
		g:\R^n\to\R:
		g\ \text{is a homogeneous polynomial of degree}\ k
		\}
	\comma
\end{equation}
for $k\ge 1$, and let $\mathscr R_0$ be the space of constants. 
We also introduce the corresponding tensor spaces as
\begin{align}\label{eq:def-P-k-tensor}
	\mathscr P_k^\per
	&\eqdef
	\{
		F:\Omega^k\to\R:
		F\ \text{is linear in each variable separately}\}\comma
\\
\label{eq:def-R-k-tensor}
	\mathscr R_k^\per
	&\eqdef
	\{
		G:(\R^n)^k\to\R:
		G\ \text{is linear in each variable separately}
		\}\fstop
\end{align}
Thus,  $\mathscr R_k^\per$ is the space of multilinear forms on
$(\R^n)^k$, whereas $\mathscr P_k^\per$ is obtained by restricting these multilinear forms to the convex set $\Omega^k\subset (\R^n)^k$. Equivalently, letting $e_x\in\Omega$ denote the $x$-th standard basis vector of $\R^n$,	 a function $F:\Omega^k\to\R$ belongs to $\mathscr P_k^\per$ if and only if, for every $i=1,\ldots,k$,
\begin{equation}\label{eq:F-characterization}
	\textstyle
	F(\eta^{(1)},\ldots,\eta^{(k)})
	=
	\sum_{x\in[n]}\eta_x^{(i)}
	F(\eta^{(1)},\ldots,\eta^{(i-1)},e_x,\eta^{(i+1)},\ldots,\eta^{(k)})
	\fstop
\end{equation}
An analogous relation, with $\eta_x^{(i)}$ replaced by $\theta_x^{(i)}$, characterizes the functions in $\mathscr R_k^\per$.

The next result identifies the polynomial spaces introduced so far with the images of the maps in 
\eqref{eq:hat-tensor}--\eqref{eq:hat-tilde-diagonal}.

\begin{proposition}\label{pr:polynomial-identifications}
	For every $k\ge1$, the maps in
	\eqref{eq:hat-tilde-diagonal} are linear isomorphisms
	\begin{equation}\label{eq:basic-isomorphisms-symmetric}
		\mathscr H_k
		\xrightarrow[\hspace*{7mm}]{\ \psi\longmapsto\widehat\psi\ }
		\mathscr P_k
		\comma
		\qquad
		\mathscr H_k
		\xrightarrow[\hspace*{7mm}]{\ \psi\longmapsto\widetilde\psi\ }
		\mathscr R_k\fstop
	\end{equation}
	Likewise, the maps in
	\eqref{eq:hat-tensor}--\eqref{eq:tilde-tensor} are linear
	isomorphisms
	\begin{equation}\label{eq:basic-isomorphisms-tensor}
		\mathscr H_k^\per
		\xrightarrow[\hspace*{7mm}]{\ \psi\longmapsto\widehat\psi^\per\ }
		\mathscr P_k^\per
		\comma
		\qquad
		\mathscr H_k^\per
		\xrightarrow[\hspace*{14mm}]{\ \psi\longmapsto\widetilde\psi^\per\ }
		\mathscr R_k^\per\fstop
	\end{equation}
\end{proposition}

\begin{proof} First, consider $f\in \mathscr P_k$.
 Decomposing $f$ into its homogeneous components
	and using $\sum_{x\in[n]}\eta_x=1$ on $\Omega$, each component of degree
	$j\le k$ may be multiplied by
	$(\sum_x\eta_x)^{k-j}$. Hence, $f$ admits a homogeneous
	degree-$k$ representation, and therefore may be written as $\widehat\psi$ for
	some $\psi\in\mathscr H_k$. 	
	This representation is unique. Indeed, if a homogeneous polynomial $q$ of
	degree $k$ vanishes on $\Omega$, then
	$q(s\eta)=s^kq(\eta)=0$ for all $s\in \R$ and $\eta\in\Omega$, forcing $q\equiv 0$. This proves the
	first isomorphism in \eqref{eq:basic-isomorphisms-symmetric}. The second follows
	from the uniqueness of the monomial expansion of a homogeneous polynomial and
	strict positivity of  $\mu_k$.
	
	For the tensor hat representation, iterating \eqref{eq:F-characterization} over all $i=1,\ldots,k$ shows that
	$F\in\mathscr P_k^\per$ if and only if
$
		F(\eta^{(1)},\ldots,\eta^{(k)})
		=
		\sum_{x_1,\ldots,x_k\in[n]}
		F(e_{x_1},\ldots,e_{x_k})\,
		\eta_{x_1}^{(1)}\cdots\eta_{x_k}^{(k)}
$.	
	Hence, $F=\widehat\psi^\per$ with
	$\psi(x_1,\ldots,x_k)=F(e_{x_1},\ldots,e_{x_k})$, and this representation is
	unique. Finally, every multilinear form on $(\R^n)^k$ has a unique expansion
	in the monomials
	$\theta_{x_1}^{(1)}\cdots\theta_{x_k}^{(k)}$. Since $\mu_k$ is strictly
	positive, this expansion is uniquely of the form
	$\widetilde\psi^\per$. This proves
	\eqref{eq:basic-isomorphisms-tensor}.
\end{proof}
In view of Proposition \ref{pr:polynomial-identifications}, we shall freely use the identifications in
\eqref{eq:basic-isomorphisms-symmetric}--\eqref{eq:basic-isomorphisms-tensor}. Moreover, the diagonal evaluations in \eqref{eq:hat-tilde-diagonal} identify the subspaces of $\mathscr P_k^\per$ and $\mathscr R_k^\per$ invariant under permutations of the $k$ arguments with $\mathscr P_k$ and $\mathscr R_k$, respectively.

Next, we generalize and rigorously prove the intertwining relations in \eqref{eq:intro-intertwining} and \eqref{eq:intro-forward-backward-intertwining}.

\begin{proposition}[Intertwining relations]\label{pr:particle-system-intertwining}
	For every $k\ge1$ and $\psi\in\mathscr H_k^\per$,
	\begin{equation}\label{eq:tensor-intertwinings}
		\cL_k^\per\widehat\psi^\per
		=
		\widehat{L_k\psi}^\per
		\comma
		\qquad
		\mathscr L_k^\per\widetilde\psi^\per
		=
		\widetilde{L_k\psi}^\per\fstop
	\end{equation}
	Consequently, for every $\psi\in\mathscr H_k$,
	\begin{equation}\label{eq:diagonal-intertwinings}
		\cL\widehat\psi
		= \widehat{L_k\psi}=
		\widehat{L_k^\sym\psi}
		\comma
		\qquad
		\mathscr L\widetilde\psi
		=\widetilde{L_k\psi}=
		\widetilde{L_k^\sym\psi}\fstop
	\end{equation}
\end{proposition}

\begin{proof}
	Fix a block $B\subseteq[n]$. For the hat map	in \eqref{eq:hat-tensor} and \eqref{eq:hat-representation-expectation}, right multiplication
	gives, conditionally on $U$ and for all $\eta^{(1)},\ldots, \eta^{(k)}\in \Omega$,
	\begin{equation}
		\widehat\psi^\per
		(
		\eta^{(1)}K_B^U,\ldots,\eta^{(k)}K_B^U
		)
		=
		(\eta^{(1)}K_B^U\otimes \cdots \otimes \eta^{(k)}K_B^U)(\psi)= (\eta^{(1)}\otimes \cdots \otimes \eta^{(k)})((K_B^U)^{\otimes k}\psi)\fstop
	\end{equation}
	Averaging over $U$ yields the first identity in
	\eqref{eq:tensor-intertwinings}.

	For the tilde map  in \eqref{eq:tilde-tensor} and \eqref{eq:tilde-representation2}, left multiplication gives, for all $\theta^{(1)},\ldots, \theta^{(k)}\in \R^n$,
	\begin{align}
		&\mathbf E_B[
		\widetilde\psi^\per
		(
		K_B^U\theta^{(1)},\ldots,K_B^U\theta^{(k)}
		)
	]
	=
		\sum_{x_1,\ldots,x_k\in[n]}
		\mu_k(x_1,\ldots,x_k)\,
		(\Pi_{B,k}\psi)(x_1,\ldots,x_k)\,
		\theta_{x_1}^{(1)}\cdots\theta_{x_k}^{(k)}
		\comma
	\end{align}
	where we used that the heat-bath operator $\Pi_{B,k}$ given in \eqref{eq:Pi-B-k-matrix} is self-adjoint on
	$L^2(\mu_k)$. This gives the second identity in
	\eqref{eq:tensor-intertwinings}.
	
	The diagonal configurations $\eta^{(1)}=\cdots = \eta^{(k)}$ and $\theta^{(1)}=\ldots=\theta^{(k)}$ are both preserved by every synchronous update. Moreover, the
	restrictions of $\cL_k^\per$ and $\mathscr L_k^\per$ to the diagonal
	coincide with $\cL$ and $\mathscr L$, respectively. Since $L_k$ preserves
	$\mathscr H_k$, the identities in
	\eqref{eq:diagonal-intertwinings} follow.
\end{proof}

By Proposition \ref{pr:polynomial-identifications}, the intertwinings in Proposition \ref{pr:particle-system-intertwining} establish the equivalence \eqref{eq:eigenvalue-problems-three}, and induce similarity relations between finite-dimensional projections of the generators. Consequently, one obtains, counting multiplicities,
\begin{equation}\label{eq:isospectral-basic}
	\spec(\cL|_{\mathscr P_k})
	=\spec(
	L_k^\sym)=\spec(
	\mathscr L|_{\mathscr R_k})\comma
\end{equation}
\begin{equation}\label{eq:isospectral-tensor}
	\spec(\cL_k^\per|_{\mathscr P_k^\per})
	=\spec(
	L_k)
	=\spec(
	\mathscr L_k^\per|_{\mathscr R_k^\per})\fstop
\end{equation}

We conclude this part with a simple observation establishing \eqref{eq:equiv-gap>0}.
\begin{remark}\label{rem:positivity-gap}
	For fixed weights $\alpha$ and $w$, the $L_k$-particle system and the system of $k$ independent and labeled $L_1$-particles not only share the same state space $[n]^k$, but also the same communicating classes. Indeed, whenever $w_B>0$, both dynamics allow any single particle to move alone between sites of $B$ with positive rate, while not allowing any particles to leave a connected component of the $w$-hypergraph. Thus, if the latter has $m$ connected components, 
	the multiplicity of the zero eigenvalue of $L_k$ (resp.\ $L_k^\sym$) is $m^k$
	(resp.\ $\binom{k+m-1}{m-1}$).
	In particular, $L_k$ and $L_k^\sym$ have a simple zero eigenvalue if and only if $m=1$, namely, when the $w$-hypergraph is connected. By \eqref{eq:isospectral-basic}--\eqref{eq:isospectral-tensor} and \eqref{eq:gap-inf}, this proves \eqref{eq:equiv-gap>0}. 
	
	For an alternative approach to \eqref{eq:equiv-gap>0}, see \cite{alon2026aldous} and, in particular, Section 5.2 therein.
\end{remark}

\subsubsection{Decomposition of polynomial spaces} In this section, we systematically describe the subspaces $\mathscr P_{k,\ast}\subseteq \mathscr P_k$ introduced in \eqref{eq:P-k-star}, together with their particle-system and hidden-model counterparts. The resulting decompositions will be instrumental in identifying the eigenfunctions genuinely arising at degree $k$.
In view of Proposition \ref{pr:polynomial-identifications}, it is natural to start from the spaces of polynomial coefficients, thus, from the particle system.

Fix $k\ge 1$, and recall $\mathscr H_k^\per$ and $\mathscr H_k$ from \eqref{eq:coeff-space} and \eqref{eq:coeff-space-sym}, respectively.
For every $i=1,\ldots, k$, define the \textit{particle-removal} and \textit{particle-addition operators} acting on $\phi \in \mathscr H_{k-1}^\per$ and $\psi \in \mathscr H_k^\per$, respectively, as
\begin{equation}\label{eq:particle-removal}
	\mathfrak a_{k,i}\phi(x_1,\ldots,x_k)
	\eqdef
	\phi(x_1,\ldots,x_{i-1},x_{i+1},\ldots,x_k)
	\comma\qquad x_1,\ldots,x_k \in [n]\comma
	\end{equation}
	\begin{equation}\label{eq:particle-addition}
	\mathfrak b_{k-1,i}\psi(x_1,\ldots,x_{k-1})
	\eqdef
	\sum_{y\in[n]}
	\mu_k(y\mid x_1,\ldots,x_{k-1})\,
	\psi(x_1,\ldots,x_{i-1},y,x_i,\ldots,x_{k-1})
	\comma
\end{equation}
where $
	\mu_k(y\mid x_1,\ldots,x_{k-1})
	\eqdef
	\frac{\mu_k(x_1,\ldots,x_{k-1},y)}
	{\mu_{k-1}(x_1,\ldots,x_{k-1})}
	$; see \eqref{eq:mu-k}.
Thus,
\begin{equation}
	\mathfrak a_{k,i}:\mathscr H_{k-1}^\per\to\mathscr H_k^\per
	\comma\qquad
	\mathfrak b_{k-1,i}:\mathscr H_k^\per\to\mathscr H_{k-1}^\per\comma
\end{equation}
and, as their names suggest, both $\mathfrak a_{k,i}$ and $\mathfrak b_{k-1,i}$ are stochastic operators between particle configuration spaces, with a transparent probabilistic interpretation: the former induces a deterministic removal of the $i$-th particle, whereas the latter adds a particle whose position is averaged over its conditional equilibrium distribution given the positions of all remaining particles.

A simple computation shows that these operators are adjoint of one another, namely,
\begin{equation}\label{eq:particle-removal-addition-adjoint}
	\scalar{\mathfrak a_{k,i}\phi}{\psi}_{\mu_k}
	=
	\scalar{\phi}{\mathfrak b_{k-1,i}\psi}_{\mu_{k-1}}
	\comma\qquad
	\phi\in\mathscr H_{k-1}^\per
	\comma
	\psi\in\mathscr H_k^\per\fstop
\end{equation}
As $\mathfrak a_{k,i}$ is clearly injective, this ensures that $\mathfrak b_{k-1,i}$ is surjective.

The tensor form of the matrix update in 	\eqref{eq:Pi-B-k-matrix} gives
\begin{equation}
	(K_B^U)^{\otimes k}\mathfrak a_{k,i}
	=
	\mathfrak a_{k,i}(K_B^U)^{\otimes(k-1)}
	\fstop
\end{equation}
Combining this with \eqref{eq:gen-particle} and \eqref{eq:particle-removal-addition-adjoint} yields, respectively,
\begin{equation}\label{eq:particle-consistency}
	L_k\mathfrak a_{k,i}
	=
	\mathfrak a_{k,i}L_{k-1}
	\comma\qquad
	\mathfrak b_{k-1,i}L_k
	=
	L_{k-1}\mathfrak b_{k-1,i}\fstop
\end{equation}
The first relation in \eqref{eq:particle-consistency} is sometimes referred to as the \textit{consistency} of the particle system: evolving $k$ particles and then removing the $i$-th one yields the same law as first removing that particle and then evolving the remaining $k-1$ particles.

Next, we introduce the \textit{symmetrized} operators 
\begin{equation}
\label{eq:particle-removal-addition-sym}	\textstyle
	\mathfrak a_k
	\eqdef
	\frac1k\sum_{i=1}^k\mathfrak a_{k,i}
	\comma\qquad
	\mathfrak b_{k-1}
	\eqdef
	\frac1k\sum_{i=1}^k\mathfrak b_{k-1,i}\comma
\end{equation}
which admit the analogous interpretation of removing or adding a particle, whose label is now chosen uniformly at random. Moreover, one has
\begin{equation}\label{eq:particle-removal-addition-relations}
	\mathfrak a_k(\mathscr H_{k-1})\subseteq\mathscr H_k\comma\qquad \mathfrak b_{k-1}(\mathscr H_k)=\mathscr H_{k-1}\comma
\end{equation}
and it is not difficult to check that the restriction $\mathfrak a_k|_{\mathscr H_{k-1}}$ is injective; see \cite[Lemma~A.2]{kim_sau_spectral_2023}. Finally, by linearity,  $\mathfrak a_k$ and $\mathfrak b_{k-1}$  satisfy commutation relations analogous to those in \eqref{eq:particle-consistency}.

 The first identity in \eqref{eq:particle-consistency}, its symmetrized counterpart, and the aforementioned injectivity of the particle-removal operators show that lower-level eigenfunctions may be canonically lifted. More precisely, if $0\neq \phi\in \mathscr H_{k-1}$ is an eigenfunction of $L_{k-1}^\sym$ with eigenvalue $\lambda$, then $0\neq\mathfrak a_k\phi\in \mathscr H_k$ is an eigenfunction of $L_k^\sym$ with the same eigenvalue; likewise, $0\neq\mathfrak a_{k,i}\phi\in \mathscr H_k^\per$, $i=1,\ldots, k$, is an eigenfunction of $L_k$ whenever $0\neq\phi\in \mathscr H_{k-1}^\per$ is an eigenfunction of $L_{k-1}$. Injectivity of $\mathfrak a_k$ and $\mathfrak a_{k,i}$  ensures that these liftings are nontrivial.	
 Thus, the images of the removal operators contain no genuinely new eigenpairs. Since $L_k^\sym$ and $L_k$ are self-adjoint, the orthogonal complement of these images are invariant and indicate the natural spaces in which to search for new eigenfunctions. Letting  $\mathscr H_{k,\ast}\subseteq \mathscr H_k$ and $\mathscr H_{k,\ast}^\per\subseteq \mathscr H_k^\per$ denote such complements, one obtains the following orthogonal decompositions
 \begin{equation}\label{eq:decomp-particle}
 	\textstyle	
 	\mathscr H_k
 	=
 	\mathfrak a_k(\mathscr H_{k-1})
 	\oplus_{\mu_k}
 	\mathscr H_{k,\ast}\comma\qquad \mathscr H_k^\per
 	=
 	\ttonde{
 		\mathfrak a_{k,1}(\mathscr H_{k-1}^\per)+\cdots + \mathfrak a_{k,k}(\mathscr H_{k-1}^\per)
 	}
 	\oplus_{\mu_k}
 	\mathscr H_{k,\ast}^\per
 	\fstop
 \end{equation}
 Here, the symbol $\oplus_{\mu_k}$ stands for a direct sum of $\mu_k$-orthogonal subspaces, whereas $+$ indicates a (not necessarily direct) sum of the subspaces $\mathfrak a_{k,i}(\mathscr H_{k-1}^\per)$, $i=1,\ldots, k$. 
 
 In view of the above decompositions and of the adjointness relation \eqref{eq:particle-removal-addition-adjoint}, we find that such orthogonal complements may be expressed in terms of kernels of the addition operators.
 \begin{proposition}\label{pr:Hk-star}
 The orthogonal complements in \eqref{eq:decomp-particle} are given by
 	\begin{equation}\label{eq:H-k-star-symmetric}
 		\textstyle
 		\mathscr H_{k,\ast}
 		=
 		\ker(\mathfrak b_{k-1}|_{\mathscr H_k})\subseteq \mathscr H_k\comma\qquad
 			\mathscr H_{k,\ast}^\per
 		=
 		\bigcap_{i=1}^k
 		\ker(\mathfrak b_{k-1,i})\subseteq \mathscr H_k^\per
 \fstop
 	\end{equation}	
 \end{proposition}
 
 The spaces $\mathscr H_{k,\ast}$ and $\mathscr H_{k,\ast}^\per$ consist precisely of the eigenfunctions that do not arise from $\mathscr H_{k-1}$ and $\mathscr H_{k-1}^\per$. By Proposition \ref{pr:polynomial-identifications}, the restrictions of the hat and tilde maps to these spaces are linear isomorphisms onto their respective images. The intertwining relations in Proposition \ref{pr:particle-system-intertwining} then imply that these images are invariant under the corresponding generators, and that the restricted maps identify their eigenfunctions bijectively while preserving the associated eigenvalues.
 Hence, our final task is to identify these images in the polynomial spaces of the ${\rm KMP}$ and its hidden model.
 
 We start by recovering the $\eta$-polynomials in \eqref{eq:P-k-star}.
 \begin{proposition}\label{pr:Pk-star} The diagonal-restriction hat map in \eqref{eq:hat-tilde-diagonal} is a linear isomorphism between  $\mathscr H_{k,\ast}$ and the polynomials $\mathscr P_{k,\ast}$ defined in \eqref{eq:P-k-star}. Consequently, $\cL\mathscr P_{k,\ast}\subseteq \mathscr P_{k,\ast}$.
 \end{proposition}
 \begin{proof}Recall from \eqref{eq:decomp-particle} and \eqref{eq:P-k-star} that $\mathscr H_k=\mathfrak a_k(\mathscr H_{k-1})\oplus_{\mu_k}\mathscr H_{k,\ast}$ and $\mathscr P_k=\mathscr P_{k-1}\oplus_\mu \mathscr P_{k,\ast}$.	Since the hat map in \eqref{eq:hat-tilde-diagonal} is a linear isomorphism for both $\mathscr H_k\to \mathscr P_k$ and $\mathfrak a_k(\mathscr H_{k-1})\to \mathscr P_{k-1}$, it suffices to prove that any $\psi \in \mathscr H_{k,\ast}$ satisfies $\widehat\psi \in \mathscr P_{k,\ast}$. 
 	
Recall from \eqref{eq:mu-k} that, for every $j\ge 1$ and $z_1,\ldots, z_j\in [n]$, 
$
	\mu(\eta_{z_1}\cdots \eta_{z_j})=\mu_j(z_1,\ldots, z_j)
$, and, for later reference, we record the occupation-number notation already used there: 
\begin{equation}\label{eq:N-x-bd}
	\textstyle
	\mathfrak n_x(z_1,\ldots,z_j)
	\eqdef
	\sum_{i=1}^j\car_{\{z_i=x\}}\comma\qquad x \in [n]\fstop
\end{equation}
Hence, for every $\psi \in \mathscr H_k$, $1\le \ell\le k-1$, and $y_1,\ldots, y_\ell\in [n]$, 
\begin{equation}\label{eq:inner-product}
	\mu(\widehat \psi\,\eta_{y_1}\cdots\eta_{y_\ell})
	= \sum_{x_1,\ldots, x_k\in [n]}\psi(x_1,\ldots, x_k)\, \mu(\eta_{x_1}\cdots\eta_{x_k}\eta_{y_1}\cdots \eta_{y_\ell})
	=\scalar{\psi}{\phi_{y_1,\ldots, y_\ell}}_{\mu_k}\comma
\end{equation}
where $\phi_{y_1,\ldots, y_\ell}\in \mathscr H_k$ is given, for all $x_1,\ldots, x_k \in [n]$, by 
\begin{align}
	&\phi_{y_1,\ldots,y_\ell}(x_1,\ldots, x_k)\eqdef \frac{\mu_{k+\ell}(x_1,\ldots, x_k,y_1,\ldots, y_\ell)}{\mu_k(x_1,\ldots, x_k)}\\
	&\qquad= \frac{\Gamma(\alpha_0+k)}{\Gamma(\alpha_0+k+\ell)} \prod_{x\in [n]}\frac{\Gamma(\alpha_x+\mathfrak n_x(x_1,\ldots, x_k)+\mathfrak n_x(y_1,\ldots, y_\ell))}{\Gamma(\alpha_x+\mathfrak n_x(x_1,\ldots, x_k))}\comma
	\label{eq:ratio-phi}
\end{align} 
where the last step used \eqref{eq:mu-k} and $\mathfrak n_x(x_1,\ldots, x_k,y_1,\ldots, y_\ell)=\mathfrak n_x(x_1,\ldots, x_k)+\mathfrak n_x(y_1,\ldots, y_\ell)$. By the property of the gamma function $\Gamma(a+1)=a\Gamma(a)$, $a>0$, each ratio in the product in \eqref{eq:ratio-phi} is a polynomial of degree at most $\mathfrak n_x(y_1,\ldots, y_\ell)$ in the occupation numbers $\mathfrak n_x(x_1,\ldots, x_k)$. Since $\sum_{x\in [n]}\mathfrak n_x(y_1,\ldots, y_\ell)=\ell\le k-1$, 
this ensures that $\phi_{y_1,\ldots, y_\ell} \in \mathfrak a_k(\mathscr H_{k-1})$. Taking $\psi \in \mathscr H_{k,\ast}$ shows that the expression in \eqref{eq:inner-product} vanishes. Since  monomials of the type $\eta_{y_1}\cdots \eta_{y_\ell}$, $\ell \le k-1$, span $\mathscr P_{k-1}$, this proves $\widehat \psi \in \mathscr P_{k,\ast}$ as desired.
 \end{proof}
Although an analogous characterization is available for $\mathscr P_{k,\ast}^\per$ $\eqdef$  the image of $\mathscr H_{k,\ast}^{\per}$ under \eqref{eq:hat-tensor}, it will not be needed here and we omit it. We turn instead to the images under the tilde maps.
 \begin{proposition}\label{pr:Rk-star}
 	Let  $\mathscr R_{k,\ast}\subseteq \mathscr R_k$ and $\mathscr R_{k,\ast}^\per\subseteq\mathscr R_k^\per$ be the images of $\mathscr H_{k,\ast}$ and  $\mathscr H_{k,\ast}^\per$ under the tilde isomorphisms in  \eqref{eq:hat-tilde-diagonal} and \eqref{eq:tilde-tensor}, respectively. Then, 
 	\begin{equation}\label{eq:R-k-ast-inv}
 		\mathscr L\mathscr R_{k,\ast}\subseteq \mathscr R_{k,\ast}\comma\qquad \mathscr L_k^\per\mathscr R_{k,\ast}^\per\subseteq \mathscr R_{k,\ast}^\per\comma
 	\end{equation}
 	and
 	\begin{equation}\label{eq:R-k-ast}
 		\mathscr R_{k,\ast}=\tset{g\in \mathscr R_k\mid s\longmapsto g(\theta+s\mathbf1)\ \text{is constant on $\R$, for all}\ \theta \in \R^n}\comma
 	\end{equation}	
 	\begin{equation}\label{eq:R-k-ast-per}
 		\mathscr R_{k,\ast}^\per=\set{G\in \mathscr R_k^\per\middle| 
 			\begin{array}{c}
 			s\longmapsto G(\theta^{(1)},\ldots, 
 			\theta^{(i)}+s\mathbf1,
 			\ldots, \theta^{(k)})\
 			\text{is constant on $\R$}\comma\\[.1cm] \text{for all}\ i=1,\ldots, k\ 
 			\text{and}\ \theta^{(1)},\ldots, \theta^{(k)}\in \R^n
 			\end{array}
 		}\fstop
 	\end{equation}
 \end{proposition}
 \begin{proof}The invariance in \eqref{eq:R-k-ast-inv} follows at once from the intertwinings in Proposition \ref{pr:particle-system-intertwining}. 
 	
 	Next, focus on \eqref{eq:R-k-ast-per}.
 		Let $\psi\in\mathscr H_k^\per$ and $G=\widetilde\psi^\per\in \mathscr R_k^\per$. By multilinearity,
 	for $i=1,\ldots,k$,
 	\begin{align}
 		&\frac{\dd}{\dd s}G(\theta^{(1)},\ldots, \theta^{(i)}+s\mathbf1,\ldots, \theta^{(k)})\bigg|_{s=0}\\
 		&\qquad= \sum_{\substack{x_1,\ldots, x_{i-1},\\
 				x_{i+1},\ldots,x_k\in [n]}}\bigg(\sum_{x_i\in [n]}\mu_k(x_1,\ldots, x_k)\,\psi(x_1,\ldots, x_k)\bigg)\,\theta_{x_1}^{(1)}\cdots \theta_{x_{i-1}}^{(i-1)}\theta_{x_{i+1}}^{(i+1)}\cdots \theta_{x_k}^{(k)}\\
 		&\qquad= \widetilde{\mathfrak b_{k-1,i}\psi}^{\!\!\per}\!\!(\theta^{(1)},\ldots, \theta^{(i-1)},\theta^{(i+1)},\ldots, \theta^{(k)})\comma
 		\label{eq:tilde-shift-contraction}
 	\end{align}
 	where for the last step we used the symmetry of $\mu_k$, 
 	\begin{equation}
 	\mu_k(x_1,\ldots,x_k)
 	=
 	\mu_{k-1}(x_1,\ldots,x_{i-1},x_{i+1},\ldots, x_k)\,
 	\mu_k(x_i\mid x_1,\ldots,x_{i-1},x_{i-1},\ldots, x_k)
 	\comma
 	\end{equation}
 	and the definitions of $\mathfrak b_{k-1,i}$ and of the tilde map in \eqref{eq:particle-addition} and \eqref{eq:tilde-tensor}, respectively.	
  	By Proposition \ref{pr:polynomial-identifications}, the right-hand side of
 	\eqref{eq:tilde-shift-contraction} vanishes identically if and only if
 	$\mathfrak b_{k-1,i}\psi=0$. Intersecting over $i$ and using the second identity
 	in \eqref{eq:H-k-star-symmetric} proves \eqref{eq:R-k-ast-per}.
 	
 	The argument for \eqref{eq:R-k-ast} is analogous.
 	If $\psi\in\mathscr H_k$ and $g=\widetilde\psi\in \mathscr R_k$, differentiating gives
 	\begin{equation}
 	\frac{\dd}{\dd s}g(\theta+s\mathbf1)\bigg|_{s=0}
 		=
 		k\,\widetilde{\mathfrak b_{k-1}\psi}(\theta)
 		\fstop
 	\end{equation}
 	Therefore, $s\longmapsto g(\theta+s\mathbf1)$ is constant for every $\theta$ if
 	and only if $\mathfrak b_{k-1}\psi=0$. The first identity in
 	\eqref{eq:H-k-star-symmetric} now yields \eqref{eq:R-k-ast}.
 \end{proof}
 
 \subsection{Proof of Theorems \ref{th:main-KMP} and \ref{th:disc-ell-KMP}}\label{sec:proof-main-KMP}
 As in the rest of the section, the weights $\alpha$ and $w$ are fixed throughout.
 By the previous intertwinings, the corresponding restricted generators are similar. Hence, counting algebraic multiplicities,
  \begin{align}\label{eq:restricted-isospectrality-sym}
 		\spec(-\cL|_{\mathscr P_{k,\ast}})
 		&=
 		\spec(-L_k|_{\mathscr H_{k,\ast}})
 		=
 		\spec(-\mathscr L|_{\mathscr R_{k,\ast}})
 		\comma\\
 		\label{eq:restricted-isospectrality-per}
 		\spec(-\cL_k^\per|_{\mathscr P_{k,\ast}^\per})
 		&=
 		\spec(-L_k|_{\mathscr H_{k,\ast}^\per})
 		=
 		\spec(-\mathscr L_k^\per|_{\mathscr R_{k,\ast}^\per})
 		\fstop
 \end{align}
 Since $L_k$ is self-adjoint and the spaces $\mathscr H_{k,\ast}$ and $\mathscr H_{k,\ast}^\per$ are invariant, these spectra are real. Moreover, $\gap_{k,\ast}(\alpha,w)$ from \eqref{eq:gap-inf-star-intro} is the smallest eigenvalue in \eqref{eq:restricted-isospectrality-sym}; thus, 
 \begin{equation}
 \gap_{k,\ast}(\alpha,w)=\gap(L_k|_{\mathscr H_{k,\ast}})\comma\qquad k\ge 1\comma
 \end{equation}
 and we analogously define 
 \begin{equation}\label{eq:gap-k-ast-per}
 \gap_{k,\ast}^\per(\alpha,w)\eqdef \gap(L_k|_{\mathscr H_{k,\ast}^\per})\comma\qquad k\ge 1\fstop
 \end{equation}
 
 With this notation, both Theorems \ref{th:main-KMP} and \ref{th:disc-ell-KMP} follow from the lower bound
 \begin{equation}\label{eq:target-gap-star}
 	\gap_{k,\ast}^{\per}(\alpha,w)
 	\ge
 	\gap_{2,\ast}(\alpha,w)
 	\comma
 	\qquad k\ge2\fstop
 \end{equation}
 Indeed, since $\mathscr H_{k,\ast}\subseteq\mathscr H_{k,\ast}^{\per}$,
 \begin{equation}\label{eq:gap-gap-per}
 	\gap_{k,\ast}^{\per}(\alpha,w)
 	\le
 	\gap_{k,\ast}(\alpha,w)
 \end{equation}
 Moreover, iterating the decompositions in \eqref{eq:decomp-particle} yields
 \begin{equation}\label{eq:gap-decomp-particle}
 	\gap(L_k^{\sym}) = \gap(L_k|_{\mathscr H_k})
 	=
 	\min_{1\le j\le k}\gap_{j,\ast}(\alpha,w)
 	\comma
 	\qquad
 	\gap(L_k)
 	=
 	\min_{1\le j\le k}\gap_{j,\ast}^{\per}(\alpha,w)
 	\fstop
 \end{equation}
 Together with
$ 	\gap_{1,\ast}^{\per}(\alpha,w)
 	=
 	\gap_{1,\ast}(\alpha,w)
 	=
 	\gap_1(\alpha,w)
 	$, 
 these identities show that \eqref{eq:target-gap-star} implies precisely the claims of Theorems \ref{th:main-KMP} and \ref{th:disc-ell-KMP}.
 
 By \eqref{eq:restricted-isospectrality-sym}--\eqref{eq:restricted-isospectrality-per}, it suffices to prove \eqref{eq:target-gap-star} for the hidden model. This is the convenient choice: flat configurations are fixed by the hidden-model dynamics, while departure from flatness is measured by the variance $\var_\pi$, with $\pi=\alpha/\alpha_0$ and $\var_\pi$ defined in \eqref{eq:var-pi-intro}. 
 
 The proof of \eqref{eq:target-gap-star} rests on the characterizations of $\mathscr R_{k,\ast}$ and $\mathscr R_{k,\ast}^{\per}$ in Proposition \ref{pr:Rk-star}, and two elementary observations, Lemmas \ref{lem:var-power} and \ref{lem:G-Var} below.

\begin{lemma}\label{lem:var-power}
	For every even integer $k\ge 2$, $\var_\pi^{k/2}\in\mathscr R_{k,\ast}$.
	Moreover, $\var_\pi^{k/2}$ vanishes precisely on flat configurations and, thus, is not identically zero.
\end{lemma}

\begin{proof}
	The function $\var_\pi$ from \eqref{eq:var-pi-intro} is a homogeneous polynomial of degree two and satisfies
	\begin{equation}
		\var_\pi(\theta+s\mathbf1)
		=
		\var_\pi(\theta)
		\comma\qquad  \theta\in\R^n\comma
		s\in\R\fstop
	\end{equation}
	Hence, $\var_\pi^{k/2}$ is homogeneous of degree $k$ and is invariant under translations along $\mathbf1$. By \eqref{eq:R-k-ast}, this proves that $\var_\pi^{k/2}\in\mathscr R_{k,\ast}$.
The second claim follows as 
 all weights $\pi_x=\alpha_x/\alpha_0$ are strictly positive.
\end{proof}

The next result is the multilinear counterpart of \eqref{eq:g-var}; see also \cite[Eq.\ (4.1)]{kim_quattropani_sau_spectral_2025}.

\begin{lemma}\label{lem:G-Var}
	Fix $k\ge 1$ and $G\in\mathscr R_{k,\ast}^\per$. Then, for some constant $C=C(G,k,\pi)>0$,
	\begin{equation}\label{eq:G-Var}
	\textstyle	|G(\theta^{(1)},\ldots,\theta^{(k)})|
		\le
		C
		\var_\pi^{1/2}(\theta^{(1)})\cdots \var_\pi^{1/2}(\theta^{(k)})
		\comma\qquad
		\theta^{(1)},\ldots,\theta^{(k)}\in\R^n\fstop
	\end{equation}
\end{lemma}

\begin{proof}
	By \eqref{eq:R-k-ast-per}, $G$ is invariant under the addition of a constant vector to any one of its arguments. Hence (recall  $\pi(\theta)\eqdef\sum_{x\in [n]}\pi_x\theta_x\in \R$),	
	\begin{equation}
		G(\theta^{(1)},\ldots,\theta^{(k)})
		=
		G\big(
		\theta^{(1)}-\pi(\theta^{(1)})\mathbf1,
		\ldots,
		\theta^{(k)}-\pi(\theta^{(k)})\mathbf1
		\big)\fstop
	\end{equation}
Since $\pi_x>0$ for all $x\in [n]$, and $\var_\pi^{1/2}(\theta)$ is the $\pi$-weighted $\ell_2$-norm of $\theta-\pi(\theta)$ on $\R^n$,  the bound in \eqref{eq:G-Var} follows from the multilinearity of $G$.
\end{proof}

We now have all the ingredients to prove Aldous' phenomenon in ${\rm KMP}$.

\begin{proof}[Proof of Theorems \ref{th:main-KMP} and \ref{th:disc-ell-KMP}]
	Fix $k\ge 2$, and let $G\in\mathscr R_{k,\ast}^\per$ be a nonzero eigenfunction of $-\mathscr L_k^\per$ given in \eqref{eq:gen-hidden-tensor} associated with an eigenvalue $\lambda$. As discussed at the beginning of the section, $\lambda$ is real and satisfies $\lambda \ge  \gap_{k,\ast}^\per(\alpha,w)$ as defined in \eqref{eq:gap-k-ast-per}.
	
	Let $(e^{t\mathscr L_k^\per})_{t\ge0}$ denote the synchronized tensor-hidden-model semigroup on $\mathscr R_k^\per$. Throughout the proof, we repeatedly use its Markovianity, which implies
	\begin{equation}\textstyle
		G_1\le G_2
		\qquad\Longrightarrow\qquad
		e^{t\mathscr L_k^\per}G_1
		\le
		e^{t\mathscr L_k^\per}G_2
		\comma t\ge0\fstop
	\end{equation}
	In particular, for every $t>0$ and $\theta^{(1)},\ldots,\theta^{(k)}\in\R^n$, Lemma \ref{lem:G-Var} gives 
	\begin{align}\label{eq:eigenfunction-variance-bound}
		\begin{aligned}
		e^{-\lambda t}
		|G(\theta^{(1)},\ldots,\theta^{(k)})|
		&=
		|
		e^{t\mathscr L_k^\per}G
		(\theta^{(1)},\ldots,\theta^{(k)})
		|\\
		&\le
		e^{t\mathscr L_k^\per}|G|
		(\theta^{(1)},\ldots,\theta^{(k)})\\
		&\le
		C
		e^{t\mathscr L_k^\per}
		(\var_\pi^{1/2})^{\otimes k}
		(\theta^{(1)},\ldots,\theta^{(k)})\fstop
		\end{aligned}
	\end{align}
	The constant $C=C(G,k,\pi)>0$ does not depend on $t$.
	
	Let now $2\le\ell\le k$ be even, fix $-\infty<a<b<\infty$, and suppose that $\theta^{(1)},\ldots,\theta^{(k)}\in[a,b]^n$. The maximum principle for the hidden-model dynamics guarantees that every evolved profile remains in $[a,b]^n$ almost surely. Since
	\begin{equation}\label{eq:var-a-b}
		\var_\pi^{1/2}(\theta)\le b-a
		\comma\qquad \theta\in[a,b]^n\comma
	\end{equation}
	and the first $\ell$ coordinates of the synchronized $k$-tensor dynamics evolve according to the synchronized $\ell$-tensor dynamics, we obtain
	\begin{align}	\label{eq:tensor-holder}
		\begin{aligned}
		e^{t\mathscr L_k^\per}
		(\var_\pi^{1/2})^{\otimes k}
		(\theta^{(1)},\ldots,\theta^{(k)})
		&\le
		(b-a)^{k-\ell}
		e^{t\mathscr L_\ell^\per}
		(\var_\pi^{1/2})^{\otimes\ell}
		(\theta^{(1)},\ldots,\theta^{(\ell)})\\
		&
		\textstyle\le
		(b-a)^{k-\ell}
		\prod_{i=1}^{\ell}
		(
		e^{t\mathscr L}
		\var_\pi^{\ell/2}(\theta^{(i)})
		)^{1/\ell}\comma
	\end{aligned}
	\end{align}
	where the last inequality follows from H\"older's inequality. 
		
	For every $g:\R^n\to\R$, set
	$
		\ttnorm{g}_{a,b}
		\eqdef
		\sup_{\theta\in[a,b]^n}|g(\theta)|$.
	Thus, \eqref{eq:tensor-holder} yields
	\begin{equation}\label{eq:tensor-semigroup-norm}
		e^{t\mathscr L_k^\per}
		(\var_\pi^{1/2})^{\otimes k}
		(\theta^{(1)},\ldots,\theta^{(k)})
		\le
		(b-a)^{k-\ell}
		\ttnorm{
			e^{t\mathscr L}\var_\pi^{\ell/2}
		}_{a,b}\fstop
	\end{equation}
		By Lemma \ref{lem:var-power},
	$\var_\pi^{\ell/2}\in\mathscr R_{\ell,\ast}$.
	Moreover, $\ttnorm{\cdot}_{a,b}$ is a norm on the finite-dimensional space $\mathscr R_{\ell,\ast}$, and this space is invariant under $\mathscr L$ by \eqref{eq:R-k-ast-inv}.	
	Since $G\neq0$, we may choose $\theta^{(1)},\ldots,\theta^{(k)}\in[a,b]^n$ such that
	$
	G(\theta^{(1)},\ldots,\theta^{(k)})\neq0$.
	Altogether, combining \eqref{eq:eigenfunction-variance-bound}, \eqref{eq:tensor-holder} and \eqref{eq:tensor-semigroup-norm} yields
	\begin{equation}\label{eq:key-bound}
		e^{-\lambda t}|G(\theta^{(1)},\ldots,\theta^{(k)})|\le C(b-a)^{k-\ell} \ttnorm{e^{t\mathscr L}\var_\pi^{\ell/2}}_{a,b}\le C(b-a)^k\ttnorm{e^{t\mathscr L|_{\mathscr R_{\ell,\ast}}}}_{a,b}\fstop
	\end{equation}	
 Hence, taking logarithms in \eqref{eq:key-bound}, dividing by $t$, and letting $t\to\infty$ yield
		\begin{equation}
			-\lambda
			\le
			\lim_{t\to\infty}
			\frac1t
			\log
			\ttnorm{e^{t\mathscr L|_{\mathscr R_{\ell,\ast}}}}_{a,b}
			=
			-\gap_{\ell,\ast}(\alpha,w)
			\fstop
		\end{equation}
		Indeed, by Gelfand's formula (see, e.g., \cite[Corollary 5.6.14]{horn_matrix_2012} or, for a continuous-time statement, \cite[Exercise 1.3]{salez2025modern}), the limit on the right exists and equals the spectral bound $\max\{{\rm Re}(\sigma): \sigma \in \spec(\mathscr L|_{\mathscr R_{\ell,\ast})}\}$. The latter is $-\gap_{\ell,\ast}(\alpha,w)$, as the spectrum is real and $\gap_{\ell,\ast}(\alpha,w)$ is the smallest eigenvalue of $-\mathscr L|_{\mathscr R_{\ell,\ast}}$.
		
		As $\lambda \ge \gap_{k,\ast}^\per(\alpha,w)$, this proves, for every $2\le \ell \le k$ with $\ell\ge 2$ even,
		\begin{equation}\label{eq:lambda-gap-ell}
			\gap_{\ell,\ast}(\alpha,w)\le \gap_{k,\ast}^\per(\alpha,w)\fstop
		\end{equation} 
		Choosing $\ell=2$ proves \eqref{eq:target-gap-star} and, thus, concludes the proof.	
	\end{proof}

\begin{remark}\label{rem:new-ordering}
	The estimate in \eqref{eq:lambda-gap-ell} yields two nontrivial refinements of Theorems \ref{th:main-KMP} and \ref{th:disc-ell-KMP}. First, by choosing $\ell=k-1$ when $k$ is odd and $\ell=k-2$ when $k$ is even, one obtains
	\begin{equation}\label{eq:gap-even-odd}
		\gap_{k,\ast}(\alpha,w)
		\ge
		\begin{dcases}
			\gap_{k-1,\ast}(\alpha,w)
			&\text{if $k\ge3$ is odd}\comma\\
			\gap_{k-2,\ast}(\alpha,w)
			&\text{if $k\ge4$ is even}\fstop
		\end{dcases}
	\end{equation}
	Second, if $k\ge2$ is even, choosing $\ell=k$ in \eqref{eq:lambda-gap-ell} together with \eqref{eq:gap-gap-per} gives 
	\begin{equation}
		\gap_{k,\ast}^\per(\alpha,w)
		=
		\gap_{k,\ast}(\alpha,w)\fstop
	\end{equation}

	In particular, iterating \eqref{eq:gap-even-odd} over the even indices $k\ge2$ proves, for all weights $\alpha$ and $w$,
	\begin{equation}\label{eq:even-chain}
		\gap_{2,\ast}(\alpha,w)\le \gap_{4,\ast}(\alpha,w)\le \gap_{6,\ast}(\alpha,w)\le \ldots\comma
	\end{equation}
	which, in the special case $\alpha\equiv1$, is precisely the second claim of \cite[Conjecture~4.16]{alon2026aldous}. In the same conjecture, Alon and Puder also propose the  odd-index analogue
	\begin{equation}\label{eq:odd-chain}
		\gap_1(\alpha,w)\le \gap_{3,\ast}(\alpha,w)\le \gap_{5,\ast}(\alpha,w)\le \ldots\comma
	\end{equation}for $\alpha \equiv 1$ and arbitrary block weights $w$. While our results neither prove or disprove their claim, it is not difficult to see that \eqref{eq:odd-chain} cannot hold for general weights. Indeed, consider the homogeneous $n$-site complete graph with
	\begin{equation}
		\alpha \equiv {\rm const.}= \tau >0\comma\qquad w_B=\car_{|B|=2}\comma
	\end{equation}
	for which $\gap_1(\alpha,w)=n/2$.
	A direct computation of the action of $\cL$ on the cubic polynomial $\eta \mapsto \sum_{x\in [n]}\eta_x^3 $ 	gives
$
		 \gap_{3,\ast}(\alpha,w)\le \frac{3}{2}\frac{(\tau n+2)}{(2\tau+1)}$. 
	Hence, for every $n\ge 7$ and $0<\tau<1-6/n$, one finds $\gap_{3,\ast}(\alpha,w)<\gap_1(\alpha,w)$.
\end{remark}

\subsection{Gibbs sampler for cone measures}\label{sec:Gibbs-sampler}
Here we prove Theorem \ref{th:GS-p}. 
Throughout, we set
$\alpha_x=1/p_x$, $x\in[n]$, and $\mu={\rm Dir}(\alpha)$, with $\kappa$ denoting the corresponding cone measure on $\mathbb S_p$.
In view of \eqref{eq:dirp}, we identify $L^2(\kappa)$ with the space of square-integrable functions of $(\eta,\sigma)$, where $\eta\sim\mu$, while $\sigma$ is independent of $\eta$ and has i.i.d.\ coordinates uniform on $\{\pm1\}$. For $J\subseteq[n]$, set
\begin{equation}\textstyle
	\sigma_J\eqdef\prod_{x\in J}\sigma_x\comma
\end{equation}
with $\sigma_\emp=1$. Every $h\in L^2(\kappa)$ admits the orthogonal expansion
\begin{equation}\label{eq:Walsh-decomposition-Gibbs}
	\textstyle	h(\eta,\sigma)
	=
	\sum_{J\subseteq[n]} h_J(\eta)\,\sigma_J
	\comma
	\qquad h_J\in L^2(\mu)\fstop
\end{equation}

\begin{lemma}\label{lem:Walsh-sectors-Gibbs}
	For every $B,J\subseteq[n]$ and $h_J\in L^2(\mu)$,
	\begin{equation}\label{eq:block-action-Walsh-Gibbs}
		\kappa_{B}(h_J\sigma_J)
		=
		\begin{cases}
			(\mu_Bh_J)\,\sigma_J&\text{if $B\cap J=\emp$}\\
			0&\text{if $B\cap J\neq\emp$}\fstop
		\end{cases}
	\end{equation}
	Consequently, {each subspace $L^2(\mu)\otimes \spanop(\sigma_J)$} is invariant under $\cG$, and
	\begin{align}\label{eq:Dirichlet-Walsh-Gibbs}
		\ttscalar{h_J\sigma_J}{-\cG(h_J\sigma_J)}_{\kappa}
		=
		\sum_{B:B\cap J=\emp}
		w_B\,\ttnorm{h_J-\mu_Bh_J}_\mu^2
		+
		\sum_{B:B\cap J\neq\emp}
		w_B\,\ttnorm{h_J}_\mu^2\fstop
	\end{align}
\end{lemma}

\begin{proof}
	Conditionally on the variables outside $B$, the signs $(\sigma_x)_{x\in B}$ are resampled independently and uniformly on $\{\pm1\}$, while the energy variables are resampled according to the $B$-update of the $(\alpha,w)$-${\rm KMP}$ model. This gives \eqref{eq:block-action-Walsh-Gibbs}. Since $\mu_B$ is an orthogonal projection in $L^2(\mu)$, \eqref{eq:Dirichlet-Walsh-Gibbs} follows from \eqref{eq:gen-cone}.
\end{proof}

\begin{proof}[Proof of Theorem \ref{th:GS-p}]
	For $J=\emp$, \eqref{eq:block-action-Walsh-Gibbs} shows that the restriction of $\cG$ to functions of the energy variables alone coincides with $\cL$. Hence,
	\begin{equation}\label{eq:Gibbs-gap-upper-KMP}
		\gap(\cG)\le \gap(\cL)\fstop
	\end{equation}
	
	We prove the reverse inequality. For every $x\in[n]$, the nonzero function
	$f_x(\eta)=\eta_x-\mu(\eta_x)$ satisfies
	\begin{align}
		\ttscalar{f_x}{-\cL f_x}_\mu
		=
		\sum_{B\ni x}w_B\,\ttnorm{f_x-\mu_Bf_x}_\mu^2
		\le
		\ttonde{\sum_{B\ni x}w_B}\ttnorm{f_x}_\mu^2\fstop
	\end{align}
	Therefore, by the variational characterization of the gap,
	\begin{equation}\label{eq:KMP-gap-refresh-rate}
		\textstyle	\gap(\cL)
		\le
		\min_{x\in[n]}\sum_{B\ni x}w_B\fstop
	\end{equation}
	
	Now fix $J\neq\emp$. By \eqref{eq:Dirichlet-Walsh-Gibbs}, for any $x\in J$,
	\begin{equation}
		\ttscalar{h_J\sigma_J}{-\cG(h_J\sigma_J)}_{\kappa}
		\textstyle\ge
		\tonde{\sum_{B:B\cap J\neq\emp}w_B}\ttnorm{h_J}_\mu^2
		\textstyle\ge
		\tonde{\sum_{B\ni x}w_B}\ttnorm{h_J}_\mu^2
		\ge
		\gap(\cL)\ttnorm{h_J}_\mu^2\fstop
	\end{equation}
	For $J=\emp$, the same lower bound holds whenever $\mu(h_\emp)=0$, by the Poincar\'e inequality for $\cL$. If $h$ in \eqref{eq:Walsh-decomposition-Gibbs} satisfies $\kappa(h)=0$, then $\mu(h_\emp)=0$. Summing the preceding bounds over $J\subseteq[n]$ and using the orthogonality of the expansion in \eqref{eq:Walsh-decomposition-Gibbs}, we obtain
	\begin{equation}
		\ttscalar{h}{-\cG h}_{\kappa}
		\ge
		\gap(\cL)\ttnorm{h}_{\kappa}^2\fstop
	\end{equation}
	
	Together with \eqref{eq:Gibbs-gap-upper-KMP}, this proves $\gap(\cG)=\gap(\cL)$. Theorem \ref{th:main-KMP} gives
	\begin{equation}
		\gap(\cG)=\gap(\cL)=\gap(\cL|_{\mathscr P_2})\fstop
	\end{equation}
	Finally, let $f\in\mathscr P_2$ be a $\gap(\cL)$-eigenfunction of $-\cL$. Since $f$ depends only on $\eta$, it gives rise to a $\gap(\cG)$-eigenfunction of $-\cG$. This proves the last claim.
\end{proof}

\section{Perron--Frobenius theorems on quadratic cones for KMP}\label{sec:PF}

In this section, we prepare the ground for establishing the main results described in Section \ref{sec:comparison-intro}, aimed at comparing $\gap_1(\alpha,w)$ with $\gap_{2,\ast}(\alpha,w)$. As already announced in Section \ref{sec:PF-intro} and motivated by the role played by the variance functional $\var_\pi$ in \eqref{eq:var-pi-intro}, the key idea is to identify a \textit{nonnegative quadratic eigenfunction} of the hidden model, and exploit the accompanying structural properties---most notably, its strict positivity and uniqueness, together with the analogous properties of the dual eigenfunction. These conclusions will follow from the Perron--Frobenius theory for cone-preserving operators---the finite-dimensional counterpart of the Krein--Rutman theorem---and its refinements. Although we focus on quadratic cones, some of our conclusions admit extensions to even-degree polynomials (Remark \ref{rem:even-degree-cones}).

Throughout the section, the weights $\alpha$ and $w$ are fixed, while additional assumptions on them will be introduced when needed.
 	
 	\subsection{The quadratic cone and positivity}\label{sec:positivity}
 	
 	Our first task is to identify the appropriate setting for analyzing the eigenspace associated to $\gap_{2,\ast}(\alpha,w)$. In view of the linear isomorphisms intertwining the operators in \eqref{eq:restricted-isospectrality-sym}, we focus, as already done in Section \ref{sec:proof-main-KMP}, on the restriction $\mathscr L|_{\mathscr R_{2,\ast}}$ of the hidden-model generator to the finite-dimensional space $\mathscr R_{2,\ast}$ described in \eqref{eq:R-k-ast}. For brevity, throughout this section we simply write
 	\begin{equation}\label{eq:R=R_2-ast}
 		\mathscr R\eqdef\mathscr R_{2,\ast}
 	\fstop
 	\end{equation} Recall that every $g\in\mathscr R$ is a homogeneous quadratic $\theta$-polynomial invariant under translations along $\mathbf1\in \R^n$. In particular, since $g\in \mathscr R$ satisfies $g(\mathbf0)=0$, it must also vanish on constant configurations: $g(s\mathbf1)=0$ for every $s\in \R$.

 	The natural subset in which to seek the desired eigenfunction is therefore
 	\begin{equation}\label{eq:cone-C-2-star}
 		\mathscr C\eqdef	\mathscr R_\tp=
 		\{
 		g\in\mathscr R:
 		g(\theta)\ge 0
 		\ \text{for every}\ \theta\in\R^n
 		\}\fstop
 	\end{equation}
 	Following the terminology in, e.g., \cite[Definition 1]{schneider_cross_1970} or \cite[\S36.1]{hogben_handbook_2014}, it is immediate to verify that $\mathscr C$ is a proper cone in $\mathscr R$. Indeed, it is closed and convex, stable under multiplication by nonnegative scalars, and pointed, since
 	$
 	\mathscr C
 	\cap
 	(-\mathscr C)
 	=
 	\{0\}$.
 	Moreover, its interior
 	\begin{equation}\label{eq:cone-interior}
 		\interior(\mathscr C)=\set{g\in \mathscr C: g(\theta)=0\ \text{if and only if}\ \theta = {\rm const.}}
 	\end{equation}
 	is nonempty as, for instance, $\var_\pi$ belongs to it; cf.\  Lemma \ref{lem:var-power} with $k=2$. Here and below, interiors are taken in the finite-dimensional space $\mathscr R$, endowed with any norm; for instance, for any $-\infty<a<b<\infty$, one may consider the uniform norm $\norm{g}_{a,b}=\sup_{\theta \in [a,b]^n}|g(\theta)|$ already used in \eqref{eq:tensor-semigroup-norm}.
 	
 	Crucially, the hidden-model dynamics preserves the cone $\mathscr C$: recalling that $(e^{t\mathscr L})_{t\ge0}$ denotes the hidden-model semigroup,
 	\begin{equation}\label{eq:positivity}
 		e^{t\mathscr L}\mathscr C
 		\subseteq
 		\mathscr C\comma
 		\qquad t\ge0\fstop
 	\end{equation}
 	Indeed, \eqref{eq:R-k-ast-inv} gives
 	$e^{t\mathscr L}\mathscr R\subseteq\mathscr R$, while the Markov property of the hidden-model semigroup ensures that nonnegative functions remain nonnegative under the semigroup action.

We now record the first consequence of the positivity property in
\eqref{eq:positivity}. The Perron--Frobenius theorem for an operator preserving a proper cone in a finite-dimensional space states that its spectral radius is an eigenvalue and admits a corresponding eigenfunction in the cone; see, e.g.,
\cite[Theorem 3.1]{vandergraft_spectral_1968} and
\cite[Theorem 6]{schneider_cross_1970}.
In the present setting, the first assertion is in fact automatic: by
\eqref{eq:restricted-isospectrality-sym},
$\mathscr L|_{\mathscr R}$ is isospectral to a self-adjoint operator, and hence
$e^{t\mathscr L}|_{\mathscr R}$ has real positive spectrum. Hence, here the main role of the cone theorem is to locate an eigenfunction associated with its spectral radius inside $\mathscr C$.

The same conclusion applies to the dual operator. More precisely, let
$
	\mathscr R'
	\eqdef
	\{
	h:\mathscr R\to\R\ \text{linear}
	\}
$
be the dual space of $\mathscr R$, and let
\begin{equation}\label{eq:dual-cone-C-2-star}
	\mathscr C'
	\eqdef
	\{
	h\in\mathscr R':
	h(g)\ge0
	\ \text{for every}\ g\in\mathscr C
	\}
\end{equation}
be the dual cone. If we endow $\mathscr R$ with the norm $\norm{\emparg}_{a,b}$ for some $-\infty<a<b<\infty$, then $\mathscr R'$ carries the corresponding dual norm $\norm{h}_{a,b}\eqdef \sup_{\norm{g}_{a,b}=1}|h(g)|$, $h\in \mathscr R'$. Remark that, since $\mathscr C$ has nonempty interior, every nonzero element of $\mathscr C'$ is strictly positive on $\interior(\mathscr C)$, that is,
\begin{equation}\label{eq:dual-positive-on-interior}
	h(g)>0\comma
	\qquad
	0\neq h\in\mathscr C'\comma
	g\in\interior(\mathscr C)\fstop
\end{equation}
In what follows, for $h\in\mathscr R'$, we write
\begin{equation}
	(h\mathscr L)(g)
	\eqdef
	h(\mathscr Lg)\comma
	\qquad g\in\mathscr R\fstop
\end{equation}

Applying Perron--Frobenius theory to the hidden-model semigroup and its dual, and then using the spectral mapping theorem, yields the following.

\begin{proposition}[Perron eigenfunctions I]\label{pr:PF-positive}
	There exist
	$
		0\neq g_\ast\in\mathscr C
		$, $
		0\neq h_\ast\in\mathscr C'
	$
	satisfying
	\begin{equation}\label{eq:PF-right-left}
		\mathscr Lg_\ast
		=
		-\gap_{2,\ast}(\alpha,w)g_\ast
		\qquad\text{and}\qquad
		h_\ast\mathscr L
		=
		-\gap_{2,\ast}(\alpha,w)h_\ast\fstop
	\end{equation}
	We also refer to $g_\ast$ and $h_\ast$ as, respectively, right and left Perron eigenfunctions of $\mathscr L|_{\mathscr R}$.
\end{proposition}

We stress that Proposition \ref{pr:PF-positive} requires no geometric assumption on the weights: it holds for arbitrary positive site weights $\alpha$ and nonnegative block weights $w$.

For the comparison arguments developed in Section \ref{sec:comparison}, the mere existence of Perron eigenfunctions will not suffice. In fact, we shall need both $g_\ast$ and $h_\ast$ to be:
\begin{itemize}
	\item \textit{unique} up to positive scaling;
	\item \textit{nondegenerate}, in the sense that
	$
	g_\ast\in\interior(\mathscr C)
	$ and 
	$	h_\ast\in\interior(\mathscr C')$.
\end{itemize}  
Here, $\interior(\mathscr C)$ is described in \eqref{eq:cone-interior}, whereas, as $\mathscr C\subseteq \mathscr R$ is a finite-dimensional proper (i.e., closed, convex and pointed) cone,
\begin{equation}\label{eq:cone-interior-dual}
	\interior(\mathscr C')
	=
	\{
	h\in\mathscr C':
	h(g)>0
	\ \text{for every}\ g\in\mathscr C\setminus\{0\}
	\}\fstop
\end{equation}

These stronger conclusions rest on a corresponding strengthening of \eqref{eq:positivity}, that is, the \textit{strong positivity} of the hidden-model semigroup, i.e., 
\begin{equation}\label{eq:strong-positivity}
	e^{t\mathscr L}
	(\mathscr C\setminus\{0\})
	\subseteq
	\interior(\mathscr C)\comma\qquad t >0\fstop
\end{equation}
Unlike \eqref{eq:positivity}, strong positivity need not hold for arbitrary hypergraph weights. It does, however, hold under a natural geometric condition on the block weights, which will turn out to be optimal when analyzing spectral gaps.

\subsection{Strong positivity}\label{sec:strong-positivity}
The aim of this section is to establish \eqref{eq:strong-positivity} under optimal geometric conditions. Since both the relevant condition and the proof of strong positivity take a particularly transparent form when $w$ induces a graph, we treat the graph and hypergraph cases separately when needed.

\subsubsection*{Graph case.}
 When $w$ induces a graph, connectivity alone upgrades the positivity property \eqref{eq:positivity} to strong positivity.

\begin{proposition}[Strong positivity: graph case]\label{pr:strong-positivity-graph}
	Assume that $n\ge3$, that $w$ satisfies \eqref{eq:graph-case}, and that the graph induced by $w$ is connected. Then, the strong-positivity property \eqref{eq:strong-positivity} holds.
\end{proposition}
For the proof of this and the next result, see Section \ref{sec:strong-positivity-proof} below.

\subsubsection*{Hypergraph case.}
For general block weights, connectivity alone is not sufficient for strong positivity, and the appropriate condition is expressed through the notion of a \textit{minimal hypergraph}, which we now introduce.

Recall that a partition of $[n]$ as in \eqref{eq:partition} is $w$-compatible if every block $B$ with $w_B>0$ is either contained in one of its atoms or is a union of atoms, and that, under \eqref{eq:w-n}--\eqref{eq:connected-hypergraph}, its coarsest partition is the unique one with the smallest number of atoms. We call the hypergraph induced by $w$ \textit{minimal} if $w$ satisfies \eqref{eq:w-n}--\eqref{eq:connected-hypergraph} and its coarsest partition is the discrete partition
$[n]
	=
	\{1\}\sqcup\cdots\sqcup\{n\}$.
Equivalently, a minimal hypergraph is connected, with $w_{[n]}=0$, and admits no nontrivial $w$-compatible partition. On the other hand, connectedness plus $w_{[n]}=0$ is strictly weaker than minimality, as the example in Figure \ref{fig:nonminimal-hypergraph} shows.

\begin{figure}[t]
	\centering
\begin{tikzpicture}[
	vertex/.style={circle,fill=black,inner sep=2pt},
	internal/.style={line width=.9pt},
	hyperedge/.style={
		draw,
		rounded corners=8pt,
		line width=1pt,
		inner xsep=16pt,
		inner ysep=10pt
	},
	blocklabel/.style={font=\small},
	atomlabel/.style={font=\small}
	]
	
	% ------------------------------------------------
	% Vertices
	% ------------------------------------------------
	\node[vertex,label={[yshift=1pt]below:$1$}] (v1) at (0,0) {};
	\node[vertex,label={[yshift=1pt]below:$2$}] (v2) at (1,0) {};
	\node[vertex,label={[yshift=1pt]below:$3$}] (v3) at (2,0) {};
	
	\node[vertex,label={[yshift=1pt]below:$4$}] (v4) at (3.5,0) {};
	\node[vertex,label={[yshift=1pt]below:$5$}] (v5) at (4.5,0) {};
	\node[vertex,label={[yshift=1pt]below:$6$}] (v6) at (5.5,0) {};
	
	\node[vertex,label={[yshift=1pt]below:$7$}] (v7) at (7,0) {};
	
	% Internal blocks
	\draw[internal] (v1) -- (v2);
	\draw[internal] (v2) -- (v3);
	\draw[internal] (v4) -- (v5);
	\draw[internal] (v5) -- (v6);
	
	% ------------------------------------------------
	% Large hyperedges
	% ------------------------------------------------
	\begin{scope}[on background layer]
		
		\node[
		hyperedge,
		fit=(v1)(v2)(v3)(v4)(v5)(v6),
		inner xsep=18pt,
		inner ysep=12pt
		] (B1) {};
		
		% Slight vertical shift separates the two boundaries
		\node[
		hyperedge,
		dashed,
		fit=(v4)(v5)(v6)(v7),
		inner xsep=18pt,
		inner ysep=12pt,
		yshift=-5pt
		] (B2) {};
		
	\end{scope}
	
	% ------------------------------------------------
	% Block labels
	% ------------------------------------------------
	
	% Well above B1
	\node[blocklabel,anchor=south]
	at ([yshift=8pt]B1.north)
	{$B_1=N_1\sqcup N_2$};
	
	% Put B2 label below the atoms, rather than between them
	\node[blocklabel,anchor=north]
	at ([yshift=-18pt]B2.south)
	{$B_2=N_2\sqcup N_3$};
	
	% ------------------------------------------------
	% Atom labels
	% ------------------------------------------------
	\node[atomlabel,anchor=north]
	at (1,-0.72)
	{$N_1=\{1,2,3\}$};
	
	\node[atomlabel,anchor=north]
	at (4.5,-0.72)
	{$N_2=\{4,5,6\}$};
	
	\node[atomlabel,anchor=north]
	at (7,-0.72)
	{$N_3=\{7\}$};
	
\end{tikzpicture}
	
	\caption{
		An example of a connected non-minimal hypergraph on $n=7$ vertices with $w_{[7]}=0$. The positive weight blocks are $\{1,2\}$, $\{2,3\}$, $\{4,5\}$, $\{5,6\}$, $B_1$ and $B_2$ (defined in the picture). 
		The partition
		$[7]=N_1\sqcup N_2\sqcup N_3$ is $w$-compatible, hence the non-minimality. 
	}
	\label{fig:nonminimal-hypergraph}
\end{figure}
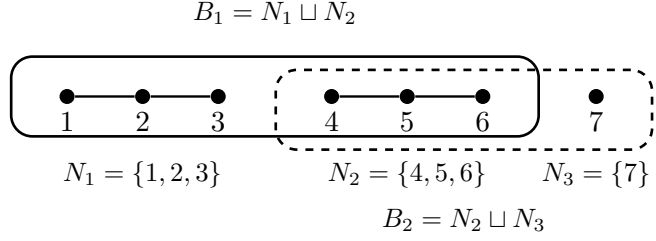

\begin{proposition}[Strong positivity: hypergraph case]\label{pr:strong-positivity-hypergraph}
	Assume that the hypergraph induced by $w$ is minimal as defined above. Then, the strong-positivity property \eqref{eq:strong-positivity} holds.
\end{proposition}
\begin{remark}\label{rem:minimality-graph-connected}
In the graph case \eqref{eq:graph-case}, minimality reduces to ordinary connectivity.
 Hence, under \eqref{eq:graph-case}, Proposition \ref{pr:strong-positivity-hypergraph} reduces to Proposition \ref{pr:strong-positivity-graph}, as expected.
\end{remark}

Before turning to the proof of Propositions \ref{pr:strong-positivity-graph} and \ref{pr:strong-positivity-hypergraph}, we record some important implications of the strong-positivity property
\eqref{eq:strong-positivity}.
\subsubsection*{Perron--Frobenius consequences.}
 The strong Perron--Frobenius theorem for an operator preserving a proper cone in a finite-dimensional space states that, if every nonzero element of the cone is mapped into its interior, then the spectral radius is an algebraically simple eigenvalue, strictly larger in modulus than every other eigenvalue, and its eigenfunction is unique up to positive scaling and belongs to the interior of the cone; see, e.g.,
\cite[Theorem~4.4]{vandergraft_spectral_1968} or \cite[Theorems~2 and 5]{schneider_cross_1970}.
The same conclusions hold for the dual operator. Indeed, strong positivity on $\mathscr C$ implies strong positivity on the dual cone $\mathscr C'$: if $0\neq h\in\mathscr C'$ and $0\neq g\in\mathscr C$, then
\begin{equation}
	h(e^{t\mathscr L}g)>0\comma
	\qquad t>0\fstop
\end{equation}
Thus, the right and left Perron eigenfunctions lie in $\interior(\mathscr C)$ and $\interior(\mathscr C')$, respectively; see \eqref{eq:cone-interior} and \eqref{eq:cone-interior-dual}.
Applying these conclusions to $e^{t\mathscr L}|_{\mathscr R}$ and using the spectral mapping theorem gives the following strengthening of Proposition \ref{pr:PF-positive}.

\begin{proposition}[Perron eigenfunctions II]
	\label{pr:PF-strong}
	Under the assumptions of either Proposition
	\ref{pr:strong-positivity-graph} or Proposition
	\ref{pr:strong-positivity-hypergraph}, the eigenvalue
	$-\gap_{2,\ast}(\alpha,w)$ of $\mathscr L|_{\mathscr R}$ is algebraically simple. Moreover, the right and left Perron eigenfunctions $g_\ast$ and $h_\ast$ from Proposition \ref{pr:PF-positive} are unique up to multiplication by a positive scalar and satisfy
	\begin{equation}\label{eq:PF-interior}
		g_\ast\in\interior(\mathscr C)
		\qquad\text{and}\qquad
		h_\ast\in\interior(\mathscr C')\fstop
	\end{equation}
\end{proposition}
\begin{remark}
	The condition $g_\ast\in\interior(\mathscr C)$ is precisely the nondegeneracy of the right Perron eigenfunction: $g_\ast(\theta)>0$
		for every nonconstant $\theta\in\R^n$.
	Actually, more is true: after fixing a normalization of $g_\ast$, there exist constants $C_1,C_2>0$, depending only on $\alpha$ and $w$, such that
	\begin{equation}\label{eq:var-g-ast}
		C_1\var_\pi(\theta)
		\le
		g_\ast(\theta)
		\le
		C_2\var_\pi(\theta)\comma
		\qquad \theta\in\R^n\fstop
	\end{equation}
	In particular, $g_\ast\in\interior(\mathscr C)$ provides a suitable substitute for $\var_\pi$ in \eqref{eq:chain-1}. Indeed, if $\lambda\ge0$ and $g\in\mathscr R_{k,\ast}$, $k\ge2$, satisfy $\mathscr Lg=-\lambda g$, then, for every $-\infty<a<b<\infty$,  there exists $C=C(g,k,\alpha,a,b)>0$ such that, for every $t>0$ and $\theta \in [a,b]^n$,
	\begin{equation}\label{eq:chain-2}
		e^{-\lambda t}|g(\theta)|
		=
		|e^{t\mathscr L}g(\theta)|
		\le
		e^{t\mathscr L}|g|(\theta)
		\le
		Ce^{t\mathscr L}g_\ast(\theta)
		=
		Ce^{-\gap_{2,\ast}(\alpha,w)t}g_\ast(\theta)
		\comma
	\end{equation}
	where the first and last steps used, respectively, that $g$ and $g_\ast$ are $\mathscr L$-eigenfunctions, the second step is a consequence of the Markovianity of the $\mathscr L$-semigroup, whereas the third step combined \eqref{eq:var-g-ast}, Lemma \ref{lem:G-Var} and \eqref{eq:var-a-b}.
	\end{remark}

\subsection{Proof of strong positivity}\label{sec:strong-positivity-proof}
The proofs of Propositions \ref{pr:strong-positivity-graph} and \ref{pr:strong-positivity-hypergraph} rely on two ingredients. The first, Lemma \ref{lem:propagation-zeros}, is a support argument requiring no geometric assumption on $w$: zeros of a nonnegative quadratic function propagate along every finite sequence of hidden-model updates. The second, Lemma \ref{lem:full-range}, is geometric: under the assumptions of either proposition, the linear range in \eqref{eq:range-theta} generated by such updates equals $\R^n$ whenever the initial configuration is nonconstant.

We start by introducing some notation.
Recall from \eqref{eq:Upsilon-KMP}--\eqref{eq:K_B^U} and \eqref{eq:K_B^U2} that, for the $(\alpha,w)$-${\rm KMP}$ model, the support of its update measure equals
\begin{equation}\label{eq:support-Upsilon-KMP} 
	\supp(\Upsilon)
	=
	\big\{\textstyle
	K_B^U: B\subseteq[n]\ \text{with}\ w_B>0\,,\,
	U\in[0,1]^B\ \text{and}\ 
	\sum_{x\in B}U_x=1
	\big\}
	\fstop
\end{equation}
Indeed, conditionally on the choice of $B$, the redistribution vector has law ${\rm Dir}(\alpha^B)$, whose support is the whole probability simplex on $B$. Further, for $\theta\in\R^n$, define	
\begin{equation}\label{eq:range-theta}
	\cS(\theta) = \cS^\tups(\theta)
	\eqdef
	\spanop
	\big\{
	K_\ell\cdots K_1\theta:
	\ell\ge0\,,\, 
	K_1,\ldots, K_\ell\in\supp(\Upsilon)
	\big\}\subseteq \R^n
\end{equation}
as the linear range generated by applying finitely many successive updates from $\supp(\Upsilon)$ to $\theta\in\R^n$. Since the empty sequence of updates is allowed, one always has $\theta\in\cS(\theta)$.

The next result is the first step in the proof of Propositions \ref{pr:strong-positivity-graph} and \ref{pr:strong-positivity-hypergraph}.	
\begin{lemma}[Propagation of zeros]\label{lem:propagation-zeros}
	Let $g\in\mathscr C$, $\theta\in\R^n$, and $t>0$. If
	\begin{equation}\label{eq:vanish-semigroup}
		(e^{t\mathscr L}g)(\theta)=0\comma
	\end{equation}
	then $g$ vanishes identically on $\cS(\theta)$.
\end{lemma}

\begin{proof}Let $\theta(t)\in \R^n$ be the hidden-model configuration at time $t>0$, started from $\theta \in \R^n$. 
 Since $g\ge0$, \eqref{eq:vanish-semigroup} implies
 $g(\theta(t))=0$ almost surely.
 	
	Fix $\ell\ge1$ and matrices
	$K_{B_1}^{U_1},\ldots,K_{B_\ell}^{U_\ell}\in\supp(\Upsilon)$. Every open neighborhood of
	\begin{equation}\label{eq:conf-K---K}
	\theta'\eqdef	K_{B_\ell}^{U_\ell}\cdots K_{B_1}^{U_1}\theta\in \R^n
	\end{equation}
	has positive probability under the law of the hidden model: one may require exactly $\ell$ updates before time $t$, with the update matrices lying, in the prescribed order, in arbitrarily small neighborhoods of $K_{B_1}^{U_1},\ldots,K_{B_\ell}^{U_\ell}$. Hence, the configuration $\theta'$ in \eqref{eq:conf-K---K} belongs to the support of ${\rm Law}(\theta(t))$, and continuity of $g$ yields $g(\theta')=0$. The same argument, applied to the event that no update occurs before time $t$ (i.e., $\ell=0$), shows that $g(\theta)=0$.
	
	This proves that $g$ vanishes on $\{K_\ell\cdots K_1\theta: \ell\ge 0\ \text{and}\ K_1,\ldots, K_\ell \in \supp(\Upsilon)\}$.	As the zero set of a nonnegative quadratic form is a linear subspace, $g$ vanishes on the span in \eqref{eq:range-theta}.
\end{proof}

The geometric ingredient is the following. 

\begin{lemma}[Full range]\label{lem:full-range}  Assume that the hypergraph induced by $w$ is minimal as in Proposition \ref{pr:strong-positivity-hypergraph}.
 Then, for every nonconstant $\theta\in\R^n$,
	\begin{equation}\label{eq:full-range}
		\cS(\theta)+\R\mathbf 1=\R^n\fstop
	\end{equation}
In particular,  \eqref{eq:full-range} holds true if $w$ induces a connected graph (Remark~\ref{rem:minimality-graph-connected}). 
\end{lemma}

\begin{proof} We shall actually prove 
	\begin{equation}\label{eq:full-range-stronger}
		\cS(\theta)=\R^n\comma
	\end{equation} 
	which is stronger than \eqref{eq:full-range}. We present the proof of \eqref{eq:full-range-stronger} first for connected graphs, and then for minimal hypergraphs.
	
\smallskip \noindent
\emph{Graph case.} Assume that $w$ and induces a connected graph as in Proposition \ref{pr:strong-positivity-graph}. Fix a nonconstant $\theta \in \R^n$, and choose $y,z \in [n]$ such that $\theta_y\neq \theta_z$. 

Fix $x\in [n]$, and let us prove that $e_x = \mathbf1_{\{x\}} \in \cS(\theta)$. Choose $x_0 \in \{y,z\}$ so that 
\begin{equation}\label{eq:theta-x0-x}\theta_{x_0}\neq \theta_x\fstop
	\end{equation} By the assumed connectivity, there exists a simple path $x_0, x_1, \ldots, x_\ell = x$ along edges $\{x_{i-1},x_i\}$ with $w_{\{x_{i-1},x_i\}}>0$, $i=1,\ldots,\ell$. Starting from $\theta^{\{0\}}=\theta \in \R^n$, we successively define, for every $i=1,\ldots,\ell$,
\begin{equation}\label{eq:K-i}
	\theta^{\{i\}}\eqdef K^{\{i\}}\theta^{\{i-1\}}\comma\quad \text{with}\   K^{\{i\}}\eqdef K_{\{x_{i-1},x_i\}}^{e_{x_{i-1}}}\comma
\end{equation}
where $e_{x_{i-1}}$ stands for the probability vector concentrated at $x_{i-1}\in [n]$. Observe that each matrix in \eqref{eq:K-i} belongs to $\supp(\Upsilon)$ in \eqref{eq:support-Upsilon-KMP} and, thus, $\theta^{\{0\}},\ldots, \theta^{\{\ell\}}\in \cS(\theta)$. Moreover, by \eqref{eq:K_B^U2}, $K^{\{i\}}$ acts on $\theta^{\{i-1\}}$ only by overwriting the value at site $x_i$ with $\theta_{x_{i-1}}^{\{i-1\}}$.
Consequently, 
\begin{equation}
	\theta^{\{\ell\}}-\theta^{\{\ell-1\}}=\tonde{\theta_{x_0}-\theta_x}e_x\in \cS(\theta)\subseteq\R^n\fstop
\end{equation}
By \eqref{eq:theta-x0-x}, the coefficient is nonzero, thus, $e_x\in\cS(\theta)$. Since $x\in[n]$ was arbitrary, $\cS(\theta)=\R^n$. 
\subsubsection*{Hypergraph case.}
Assume that the hypergraph induced by $w$ is minimal as in Proposition \ref{pr:strong-positivity-hypergraph}. Fix a nonconstant $\theta\in\R^n$. We again use only updates of the form $K_B^{e_x}$, for some $x\in B$, overwriting the values at $y\in B$ with $\theta_x\in \R$.

We first show that $\mathbf1\in\cS(\theta)$. Choose $x\in[n]$ with $\theta_x\neq0$, and let $N$ be the set of coordinates currently equal to $\theta_x$. As long as $N\neq[n]$, connectedness provides a block $B$ with $w_B>0$ crossing the partition $N\sqcup N^\complement$. Choosing $z\in B\cap N$ and applying $K_B^{e_z}$ makes every coordinate in $B$ equal to $\theta_x$, and therefore strictly enlarges $N$. Iterating, one reaches the constant configuration $\theta_x\mathbf1$. Thus,
\begin{equation}\label{eq:one-in-range}
	\mathbf1\in\cS(\theta)
	\fstop
\end{equation}

By connectedness (see \eqref{eq:connected-hypergraph}), there exists a block $B$ with $w_B>0$ on which $\theta$ is nonconstant. Choose $x,y\in B$ with $\theta_x\neq\theta_y$. Then,
\begin{equation}
	\ttonde{K_B^{e_x}-K_B^{e_y}}\theta
	=
	(\theta_x-\theta_y)\mathbf1_B
	\fstop
\end{equation}
Hence, $\mathbf1_B\in\cS(\theta)$, and \eqref{eq:one-in-range} also gives $\mathbf1_{B^\complement}\in\cS(\theta)$. Since $w_{[n]}=0$, both $B$ and $B^\complement$ are nonempty. We have therefore obtained a partition of $[n]$ into at least two atoms whose indicator vectors belong to $\cS(\theta)$.

Starting from this partition, refine it successively as follows. Suppose that
$
	[n]=N_1\sqcup\cdots\sqcup N_m
$
is a partition with $2\le m\le n$ and $\mathbf1_{N_j}\in\cS(\theta)$ for every $j=1,\ldots,m$. If it is not $w$-compatible, there exist a block $D$ with $w_D>0$ and an atom $N_j$ such that
\begin{equation}
	N_j\cap D\neq\emp\comma
	\qquad
	N_j\setminus D\neq\emp\comma
	\qquad
	D\setminus N_j\neq\emp
	\fstop
\end{equation}
Choose $z\in D\setminus N_j$. Since
$K_D^{e_z}\mathbf1_{N_j}
	=
	\mathbf1_{N_j\setminus D}$, both
\begin{equation}
	\mathbf1_{N_j\setminus D}
	\qquad\text{and}\qquad
	\mathbf1_{N_j\cap D}
	=
	\mathbf1_{N_j}-\mathbf1_{N_j\setminus D}
\end{equation}
belong to $\cS(\theta)$. We may therefore replace $N_j$ by the two nonempty atoms $N_j\cap D$ and $N_j\setminus D$.

Each refinement increases the number of atoms, so after finitely many steps one obtains a $w$-compatible partition all of whose atom indicators belong to $\cS(\theta)$. By minimality, this partition must be the discrete one. Hence, $\mathbf1_{\{x\}}=e_x \in \cS(\theta)$ for every $x\in [n]$, thus, 
 proving \eqref{eq:full-range-stronger}.	
\end{proof}

We may now conclude the proof of strong positivity.
\begin{proof}[Proof of Propositions \ref{pr:strong-positivity-graph} and \ref{pr:strong-positivity-hypergraph}]
 Fix $t>0$ and $0\neq g\in\mathscr C$. By \eqref{eq:positivity}, one has $e^{t\mathscr L}g\in\mathscr C$. If $e^{t\mathscr L}g\notin\interior(\mathscr C)$, then \eqref{eq:cone-interior} yields a nonconstant $\theta\in\R^n$ such that
\begin{equation}
	(e^{t\mathscr L}g)(\theta)=0\fstop
\end{equation}
By Lemma \ref{lem:propagation-zeros} and the definition of $\mathscr R$, $g$ vanishes, respectively, on $\cS(\theta)$ and $\R\mathbf 1$, whereas Lemma \ref{lem:full-range} gives $\cS(\theta)+\R\mathbf 1=\R^n$. Hence, $g$ vanishes identically, a contradiction. This concludes the proof.
\end{proof}

We mention below some possible refinements of Lemma \ref{lem:full-range}, which will play a role in Section \ref{sec:tau=infty}.
\begin{remark}	\label{rem:full-range-refinements}
The full-range property in Lemma \ref{lem:full-range} depends on both the geometry induced by $w$ and the chosen updates. However, the proof uses only matrices of the form $K_B^{e_x}$: setting
	\begin{equation}\label{eq:supp-Upsilon-0}
		\supp(\Upsilon_0)
		\eqdef
		\big\{
		K_B^{e_x}:
		B\subseteq[n]\ \text{with}\ w_B>0\,,\
		x\in B
		\big\}
		\subsetneq
		\supp(\Upsilon)\comma
	\end{equation}
	the proof actually shows that, for every nonconstant $\theta\in\mathbb R^n$,
	\begin{equation}\label{eq:full-range-Upsilon-0}
		\cS^{\tups_0}(\theta)
		\eqdef
		\operatorname{span}
		\big\{
		K_\ell\cdots K_1\theta:
		\ell\ge0\,,\
		K_1,\ldots,K_\ell\in\supp(\Upsilon_0)
		\big\}
		=
		\mathbb R^n\fstop
	\end{equation}
While the inclusion in \eqref{eq:supp-Upsilon-0} yields
$\cS^{\tups_0}(\theta)\subseteq\cS^\tups(\theta)=\cS(\theta)$,
the strictness of that inclusion makes \eqref{eq:full-range-Upsilon-0} a strictly stronger statement than \eqref{eq:full-range-stronger}.
	
	In Section \ref{sec:tau=infty} below, we will need a refinement of Lemma \ref{lem:full-range}, now involving (recall \eqref{eq:K_B^U2})
	\begin{equation}\label{eq:supp-Upsilon-infty}
		\supp(\Upsilon_\infty)
		\eqdef
		\big\{
		K_B^\pi:
		B\subseteq[n]\ \text{with}\ w_B>0
		\big\}
		\subsetneq
		\supp(\Upsilon)\comma
	\end{equation}
	with $\pi=\alpha/\alpha_0$ as given in \eqref{eq:var-pi-intro}. For the corresponding linear range
	\begin{equation}\label{eq:full-range-Upsilon-infty}
		\cS^{\tups_\infty}(\theta)
		\eqdef
	\spanop
		\big\{
		K_\ell\cdots K_1\theta:
		\ell\ge0\,,\
		K_1,\ldots,K_\ell\in\supp(\Upsilon_\infty)
		\big\}
		\comma
	\end{equation}
	we shall show that $\cS^{\tups_\infty}(\theta)\subseteq \cS^{\tups_0}(\theta)$ (see \eqref{eq:range-inclusion-infty-zero-app}) and that, under the assumptions of Lemma~\ref{lem:full-range}, $\cS^{\tups_\infty}(\theta)+\R\mathbf1=\R^n$. The proof of this last claim is more delicate than that of Lemma~\ref{lem:full-range} and is deferred to Appendix~\ref{app:modules-etc}.
\end{remark}

 	Finally, we conclude this section by noting a partial extension of (strong) positivity on even higher-degree polynomials.

 	\begin{remark}[Even-degree cones]\label{rem:even-degree-cones} Thus far, we have focused exclusively on the quadratic positive cone
 		$\mathscr C\subseteq\mathscr R=\mathscr R_{2,\ast}$.
	 		More generally, we could have considered, for every $k\ge1$,
 		\begin{equation}
 			\mathscr C_{k,\ast}
 			\eqdef
 			\{
 			g\in\mathscr R_{k,\ast}:
 			g(\theta)\ge0
 			\ \text{for every}\ \theta\in\R^n
 			\}\fstop
 		\end{equation}
 		It is not difficult to check that this is a proper cone in $\mathscr R_{k,\ast}$ with nonempty interior if and only if $k\ge 1$ is \textit{even}; recall from Lemma \ref{lem:var-power} that
 		$\var_\pi^{k/2}\in\interior(\mathscr C_{k,\ast})$ if $k\ge 2$ is even.	 Moreover, Markovianity of the hidden-model semigroup ensures
 		$e^{t\mathscr L}\mathscr C_{k,\ast}\subseteq\mathscr C_{k,\ast}$ for every $t\ge0$.
 		Thus, letting $\mathscr C_{k,\ast}'$ denote the dual cone, the Perron--Frobenius
 		argument of Proposition \ref{pr:PF-positive} yields nonzero right and left
 		Perron eigenfunctions
 	$
 			g_{k,\ast}\in\mathscr C_{k,\ast}
 			$ and $			h_{k,\ast}\in\mathscr C_{k,\ast}'
 		$
 		such that
 		\begin{equation}
 			\mathscr Lg_{k,\ast}
 			=
 			-\gap_{k,\ast}(\alpha,w)g_{k,\ast}
 			\qquad\text{and}\qquad
 			h_{k,\ast}\mathscr L
 			=
 			-\gap_{k,\ast}(\alpha,w)h_{k,\ast}\fstop
 		\end{equation}
 		The strong-positivity and uniqueness conclusions of Proposition
 		\ref{pr:PF-strong}, however, seem specific to the quadratic cone and need
 		not hold for $k\ge4$: taking, e.g., $k=4$ and the three-site segment, the hidden-model semigroup does not map the polynomial $\theta\longmapsto g(\theta)=(\theta_1-\theta_2)^2(\theta_2-\theta_3)^2\in \mathscr C_{4,\ast}\setminus \{0\}$ into $\interior(\mathscr C_{4,\ast})$, as $e^{t\mathscr L}g(\theta)=0$ for every $t\ge 0$ if either $\theta_1=\theta_2$ or $\theta_2=\theta_3$.
 	\end{remark}

 \section{Hypergraph reduction and gaps of ${\rm KMP}$}
 \label{sec:hypergraph-reduction}

Before turning to the comparison of $\gap_1(\alpha,w)$ and $\gap_{2,\ast}(\alpha,w)$ (Section \ref{sec:comparison}), we reduce the underlying weighted hypergraph to a minimal one; readers interested only in graphs, rather than general hypergraphs, may proceed directly to Section~\ref{sec:comparison}, as \eqref{eq:w-n}--\eqref{eq:connected-hypergraph} and \eqref{eq:graph-case} alone already imply minimality.

Throughout, we fix positive site weights
$\alpha$ and nonnegative block weights
$w$.

Recall that the hypergraph strong-positivity result of
Proposition \ref{pr:strong-positivity-hypergraph} requires minimality, whereas the
assumptions of our main theorems from
Section \ref{sec:comparison-intro} impose only
\eqref{eq:w-n}--\eqref{eq:connected-hypergraph}, namely,
$w_{[n]}=0$ and connectedness. This apparent discrepancy is removed by
collapsing the atoms of its coarsest partition into new sites. The main goal of this section is to prove that the resulting quotient hypergraph is minimal and, more importantly, preserves the relevant gaps: for each $k\ge1$, $\gap_{k,\ast}(\alpha,w)$ coincides with the corresponding gap of the quotient
${\rm KMP}$ model with aggregated site weights.
This is the content of
Propositions \ref{pr:reduction}--\ref{pr:reduction-gap} below. Thus, after this reduction, minimality may be assumed without loss of
generality when proving the main results in Section \ref{sec:comparison}.

We first need a few additional definitions. Recall from \eqref{eq:partition} that a partition
\begin{equation}\label{eq:partition-coarsest-compatible}
	[n]=N_1\sqcup\cdots\sqcup N_m
	\comma
	\qquad
	2\le m\le n\comma
\end{equation}
is $w$-compatible if every block $B$ with $w_B>0$ is either contained in one
atom or is a union of them.  Such a partition exists, since the discrete
partition is $w$-compatible. Fix one and, for $j\in [m]$ and $J\subseteq[m]$, set
\begin{equation}\label{eq:weights-quotient}
	N_J\eqdef\bigsqcup_{j\in J}N_j
	\comma
	\qquad
	\bar\alpha_j\eqdef\alpha(N_j) = \textstyle \sum_{x\in N_j}\alpha_x
	\comma
	\qquad
	\bar w_J
	\eqdef
	\begin{dcases}
		w_{N_J}
		&\text{if}\ |J|\ge2\comma\\
		0
		&\text{if}\ |J|\le1\fstop
	\end{dcases}
\end{equation}
We call the weighted hypergraph on $[m]$ with site weights
$\bar\alpha=(\bar\alpha_j)_{j\in[m]}$ and block weights
$\bar w=(\bar w_J)_{J\subseteq[m]}$ the \textit{quotient hypergraph} associated with
\eqref{eq:partition-coarsest-compatible}.

\begin{proposition}[Minimal-hypergraph reduction]
	\label{pr:reduction}
	Assume \eqref{eq:w-n}--\eqref{eq:connected-hypergraph}, and let
	\eqref{eq:partition-coarsest-compatible} be its coarsest partition. Then the associated quotient hypergraph is
	minimal. 
\end{proposition}
\begin{proof}
	Suppose that the quotient hypergraph were
	disconnected, and let $C_1,\ldots,C_\ell$, $\ell\ge2$, denote its
	connected components. Merging the original atoms according to these
	components gives the partition
	\begin{equation}
		[n]
		=
		 N_1'\sqcup\cdots\sqcup N_\ell'
		\comma
		\qquad
		 N_a'
		\eqdef
		\bigsqcup_{j\in C_a}N_j\fstop
	\end{equation}
	This partition is again $w$-compatible. Indeed, a block contained in
	some $N_j$ remains internal. If instead a block is of the form
	$N_J$ with $\bar w_J=w_{N_J}>0$, then the quotient block $J$ cannot meet two distinct connected
	components, and hence $N_J$ is contained in one of the
	$N_a'$.

	Since the original hypergraph is connected and $m\ge2$, some
	block $B\subseteq[n]$ with $w_B>0$ must meet at least two atoms. Thus, at least one component $C_a$
	contains at least two quotient sites, and consequently
	$2\le\ell<m$. This contradicts the coarseness of
the original partition \eqref{eq:partition-coarsest-compatible}. Hence, the quotient hypergraph is
	connected.
	
	If $m=2$, its only possible block connecting the two quotient sites is
	$[2]$. However, 
	$\bar w_{[2]}=w_{[n]}=0$ by \eqref{eq:w-n}, contradicting quotient connectivity. Therefore
	$m\ge3$. Finally, if the quotient admitted a compatible partition
	\begin{equation}
		[m]=C_1\sqcup\cdots\sqcup C_\ell
		\comma
		\qquad 1<\ell<m\comma
	\end{equation}
	then replacing each $C_a$ by
	$\sqcup_{j\in C_a}N_j$ would produce a $w$-compatible partition of
	$[n]$ with fewer than $m$ atoms. This is again impossible by
	coarseness. The quotient hypergraph is therefore minimal.
	\end{proof}
	
	Only in the next proposition, where ${\rm KMP}$ models with different numbers of sites are
	considered simultaneously, we indicate the number of sites by a superscript.
	Thus, for the original model on $[n]$, we write, for instance,
	$\cL^\tn=\cL$ and
	$\gap_{k,\ast}^\tn(\alpha,w)=\gap_{k,\ast}(\alpha,w)$ for the objects
	defined in \eqref{eq:gen-KMP} and \eqref{eq:gap-inf-star-intro},
	respectively; the corresponding objects for the quotient model on $[m]$ carry the
	superscript $[m]$.
	
	\begin{proposition}[Hypergraph reduction and gaps]\label{pr:reduction-gap} Given a $w$-compatible partition as in \eqref{eq:partition-coarsest-compatible}, one has,	
		for the corresponding weights defined in
		\eqref{eq:weights-quotient},
		\begin{equation}\label{eq:gap-reduction}
			\gap_{k,\ast}^\tn(\alpha,w)
			=
			\gap_{k,\ast}^\tm(\bar\alpha,\bar w)
			\comma
			\qquad k\ge1\fstop
		\end{equation}
	\end{proposition}
	\begin{proof}	
	By $w$-compatibility of \eqref{eq:partition-coarsest-compatible} and \eqref{eq:weights-quotient}, the generator in \eqref{eq:gen-KMP} decomposes as
	$
	\cL^\tn
	=
	\cL_\bullet^\tn
	+
	\cL_\circ^\tn
	$,
	where
	\begin{equation}
		\textstyle
		\cL_\bullet^\tn
		\eqdef
		\sum_{J\subseteq[m]}
		\bar w_J\,\tttonde{\mu_{N_J}^\tn-I}
		\comma
		\qquad
		\cL_\circ^\tn
		\eqdef
		\sum_{j\in[m]}\sum_{B\subseteq N_j}
		w_B\,\tttonde{\mu_B^\tn-I}
		\comma
	\end{equation}
	with $\mu_B^\tn$, $B\subseteq[n]$,  being the conditional $\mu^\tn$-expectation given $\eta^{B^\complement}$.
Thus, $\cL_\bullet^\tn$ and $\cL_\circ^\tn$ retain, respectively, the updates across multiple atoms and those within a single atom of \eqref{eq:partition-coarsest-compatible}.

Besides the orthogonal projections $\mu_B^\tn$, for $B\subseteq[n]$ with $w_B>0$, we also consider $E_j\eqdef \mu_{N_j}^\tn$, $j\in [m]$. Since the atoms $N_1,\ldots,N_m$ are disjoint, $E_1,\ldots, E_m$ are commuting orthogonal projections. Moreover, the corresponding conditional expectations commute: for nested blocks this follows from the tower property, whereas for disjoint blocks it follows from the fact that, under $\mu^{[n]}={\rm Dir}(\alpha)$, the conditional law of the configuration inside a block given its complement depends on the latter only through the total mass within that block, which is left unchanged by updates on a disjoint block. 
Hence,
\begin{equation}\label{eq:E-j1}
	E_j\mu_B^{[n]}
	=
	\mu_B^{[n]}E_j
	\comma
	\qquad
	E_j\mu_{N_J}^{[n]}
	=
	\mu_{N_J}^{[n]}E_j
	\comma
\end{equation}
for every $j\in [m]$, $B\subseteq N_\ell$, $\ell\in [m]$, and $J\subseteq[m]$, with the further identities
\begin{equation}\label{eq:E-j2}
	E_j\mu_B^{[n]}
	=
	E_j
	\quad\text{if}\ B\subseteq N_j\comma
	\qquad
	\mu_{N_J}^{[n]}E_j
	=
	\mu_{N_J}^{[n]}
	\quad\text{if}\ j\in J\fstop
\end{equation}

	Recall \eqref{eq:P-k-star}. By the argument in the proof of Proposition \ref{pr:Pk-star}, these projections preserve
	$\mathscr P_{k,\ast}^\tn$ and each summand in
	\begin{equation}\label{eq:profile-sector-decomposition}\textstyle
		\mathscr P_{k,\ast}^\tn
		=
		\bigoplus_{S\subseteq[m]}
		\mathscr P_{k,\ast}^\tn(S)\comma
		\quad\text{with}\
		\mathscr P_{k,\ast}^\tn(S)
		\eqdef
		\big(
		\prod_{j\in S}(I-E_j)
		\prod_{j\in[m]\setminus S}E_j
		\big)
		\mathscr P_{k,\ast}^\tn\fstop
	\end{equation}
	
	Let us first focus on the summand corresponding to $S=\emp$. 
	Since $E\eqdef E_1\cdots E_m$ is conditional expectation given the atom masses
	$\bar \eta=(\eta(N_j))_{j\in[m]}$, Dirichlet aggregation and self-adjointness of $E$ give
	\begin{equation}\label{eq:quotient-polynomial-component}
		\mathscr P_{k,\ast}^\tn(\emp)
		=
		\{
			F(\bar\eta):
			F\in\mathscr P_{k,\ast}^\tm
		\}\fstop
	\end{equation}
	By combining this with  $\mu_{N_J}^\tn F(\bar\eta)= (\mu_J^\tm F)(\bar \eta)$ and the fact that, by the first identity in \eqref{eq:E-j2}, $\cL_\circ^\tn$ vanishes on $\mathscr P_{k,\ast}^\tn(\emp)$, $\cL^\tn|_{\mathscr P_{k,\ast}^\tn(\emp)}$ is unitarily equivalent to $\cL^\tm|_{\mathscr P_{k,\ast}^\tm}$.

	It remains to show that the spectral gap of $\cL^\tn$ on $\mathscr P_{k,\ast}^\tn$ is not attained by functions in  $\mathscr P_{k,\ast}^\tn(S)$, $\emp\neq S\subseteq[m]$. Fix a nonempty $S\subseteq[m]$, $j\in S$, and
	$f\in\mathscr P_{k,\ast}^\tn(S)$. Then, 
	\begin{equation}\label{eq:nonquotient-gap-lower-bound}
	\textstyle	\scalar{f}{-\cL^\tn f}_{\mu^\tn}
	\ge \sum_{J\ni j}\bar w_J\,\ttscalar{f}{(I-\mu_{N_J}^\tn)f}_{\mu^\tn}= \norm{f}_{\mu^\tn}^2 \tonde{\sum_{J\ni j}
	\bar w_J}\comma
	\end{equation}
	 	where for the first step we only kept updates of blocks $B=N_J$ with $j\in J$, while the second step used $\mu_{N_J}^{[n]}f=0$, $J\ni j$, which follows from   $f\in \mathscr P_{k,\ast}^{[n]}(S)\subseteq (I-E_j)\mathscr P_{k,\ast}^{[n]}$, $j\in S$, implying $E_j f=0$, together with the second inequality in \eqref{eq:E-j2}.	
	 Here, $\sum_{J\ni j}\bar w_J$ may also be interpreted as the refresh rate of site $j\in [m]$ in the quotient hypergraph.
	 	We claim that
	 	\begin{equation}\label{eq:refresh-rate}\textstyle \big(	\sum_{J\ni j}\bar w_J\big)\ge \gap_{k,\ast}^\tm(\bar\alpha,\bar w)\fstop
	 	\end{equation}
	 Indeed, for any continuous function $F=F(\bar \eta)$ of the only coordinate $\bar\eta_j$, 
	 	\begin{equation}
	 	\textstyle	\ttscalar{F}{-\cL^\tm F}_{\mu^\tm}= \sum_{J\ni j} \bar w_J\,\ttscalar{F}{(I-\mu_J^\tm)F}_{\mu^\tm}\le \norm{F}_{\mu^\tm}^2\tonde{\sum_{J\ni j}\bar w_J}\comma
	 	\end{equation}
	 	where the first step follows because updates not involving $j\in [m]$ do not contribute to the Dirichlet form, while for the second step we used that $I-\mu_J^\tm$ is an orthogonal projection.  This proves \eqref{eq:refresh-rate} and,  together with the aforementioned unitary equivalence,   \eqref{eq:gap-reduction}.
\end{proof}

\begin{remark}Given a $w$-compatible partition as in \eqref{eq:partition}, recall from \eqref{eq:weights-tilde} the definition of the block weights $\widetilde w$ on $[n]$ obtained from $w$.  Proposition \ref{pr:reduction-gap} also shows
$
	\gap_{k,\ast}^\tn(\alpha,w)=\gap_{k,\ast}^\tn(\alpha,\widetilde w)$, $k\ge 1$.
Indeed, since $w$ and $\widetilde w$ induce the same quotient hypergraph,  $\gap_{k,\ast}^\tn(\alpha,\widetilde w)=\gap_{k,\ast}^\tm(\bar\alpha,\bar w)$, whereas the monotonicity $w_B\ge \widetilde w_B$, $B\subseteq[n]$, induces $\gap_{k,\ast}^\tn(\alpha,w)\ge\gap_{k,\ast}^\tn(\alpha,\widetilde w)$.
\end{remark}

In view of Propositions \ref{pr:reduction}--\ref{pr:reduction-gap}, the analysis of
$\gap_{k,\ast}(\alpha,w)$ under
\eqref{eq:w-n}--\eqref{eq:connected-hypergraph} may be reduced without loss
of generality to minimal hypergraphs. In the next section, we therefore no
longer need compatible partitions or quotient hypergraphs and revert to the
notation without the superscripts $[n]$ and $[m]$.

 \section{One- and two-particle gaps of ${\rm KMP}$}
 \label{sec:comparison}
 
 In this section, we compare
 $\gap_1(\alpha,w)$ and
 $\gap_{2,\ast}(\alpha,w)$, and prove the results stated in
 Section \ref{sec:comparison-intro}. The Perron--Frobenius theory developed in
 Section \ref{sec:PF}, together with the strong-positivity results of
 Section \ref{sec:strong-positivity}, is central to this comparison.

 The section is organized in three main parts: we first establish strict monotonicity in the scale parameter $\tau$, and then analyze separately the limiting regimes $\tau\to \infty$ and $\tau\to 0$. These are the contents, respectively, of Sections \ref{sec:monotonicity}, \ref{sec:tau-infty} and \ref{sec:tau-0}. At the end, we collect these ingredients and show how they yield Theorems \ref{th:gap-1-2-quantitative}, \ref{th:phase-transition}, \ref{th:dichotomy-graph} and \ref{th:dichotomy-hypergraph} in Section \ref{sec:comparison-final}.

 Throughout this section, we fix positive site weights
 $\alpha$ and nonnegative block weights $w$. For most of the analysis, in view of
 Propositions \ref{pr:reduction}--\ref{pr:reduction-gap}, we assume without loss of generality that
 $w$ induces a minimal hypergraph.
  Readers interested only in graphs \eqref{eq:graph-case} may simply assume that $n\ge 3$ and that $w$ induces a connected graph.

\subsection{Two-particle gap monotonicity}\label{sec:monotonicity}

The main goal of this section is to prove that the map in \eqref{eq:gap-2-monotonicity} is strictly increasing anytime the underlying hypergraph is minimal.

\begin{proposition}[Monotonicity]\label{pr:monotonicity} For all positive site weights $\alpha$ and nonnegative block weights $w$ inducing a minimal hypergraph, the map	$(0,\infty)\ni \tau \longmapsto \gap_{2,\ast}(\tau\alpha,w) \in (0,\infty)
$	is continuous and strictly increasing.
	\end{proposition}
	\begin{remark}\label{rem:monotonicity-gap-1}
		This monotone behavior is specific to $\gap_{2,\ast}(\tau\alpha,w)$. Indeed, the one-particle rates of the $(\tau\alpha,w)$-${\rm KMP}$ model, given by (cf.\ \eqref{eq:RW-rates})
		\begin{equation}
			r_{xy}(\tau\alpha,w)
			=
			\tau\alpha_y
			\sum_{B\ni x,y}
			\frac{w_B}{\tau\alpha(B)}
			=
			r_{xy}(\alpha,w)
			\comma
			\qquad x,y\in[n]\comma x\neq y\comma
		\end{equation}
		do not depend on $\tau>0$. Consequently, the map
		$
			\tau\longmapsto\gap_1(\tau\alpha,w)
	$
		is constant.
	\end{remark}

	We now begin the proof of Proposition \ref{pr:monotonicity}. As in most of our arguments, we work at the level of the hidden model. We start by fixing some notation used throughout.
	
	The weights $\alpha$ and $w$ are kept fixed and are often suppressed from the notation, whereas the dependence on $\tau\in(0,\infty)$ is indicated by subscripts. For instance, we write $\mathscr L_\tau$, $g_\tau$, and $h_\tau$ in place of $\mathscr L$, $g_\ast$, and $h_\ast$ from Propositions \ref{pr:PF-positive} and \ref{pr:PF-strong}, respectively, and set
	\begin{equation}\label{eq:lambda-tau}
		\lambda_\tau
		\eqdef
		\gap_{2,\ast}(\tau\alpha,w)
		\comma
		\qquad \tau\in(0,\infty)\fstop
	\end{equation}
	Note that the space $\mathscr R=\mathscr R_{2,\ast}$ and the cone $\mathscr C\subseteq\mathscr R$, defined respectively in \eqref{eq:R=R_2-ast} and \eqref{eq:cone-C-2-star}, depend neither on the weights $\alpha,w$ nor on $\tau$.
	
	We also introduce the following blockwise notation. Recall $\pi_x=\alpha_x/\alpha_0$. For every nonempty block $B\subseteq[n]$, $x\in[n]$, and $\theta\in\R^n$, let
	\begin{equation}\label{eq:var-pi-conditional-pi}
	\textstyle	\pi_{x|B}
		\eqdef
		\mathbf1_B(x)\frac{\alpha_x}{\alpha(B)} = \mathbf1_B(x)\frac{\pi_x}{\pi(B)}
		\comma
		\qquad
		\pi_B(\theta)
		\eqdef
		\sum_{x\in B}\pi_{x|B}\,\theta_x\comma
	\end{equation}
	\begin{equation}\label{eq:var-pi-conditional-var}\textstyle
		\var_{\pi_B}(\theta)
		\eqdef
		\sum_{x\in B}
		\pi_{x|B}
		\tonde{\theta_x-\pi_B(\theta)}^2\fstop
	\end{equation}
	These are the $B$-conditional analogues of the quantities in \eqref{eq:var-pi-intro}; in particular, they are invariant under the rescaling $\alpha\mapsto\tau\alpha$. Moreover,
	$\var_{\pi_B}\in\mathscr C$.

To prove Proposition \ref{pr:monotonicity}, we first establish an explicit comparison formula for the hidden-model generators corresponding to two values of $\tau$.
\begin{lemma}\label{lem:hidden-generator-monotonicity} Fix some positive site weights $\alpha$ and some nonnegative block weights $w$. Then, 	for every $0<\tau_1<\tau_2<\infty$ and $g\in\mathscr R$,
	\begin{equation}		\label{eq:hidden-generator-difference}
		\tonde{\mathscr L_{\tau_1}-\mathscr L_{\tau_2}}g
		=
		\sum_{B\subseteq[n]}
		w_B
		\left(
		\frac{1}{1+\tau_1\alpha(B)}
		-
		\frac{1}{1+\tau_2\alpha(B)}
		\right)
		g(\mathbf1_B)\,
		\var_{\pi_B}\in \mathscr R\fstop
	\end{equation}
	In particular,
	\begin{equation}\label{eq:hidden-generator-cone-monotonicity}
		\tonde{\mathscr L_{\tau_1}-\mathscr L_{\tau_2}}\mathscr C
		\subseteq
		\mathscr C\comma
	\end{equation}
	and, if $w$ additionally satisfies \eqref{eq:w-n}--\eqref{eq:connected-hypergraph}, 
	\begin{equation}\label{eq:hidden-generator-cone-monotonicity-int}
		\tonde{\mathscr L_{\tau_1}-\mathscr L_{\tau_2}}
		\interior(\mathscr C)
		\subseteq
		\interior(\mathscr C)\fstop
	\end{equation}
\end{lemma}
\begin{proof}
	For every nonempty $B\subseteq[n]$,  $\tau>0$ and $\theta \in \R^n$, let
	\begin{equation}\label{eq:U-tau-Z}
		\textstyle
		U^{B,\tau}\sim
		{\rm Dir}\tonde{(\tau\alpha_x)_{x\in B}}
\comma\qquad 
		W_{B,\tau}(\theta)
		\eqdef
		\sum_{x\in B}
		\tttonde{U_x^{B,\tau}-\pi_{x|B}}\,\theta_x\in \R\fstop
	\end{equation}
	Then, recalling \eqref{eq:K_B^U2},
	$
		K_B^{U^{B,\tau}}\theta
		=
		K_B^\pi\theta
		+
		W_{B,\tau}(\theta)\mathbf1_B\in \R^n$.
	Since taking expectation with respect to $U^{B,\tau}$ gives $\mathbf E_{B,\tau}[W_{B,\tau}(\theta)]=0$, and since $g\in \mathscr R$ is quadratic, expanding yields
	\begin{equation}
		\mathbf  E_{B,\tau}
		[
		g\tttonde{K_B^{U^{B,\tau}}\theta}
		]
		=
		g\tttonde{K_B^\pi\theta}
		+
		\mathbf E_{B,\tau}
		[
		W_{B,\tau}(\theta)^2
		]\,
		g(\mathbf1_B)\fstop
	\end{equation}
	The Dirichlet covariance formula for \eqref{eq:U-tau-Z} gives
	\begin{equation}
		\mathbf E_{B,\tau}
		[
		W_{B,\tau}(\theta)^2
		]
		=
		\frac{\var_{\pi_B}(\theta)}
		{1+\tau\alpha(B)}\fstop
	\end{equation}
	Inserting this identity into the definition of $\mathscr L_\tau$ and subtracting the resulting expressions at $\tau_1$ and $\tau_2$ proves
	\eqref{eq:hidden-generator-difference}.
	
	If $g\in\mathscr C$, then $g(\mathbf1_B)\ge0$, and every coefficient on the right-hand side of \eqref{eq:hidden-generator-difference} is nonnegative. Since $\var_{\pi_B}\in\mathscr C$, this proves \eqref{eq:hidden-generator-cone-monotonicity}.
	
	Finally,  assume \eqref{eq:w-n}--\eqref{eq:connected-hypergraph}, and let $g\in\interior(\mathscr C)$ and $\theta\in \R^n$ be nonconstant. By connectedness \eqref{eq:connected-hypergraph}, there exists a block $B$ with $w_B>0$ on which $\theta$ is nonconstant; thus,
	$
		\var_{\pi_B}(\theta)>0$.
	As by \eqref{eq:w-n} every block with positive $w$-weight is a nonempty proper subset of $[n]$,  $\mathbf1_B\in \R^n$ is nonconstant; thus, $g(\mathbf1_B)>0$.
Consequently, the $B$-summand in \eqref{eq:hidden-generator-difference} is strictly positive, while all the others are nonnegative. Therefore,
	$\tttonde{\mathscr L_{\tau_1}-\mathscr L_{\tau_2}}g(\theta)>0$
	for every nonconstant $\theta$, proving  \eqref{eq:hidden-generator-cone-monotonicity-int}.
\end{proof}

We now complete the proof of monotonicity of $\lambda_\tau=\gap_{2,\ast}(\tau\alpha,w)$ in \eqref{eq:lambda-tau}.

\begin{proof}[Proof of Proposition \ref{pr:monotonicity}]
	By the assumed hypergraph minimality and Proposition~\ref{pr:PF-strong}, for every $\tau>0$ there exist right and left Perron eigenfunctions
	\begin{equation}
		g_\tau\in\interior(\mathscr C)
		\qquad\text{and}\qquad
		h_\tau\in\interior(\mathscr C')
	\end{equation}
	of $-\mathscr L_\tau$, associated with the eigenvalue $\lambda_\tau>0$.

	Fix $0<\tau_1<\tau_2<\infty$. By linearity and the eigenvalue equations in \eqref{eq:PF-right-left},
	\begin{equation}\label{eq:mono-3}
		h_{\tau_2}\tttonde{
			\tonde{\mathscr L_{\tau_1}-\mathscr L_{\tau_2}}g_{\tau_1}
		}
		=
		-\lambda_{\tau_1}h_{\tau_2}(g_{\tau_1})
		+
		\lambda_{\tau_2}h_{\tau_2}(g_{\tau_1})
		=
		\tonde{\lambda_{\tau_2}-\lambda_{\tau_1}}
		h_{\tau_2}(g_{\tau_1})\fstop
	\end{equation}
	Since $g_{\tau_1}\in\interior(\mathscr C)$, the  inclusion in \eqref{eq:hidden-generator-cone-monotonicity-int} yields
	\begin{equation}
		\tonde{\mathscr L_{\tau_1}-\mathscr L_{\tau_2}}g_{\tau_1}
		\in\interior(\mathscr C)\fstop
	\end{equation}
	As $h_{\tau_2}\in\interior(\mathscr C')$, one further gets
	\begin{equation}
		h_{\tau_2}(g_{\tau_1})>0
		\qquad\text{and}\qquad
		h_{\tau_2}\tttonde{
			\tonde{\mathscr L_{\tau_1}-\mathscr L_{\tau_2}}g_{\tau_1}
		}>0\fstop
	\end{equation}
	Therefore, \eqref{eq:mono-3} implies
$
		\lambda_{\tau_2}>\lambda_{\tau_1}
$,
	as claimed. 
	
Finally, $\tau\longmapsto\lambda_\tau$ is continuous. Indeed, \eqref{eq:hidden-generator-difference} shows that
$\tau\longmapsto\mathscr L_\tau|_{\mathscr R}$
is continuous in operator norm.
 Since $\mathscr R$ is finite-dimensional, this proves the claim.
\end{proof}

\begin{remark}
	Proposition \ref{pr:monotonicity} shows that hypergraph minimality---or, in view of Propositions \ref{pr:reduction}--\ref{pr:reduction-gap}, conditions \eqref{eq:w-n}--\eqref{eq:connected-hypergraph}---ensures the \emph{strict} monotonicity of $\tau\mapsto\lambda_\tau$. Without these geometric assumptions, strict monotonicity may fail (take, e.g., $w_B=\car_{|B|=n}$, in which case  $\gap_{k,\ast}(\alpha,w)=1$ for every $k\ge 1$ and  $\alpha$), although the map remains \textit{nondecreasing}. Since this weaker statement will not be needed, we leave its verification to the reader.
\end{remark}

\begin{remark}[Directional \emph{vs.}\ componentwise monotonicity]
	The monotonicity in Proposition \ref{pr:monotonicity} is \textit{directional}: it holds along the rays $\tau\longmapsto\tau\alpha$, but not under arbitrary componentwise increases of the site weights. Indeed, let $n=3$ and $w_B=\car_{|B|=2}$.
	For
	$
		\alpha=(1,1,1)$
		and $ 
		\alpha'=(s,s,1)$, 
	$s\ge 1$, 
	one has $\alpha\le\alpha'$ componentwise, whereas
	\begin{equation}
		\gap_{2,\ast}(\alpha,w)
		=
		\frac49
	\comma\qquad
		\gap_{2,\ast}(\alpha',w)
		=
		\frac{
			6s^2+19s+8
			-
			\sqrt{(2s+1)(2s^3+s^2+8s+16)}
		}{
			6(s+2)(2s+1)
		}\fstop
	\end{equation}
	Taking $\lim_{s\to \infty}\gap_{2,\ast}(\alpha',w)=\frac13<\frac49$, thus, shows that increasing the site weights componentwise may decrease the two-particle gap.
\end{remark}

Proposition \ref{pr:monotonicity} and Remark \ref{rem:monotonicity-gap-1} establish Theorem \ref{th:phase-transition} for minimal hypergraphs, while Propositions \ref{pr:reduction}--\ref{pr:reduction-gap} lifts this result to the original setting of Theorem \ref{th:phase-transition}, thereby recovering the statement in its full form.
 Consequently, there exists a unique threshold $\tau_c(\alpha,w)\in[0,\infty]$ separating the regime in which $\gap_{2,\ast}(\tau\alpha,w)$ is smaller than $\gap_1(\tau\alpha,w)$ from that in which the reverse inequality holds. To characterize the extremal cases $\tau_c(\alpha,w)\in\{0,\infty\}$, in Sections \ref{sec:tau-infty} and \ref{sec:tau-0} we analyze the gaps of the limiting dynamics obtained as $\tau\to\infty$ and $\tau\to0$, respectively. We begin with the former.

\subsection{The limit $\tau\to\infty$}\label{sec:tau-infty}

To study the extremal regime $\tau\to\infty$, it is convenient to work directly with the limiting hidden-model dynamics. Indeed, for every nonempty block $B\subseteq[n]$,
\begin{equation}
	{\rm Dir}(\tau\alpha^B)
	\Longrightarrow
	\delta_{\pi_B}
	\qquad\text{as}\ \tau\to\infty\comma
\end{equation}
where $\pi_B$ is the $B$-conditional measure on $[n]$ given in \eqref{eq:var-pi-conditional-pi}.
Thus, each random Dirichlet redistribution degenerates into the deterministic weighted averaging of the coordinates in $B$. This defines a Markov dynamics also at $\tau=\infty$, namely, the hidden model associated with a block variant of the averaging process. The latter process, originally introduced by Aldous, has received considerable attention in recent years; see, e.g.,
\cite{aldous_lecture_2012,quattropani2021mixing}.

We first introduce this limiting hidden dynamics and relate its generator to $\mathscr L_\tau$ by operator-norm convergence. We then compute its action on the spatial variance, obtaining the limiting lower bound
$\lambda_\infty\ge\gap_1(\alpha,w)$, together with the quantitative finite-$\tau$ estimate from Theorem \ref{th:gap-1-2-quantitative}. Finally, we characterize the geometries for which equality holds and establish the strict upper bound
$\lambda_\infty<2\gap_1(\alpha,w)$, including the sharpness of the constant $2$.

\subsubsection{The limiting hidden model}

Recall that $\pi_x=\alpha_x/\alpha_0$ and $\pi(B)=\alpha(B)/\alpha_0$. For every nonempty $B\subseteq[n]$, the map $K_B^\pi$ replaces all coordinates in $B$ by their weighted average $\pi_B(\theta)$ and leaves the remaining coordinates unchanged. Define, for every $g\in\mathscr R=\mathscr R_{2,\ast}$,
\begin{equation}\label{eq:hidden-generator-infty}
	\mathscr L_\infty g(\theta)
	\eqdef
	\sum_{B\subseteq[n]}w_B
	\left(
	g(K_B^\pi\theta)-g(\theta)
	\right)\comma
	\qquad \theta\in\mathbb R^n\fstop
\end{equation}
Thus, $\mathscr L_\infty$ is the generator of the hidden model of the block $\pi$-averaging process; see also Section \ref{sec:SEM-AVG} below.

Next, we rigorously relate $\mathscr L_\infty$ to the hidden-model generators $\mathscr L_\tau$ from Section \ref{sec:monotonicity}.	

\begin{lemma}\label{lem:diff-generators}
	For every $\tau>0$ and $g\in\mathscr R$,
	\begin{equation}\label{eq:diff-generators-infty}
		\tonde{\mathscr L_\tau-\mathscr L_\infty}g
		=
		\sum_{B\subseteq[n]}
		\frac{w_B}{1+\tau\alpha(B)}\,
		g(\mathbf1_B)\,
		\var_{\pi_B}\fstop
	\end{equation}
	Consequently,
	\begin{equation}
		\mathscr L_\tau|_{\mathscr R}
		\longrightarrow
		\mathscr L_\infty|_{\mathscr R}
		\qquad\text{in operator norm as }\tau\to\infty\fstop
	\end{equation}
\end{lemma}

\begin{proof}
	This follows from the same computation as in \eqref{eq:hidden-generator-difference} by taking $\tau_1=\tau$ and letting $\tau_2\to\infty$. Passing to the limit, the operator-norm convergence follows since $\mathscr R$ is finite-dimensional.
\end{proof}

Finite-dimensional spectral continuity and Proposition \ref{pr:monotonicity} therefore give (recall \eqref{eq:lambda-tau})
\begin{equation}\label{eq:lambda-infty}
\spec(-\mathscr L_\infty|_{\mathscr R})\subseteq(0,\infty)\comma\qquad	\lambda_\infty
	\eqdef \sup_{\tau>0}\lambda_\tau =
	\lim_{\tau\to\infty}\lambda_\tau
	=
	\min\spec\tttonde{-\mathscr L_\infty|_{\mathscr R}}\fstop
\end{equation}

\subsubsection{Variance dissipation and gap lower bounds}\label{sec:var-dissipation-gap-lb}

We next compute the dissipation of the spatial variance $\var_\pi \in \interior(\mathscr C)$. Recall \eqref{eq:var-pi-intro} and \eqref{eq:var-pi-conditional-pi}--\eqref{eq:var-pi-conditional-var}, and that the Dirichlet form of the one-particle dynamics with rates \eqref{eq:RW-rates} reads, for every $\theta \in \R^n$, as
\begin{equation}\label{eq:dirichlet-form-1}\textstyle
	\cE_1(\theta)
	\eqdef \ttscalar{\theta}{-L_1\theta}_\pi =
	\sum_{B\subseteq[n]}
	w_B\,\pi(B)\,\var_{\pi_B}(\theta)\fstop
\end{equation}

\begin{lemma}\label{lem:variance-infty}
	For every $\theta\in\R^n$,
	\begin{equation}\label{eq:L-infty-var}
		\mathscr L_\infty\var_\pi(\theta)
		=
		-\cE_1(\theta)\fstop
	\end{equation}
\end{lemma}

\begin{proof}
	The formula for total variance gives, for every $B\subseteq[n]$,
	\begin{equation}
		\var_\pi(K_B^\pi\theta)-\var_\pi(\theta)
		=
		-\pi(B)\var_{\pi_B}(\theta)\fstop
	\end{equation}
	After recalling \eqref{eq:hidden-generator-infty} and \eqref{eq:dirichlet-form-1}, 
	summing with weights $w_B$ proves the claim.
\end{proof}

Next, we combine Lemmas \ref{lem:diff-generators} and \ref{lem:variance-infty} to prove both the lower bound in Theorem \ref{th:gap-1-2-quantitative} and its limiting counterpart, namely, the second inequality in \eqref{eq:limits}.
Recall (cf.\ Theorem \ref{th:gap-1-2-quantitative}) \begin{equation}\label{eq:alpha-min}\alpha_{w,\rm min}\eqdef \min_{B\,:\,w_B>0} \alpha(B)\fstop
\end{equation}

\begin{corollary}\label{cor:gap-lower-infty}
	For every $\tau>0$,
	\begin{equation}\label{eq:gap-lower-tau}
		\frac{\tau\alpha_{w,{\rm min}}}
		{1+\tau\alpha_{w,{\rm min}}}
		\left(1+\frac1{\tau\alpha_0}\right)
		\gap_1(\alpha,w)\le \lambda_\tau
		\fstop
	\end{equation}
	Consequently, taking $\tau\to \infty$ gives
	\begin{equation}\label{eq:gap-lower-infty}
		\gap_1(\alpha,w)\le\lambda_\infty = \lim_{\tau\to \infty}\gap_{2,\ast}(\tau\alpha,w)\fstop
	\end{equation}
\end{corollary}

\begin{proof}
	Since
	$
	\var_\pi(\mathbf1_B)=\pi(B)(1-\pi(B))
	$,
	Lemmas \ref{lem:diff-generators} and \ref{lem:variance-infty} give
	\begin{equation}\label{eq:L-tau-var-exact}
		\mathscr L_\tau\var_\pi
		=
		-\sum_{B\subseteq[n]}
		w_B\,\pi(B)\,
		\frac{\tau\alpha(B)+\pi(B)}
		{1+\tau\alpha(B)}\,
		\var_{\pi_B}\fstop
	\end{equation}
	As $\pi(B)=\alpha(B)/\alpha_0$ and $s\mapsto \tau s/(1+\tau s)$ is increasing, recalling \eqref{eq:dirichlet-form-1} and \eqref{eq:alpha-min}, 
	\begin{align}\label{eq:L-tau-var-bound}
		\begin{aligned}
		\mathscr L_\tau\var_\pi
		&\le
		-\frac{\tau\alpha_{w,{\rm min}}}
		{1+\tau\alpha_{w,{\rm min}}}
		\left(1+\frac1{\tau\alpha_0}\right)
		\cE_1
		\\
		&\le
		-\frac{\tau\alpha_{w,{\rm min}}}
		{1+\tau\alpha_{w,{\rm min}}}
		\left(1+\frac1{\tau\alpha_0}\right)
		\gap_1(\alpha,w)\var_\pi\comma
	\end{aligned}	
	\end{align}
	where the last step used the variational characterization of the spectral gap of the reversible one-particle dynamics with rates \eqref{eq:RW-rates}.
	Let $h_\tau\in\mathscr C'\setminus\{0\}$ be a left Perron eigenfunction from Proposition \ref{pr:PF-positive}. Since $\var_\pi\in\interior(\mathscr C)$, one has $h_\tau(\var_\pi)>0$; see \eqref{eq:dual-positive-on-interior}. Applying $h_\tau$ to \eqref{eq:L-tau-var-bound} and using
	$
	h_\tau\mathscr L_\tau=-\lambda_\tau h_\tau
	$
	proves \eqref{eq:gap-lower-tau}.
\end{proof}

\begin{remark}[Sharpness of the lower bound in \eqref{eq:gap-1-2-quantitative}]\label{rem:gap-lower-sharp}
	In every homogeneous mean-field setting \eqref{eq:mean-field}, the bound
	\eqref{eq:gap-lower-tau} is an equality for all $\tau>0$. Indeed, all
	positive-weight blocks have the same $\alpha$-mass
	$\alpha(B)=\alpha_{w,{\rm min}}$, so the first inequality in
	\eqref{eq:L-tau-var-bound} is an equality; moreover, the one-particle
	dynamics is mean-field, and hence
	\begin{equation}
		\cE_1
		=
		\gap_1(\alpha,w)\var_\pi\comma
	\end{equation}
	so the second inequality in \eqref{eq:L-tau-var-bound} is an equality as well. Consequently,
	$\var_\pi$ is an eigenfunction of $-\mathscr L_\tau$ and
	\begin{equation}
		\gap_{2,\ast}(\tau\alpha,w)
		=
		\frac{\tau\alpha_{w,{\rm min}}}
		{1+\tau\alpha_{w,{\rm min}}}
		\left(1+\frac1{\tau\alpha_0}\right)
		\gap_1(\alpha,w)\fstop
	\end{equation}
	In particular, the lower bound in
	\eqref{eq:gap-1-2-quantitative} is sharp. Together with Remark
	\ref{rem:lambda-infty-sharp}, this proves the sharpness of both constants
	in Theorem \ref{th:gap-1-2-quantitative}.
\end{remark}

\subsubsection{Geometries corresponding to $\tau_c(\alpha,w)=\infty$}\label{sec:tau=infty}

We proceed with the characterization of equality in \eqref{eq:gap-lower-infty}. Here, we impose minimality of the underlying hypergraph induced by $w$. Observe that, under minimality, the weights $\widetilde w$ defined in \eqref{eq:weights-tilde} coincide with $w$. The proof of Theorem \ref{th:dichotomy-graph}\ref{it:graph-dichotomy-mean-field} and Theorem \ref{th:dichotomy-hypergraph}\ref{it:hypergraph-dichotomy-mean-field} will then follow from an application of the reduction procedure from Section \ref{sec:hypergraph-reduction}.

\begin{proposition}[Characterization of $\tau_c(\alpha,w)=\infty$]\label{pr:mean-field-infty}
	Assume that $w$ induces a minimal hypergraph. Then,
	\begin{equation}\label{eq:mean-field-infty-equivalence}
		\lambda_\infty=\gap_1(\alpha,w)
		\quad\Longleftrightarrow\quad
		\sum_{B\ni{x,y}}
		\frac{w_B}{\alpha(B)}
		\quad\text{does not depend on}\ x,y\in[n], x\neq y\fstop
	\end{equation}
	Equivalently, $\tau_c(\alpha,w)=\infty$ if and only if $(\alpha,w)$ satisfies the mean-field condition on the right-hand side of \eqref{eq:mean-field-infty-equivalence}.
\end{proposition}

\begin{proof}
	For distinct $x,y\in[n]$, set
	\begin{equation}
	c_{xy}
		\eqdef
		\sum_{B\ni {x,y}}
		\frac{w_B}{\alpha(B)}\fstop
	\end{equation}
	Then, the one-particle Dirichlet form in \eqref{eq:dirichlet-form-1} and the corresponding variance read as
	\begin{equation}
		\cE_1(\theta)
		=
		\frac1{2\alpha_0}
		\sum_{x,y\in [n]}
		\alpha_x\alpha_y\,c_{xy}
		\tonde{\theta_x-\theta_y}^2\comma
		\label{eq:E1-var-pairwise}\qquad
		\var_\pi(\theta)
		=
		\frac1{2\alpha_0^2}
		\sum_{x,y\in [n]}
		\alpha_x\alpha_y
		\tonde{\theta_x-\theta_y}^2\fstop
	\end{equation}
	Thus, the condition on the right-hand side of \eqref{eq:mean-field-infty-equivalence} is equivalent to
	\begin{equation}\label{eq:Poincare-identity}
		\cE_1
		=
		\gap_1(\alpha,w)\var_\pi\fstop
	\end{equation}

	If \eqref{eq:Poincare-identity} holds, Lemma \ref{lem:variance-infty} gives
	\begin{equation}
		\mathscr L_\infty\var_\pi
		=
		-\gap_1(\alpha,w)\var_\pi\fstop
	\end{equation}
	Hence, $\lambda_\infty\le\gap_1(\alpha,w)$, and equality follows from \eqref{eq:gap-lower-infty}.
	
	Conversely, suppose that \eqref{eq:Poincare-identity} fails. By  Lemma~\ref{lem:variance-infty} and the one-particle Poincaré inequality,
	\begin{equation}
		\mathscr D
		\eqdef
		-\mathscr L_\infty\var_\pi
		-\gap_1(\alpha,w)\var_\pi = \cE_1-\gap_1(\alpha,w)\var_\pi
		\in
		\mathscr C\setminus\{0\}\fstop
	\end{equation}
	By Proposition \ref{pr:full-range-infty}, the hidden-model semigroup associated to $\mathscr L_\infty|_{\mathscr R}$ is strongly positive. 
 Hence, as in Proposition \ref{pr:PF-strong}, there exists $h_\infty\in\interior(\mathscr C')$ satisfying
	\begin{equation}\label{eq:h-dual-infty}
		h_\infty\mathscr L_\infty
		=
		-\lambda_\infty h_\infty\fstop
	\end{equation}
	Therefore, recalling that $\mathscr D\in \mathscr C\setminus \{0\}$ and $h_\infty \in \interior(\mathscr C')$,
	\begin{equation}
		0<h_\infty(\mathscr D)
		=
		\tonde{\lambda_\infty-\gap_1(\alpha,w)}
		h_\infty(\var_\pi)\fstop
	\end{equation}
	Since $h_\infty(\var_\pi)>0$, this proves
	$
	\lambda_\infty>\gap_1(\alpha,w)
	$.
	
	Finally, Proposition \ref{pr:monotonicity} and \eqref{eq:lambda-infty} give
	$
	\lambda_\tau<\lambda_\infty
	$
	for every finite $\tau>0$. Hence $\lambda_\infty=\gap_1(\alpha,w)$ precisely when
	$
	\lambda_\tau<\gap_1(\alpha,w)
	$
	for every $\tau>0$, which is equivalent to $\tau_c(\alpha,w)=\infty$.
\end{proof}

\subsubsection{Gap upper bounds} This section is devoted to the proof of the upper bound in Theorem~\ref{th:gap-1-2-quantitative}.
The idea behind it is quite simple: by lifting a one-particle gap-eigenfunction to the hidden model and considering its square, the failure of the usual derivation rule is measured by the carré-du-champ of $\mathscr L_\infty$, which provides the strict correction below $2\gap_1(\alpha,w)$.

Recall \eqref{eq:R-k-ast} and \eqref{eq:tilde-tensor}--\eqref{eq:hat-tilde-diagonal}; in particular, 
for every $\psi:[n]\to \R$ in $\mathscr H_{1,\ast}$,  
\begin{equation}
\textstyle	\widetilde\psi(\theta)
	=
	\scalar{\psi}{\theta}_\pi
	=
	\sum_{x\in[n]}
	\pi_x\,\psi(x)\,\theta_x\comma
	\qquad
	\theta\in\R^n\comma
\end{equation}
belongs to $\mathscr R_{1,\ast}$.
Further, recalling that $\mathscr R=\mathscr R_{2,\ast}$ and writing $(g\circ K_B^\pi)(\theta)=g(K_B^\pi \theta)$, $\theta \in \R^n$,
\begin{equation}\label{eq:carre-du-champ-infty}
	\textstyle
\Gamma_\infty(g)
	\eqdef
		\mathscr L_\infty(g^2)
		-
		2g\mathscr L_\infty g
	=
	\sum_{B\subseteq[n]}
	w_B
	\left(
	g\circ K_B^\pi-g
	\right)^2\comma\qquad g \in \mathscr R\comma
\end{equation}
denotes the carré-du-champ associated with $\mathscr L_\infty$.  Clearly, $\Gamma_\infty\, \mathscr R_{1,\ast}\subseteq \mathscr C$.
\begin{proposition}\label{pr:lambda-infty-upper}
	If the hypergraph induced by $w$ is minimal, then
	\begin{equation}\label{eq:lambda-infty-comparison}
		\lambda_\infty
		<
		2\gap_1(\alpha,w)\fstop
	\end{equation}
	Consequently, by \eqref{eq:lambda-tau} and \eqref{eq:lambda-infty}, 
	\begin{equation}\label{eq:gap-upper-strict-minimal}
		\gap_{2,\ast}(\tau\alpha,w)
		<
		2\gap_1(\alpha,w)\comma\qquad \tau>0\fstop
	\end{equation}
\end{proposition}

\begin{proof}
	Let $\phi\in\mathscr H_{1,\ast}$ be a one-particle gap-eigenfunction, so that
	\begin{equation}
		\pi(\phi)=0
		\qquad\text{and}\qquad
		L_1\phi
		=
		-\gap_1(\alpha,w)\phi\fstop
	\end{equation}
Hence, by an immediate extension of Proposition \ref{pr:particle-system-intertwining} to the case $\tau=\infty$,
	\begin{equation}\label{eq:L-infty-linear-eigenfunction}
		\mathscr L_\infty\widetilde\phi
		=
		\widetilde{L_1\phi}
		=
		-\gap_1(\alpha,w)\widetilde\phi\fstop
	\end{equation}
	Combining the definition in \eqref{eq:carre-du-champ-infty} with
	\eqref{eq:L-infty-linear-eigenfunction} gives
	\begin{equation}\label{eq:L-infty-square}
		-\mathscr L_\infty((\widetilde\phi)^2)
		=
		2\gap_1(\alpha,w)(\widetilde\phi)^2
		-
		\Gamma_\infty(\widetilde\phi)\fstop
	\end{equation}

	Both $(\widetilde\phi)^2$ and
	$\Gamma_\infty(\widetilde\phi)$ belong to $\mathscr C\subseteq \mathscr R$. Moreover,
	\begin{equation}
		(\widetilde\phi)^2\neq0
		\qquad\text{and}\qquad
		\Gamma_\infty(\widetilde\phi)\neq0\fstop
	\end{equation}
	Indeed, if $\Gamma_\infty(\widetilde\phi)=0$, then
	\begin{equation}
		\widetilde\phi\circ K_B^\pi
		=
		\widetilde\phi
		\qquad
		\text{for every $B\subseteq[n]$ with $w_B>0$}\comma
	\end{equation}
	and therefore $\mathscr L_\infty\widetilde\phi=0$. This contradicts
	\eqref{eq:L-infty-linear-eigenfunction}, because minimality implies
	$\gap_1(\alpha,w)>0$.
	
	By Proposition \ref{pr:full-range-infty}, there exists a left Perron eigenfunction
	$h_\infty\in\interior(\mathscr C')$ of $\mathscr L_\infty$, see \eqref{eq:h-dual-infty}. 
	Applying $h_\infty$ to \eqref{eq:L-infty-square} yields
	\begin{equation}
		\lambda_\infty
		h_\infty((\widetilde\phi)^2)
		=
		2\gap_1(\alpha,w)
		h_\infty((\widetilde\phi)^2)
		-
		h_\infty(\Gamma_\infty(\widetilde\phi))\fstop
	\end{equation}
	Since $h_\infty\in\interior(\mathscr C')$ is strictly positive on
	$\mathscr C\setminus\{0\}$,  one has
	$
		h_\infty((\widetilde\phi)^2), 
		h_\infty(\Gamma_\infty(\widetilde\phi))>0$.
	Dividing by
	$h_\infty((\widetilde\phi)^2)$ proves
	\eqref{eq:lambda-infty-comparison}.
\end{proof}

\begin{remark}[Extension to non-minimal hypergraphs]\label{rem:lambda-infty-upper-general}
	The conclusion of Proposition \ref{pr:lambda-infty-upper} is stated for minimal hypergraphs, hence, by Propositions \ref{pr:reduction}--\ref{pr:reduction-gap}, under assumptions \eqref{eq:w-n}--\eqref{eq:connected-hypergraph}. The general upper bound in Theorem \ref{th:gap-1-2-quantitative} follows by treating separately the two possible failures of these conditions.

If \eqref{eq:connected-hypergraph} fails, then
	$\lambda_\tau
			=
			\gap_1(\alpha,w)
			=
			0$; see, e.g., Remark \ref{rem:positivity-gap}.
		Thus, \eqref{eq:gap-upper-strict-minimal} remains valid with the strict inequality replaced by equality.
		 Now suppose that \eqref{eq:w-n} fails. By \eqref{eq:global-shift-gap-k-ast} and passing to the limit $\tau\to\infty$, we obtain
		\begin{equation}
			\lambda_\infty(\alpha,w) = \lambda_\infty(\alpha,w^{\scriptscriptstyle \circ})+w_{[n]}
			\le
			2\gap_1(\alpha,w^{\scriptscriptstyle\circ})
			+
			w_{[n]}
		<
			2\gap_1(\alpha,w^{\scriptscriptstyle\circ})
			+
			2w_{[n]}
			=
			2\gap_1(\alpha,w)\comma
		\end{equation}
		where for the second step one applies   either Proposition \ref{pr:lambda-infty-upper} if $w^{\scriptscriptstyle \circ}\eqdef\car_{B\neq [n]}\,w$ induces a connected hypergraph,  or the preceding disconnected case otherwise.

	Thus, the upper bound is strict whenever the hypergraph is connected, while equality occurs in the disconnected case, where both sides vanish.
\end{remark}

The constant $2$ in Proposition \ref{pr:lambda-infty-upper} cannot be improved, although equality is never attained in the nondegenerate case.

\begin{remark}[Sharpness of the upper bound in \eqref{eq:gap-1-2-quantitative}]\label{rem:lambda-infty-sharp}
	For every $\varepsilon>0$, there exist weights $\alpha$ and $w$ such that
	\begin{equation}
		\gap_{2,\ast}(\alpha,w)
		>
	\tonde{2-\varepsilon}\gap_1(\alpha,w)\fstop
	\end{equation}
	Hence, the factor $2$ in the upper bound of Theorem \ref{th:gap-1-2-quantitative} is sharp.

	A concrete example is the three-site path with
	\begin{equation}
\alpha=(1,s,1)\comma
 s>0\comma\qquad		w_B = \begin{dcases}
			1 &\text{if}\ B=\{1,2\}, \{2,3\}\\
			0 &\text{else}\fstop
		\end{dcases}
	\end{equation}
	Writing $c_s=(1+s)^{-1}$, a direct computation gives
	\begin{equation}
		\gap_1(\alpha,w)=1-c_s
		\comma\qquad
		\lambda_\infty
		=
		\lim_{\tau\to\infty}\gap_{2,\ast}(\tau\alpha,w)
		=
		1-c_s^2\fstop
	\end{equation}
	Thus, $\lambda_\infty/\gap_1=1+c_s\to2$ as $s\to 0$. Choosing first $s=s(\varepsilon)>0$ sufficiently small and then, by Proposition \ref{pr:monotonicity}, $\tau=\tau(\varepsilon)>0$ sufficiently large proves the claim.
\end{remark}

\subsection{The limit $\tau\to0$}\label{sec:tau-0}
We now turn to the regime $\tau\to0$, with the goal of proving the first inequality in \eqref{eq:limits} and characterizing the geometries for which $\tau_c(\alpha,w)=0$.

As $\tau\to0$, the Dirichlet redistribution law concentrates on the vertices of the simplex: for every nonempty $B\subseteq[n]$ with $w_B>0$,
\begin{equation}\label{eq:dirichlet-zero-limit}
	{\rm Dir}(\tau\alpha^B)
	\Longrightarrow
	\sum_{z\in B}\pi_{z|B}\,e_{e_z}
	\qquad\text{as }\tau\to0\comma
\end{equation}
where $\pi_{z|B}$ is defined in \eqref{eq:var-pi-conditional-pi} and $(e_z)_{z\in B}$ are the vertices of the simplex on $B$. Consequently, in the limiting hidden-model dynamics, when the clock of $B$ rings, a site $Z\in B$ is sampled with law $\pi_B$, all coordinates in $B$ are replaced by the common value $\theta_Z$, and the coordinates outside $B$ are left unchanged.

The corresponding particle dynamics admits an equally simple description: all particles lying in the updated block are sent to the same sampled site $Z$. Thus, once two particles meet, they subsequently move together as a single particle, effectively reducing the system to one with fewer particles.  When one restricts attention to a fixed particle number, such mergers are instead recorded as killing: the process is stopped as soon as two particles coalesce.

We shall use both viewpoints. The hidden-model description gives a direct passage from the finite-$\tau$ setup to the limiting dynamics, whereas the particle-system description, through a direct intertwining at $\tau=0$, identifies the limiting eigenvalue with the principal Dirichlet eigenvalue of a killed two-particle chain.

\subsubsection{The limiting hidden model}

We now introduce the candidate limiting hidden-model dynamics obtained by letting $\tau\to0$ in the generators $\mathscr L_\tau$, and make this convergence precise in Lemma \ref{lem:diff-generators-zero} below.

For every $g\in\mathscr R=\mathscr R_{2,\ast}$, define
\begin{equation}\label{eq:hidden-generator-zero}
	\mathscr L_0g(\theta)
	\eqdef
	\sum_{B\subseteq[n]}w_B
	\sum_{z\in B}\pi_{z|B}
	\left(
	g(K_B^{e_z}\theta)-g(\theta)
	\right)
	\comma
	\qquad \theta\in\R^n\fstop
\end{equation}

\begin{lemma}\label{lem:diff-generators-zero}
	For every $\tau>0$ and $g\in\mathscr R$,
	\begin{equation}\label{eq:diff-generators-zero}
		\tonde{\mathscr L_0-\mathscr L_\tau}g
		=
		\sum_{B\subseteq[n]}
		w_B\,
		\frac{\tau\alpha(B)}{1+\tau\alpha(B)}\,
		g(\mathbf1_B)\,
		\var_{\pi_B}\fstop
	\end{equation}
	Consequently,
	\begin{equation}
		\mathscr L_\tau|_{\mathscr R}
		\longrightarrow
		\mathscr L_0|_{\mathscr R}
		\qquad\text{in operator norm as }\tau\to0\fstop
	\end{equation}
\end{lemma}

\begin{proof}The first identity follows directly from \eqref{eq:hidden-generator-difference}. Since $\mathscr R$ is finite-dimensional, letting $\tau\to0$ in that identity also yields convergence in operator norm.
\end{proof}

Finite-dimensional spectral continuity and Proposition \ref{pr:monotonicity} now yield
\begin{equation}\label{eq:lambda-zero}
	\spec(-\mathscr L_0|_{\mathscr R})\subseteq[0,\infty)
	\comma\qquad
	\lambda_0
	\eqdef
	\inf_{\tau>0}\lambda_\tau
	=
	\lim_{\tau\to0}\lambda_\tau
	=
	\min\spec\tttonde{-\mathscr L_0|_{\mathscr R}}\fstop
\end{equation}

\subsubsection{The killed two-particle system}

We next turn the limiting particle description into an exact algebraic statement, by intertwining the hidden-model generator at $\tau=0$ with a killed two-particle generator.
Set
\begin{equation}
	[n]^2_\tneq
	\eqdef
	\tset{(x,y)\in[n]^2:x\neq y}\fstop
\end{equation}
For every block $B\subseteq[n]$, define the killed two-particle transition operator $\Pi_{B,2}^\dagger$ on functions $\psi:[n]^2_\tneq\to\R$ by
\begin{equation}\label{eq:Pi-B-two-dagger}
	\Pi_{B,2}^\dagger\psi(x,y)
	\eqdef
	\begin{dcases}\textstyle
		\psi(x,y)
		&\text{if}\ x,y\notin B\\ 
		\textstyle
		\Pi_{B,1}\psi(\emparg,y)(x)=
		\sum_{z\in B}\pi_{z|B}\,\psi(z,y)
		&\text{if}\ x\in B, y\notin B\\
		\textstyle
		\Pi_{B,1}\psi(x,\emparg)(y)=\sum_{z\in B}\pi_{z|B}\,\psi(x,z)
		&\text{if}\ x\notin B, y\in B\\
		0
		&\text{if}\ x,y\in B\fstop
	\end{dcases}
\end{equation}
Accordingly, in analogy with \eqref{eq:gen-particle}, define
\begin{equation}\label{eq:L-two-dagger}
	L_2^\dagger
	\eqdef
	\sum_{B\subseteq[n]}
	w_B\,\big(\Pi_{B,2}^\dagger-I\big)\fstop
\end{equation}
Before killing, each particle marginally follows the one-particle dynamics, with each block $B$ ringing at rate $w_B$ and relocating any particle it contains according to $\pi_B$, until the first update involving both particles, at which time they merge and the two-particle process is killed. \begin{remark}Even when $w$ induces a connected hypergraph, this killed dynamics need not be irreducible on $[n]^2_\tneq$; the form and number of its communicating classes depend on the underlying geometry and will play a central role in Proposition \ref{pr:segment-zero} below.
\end{remark}

Next, we present an intertwining relation involving $\mathscr L_0$ and $L_2^\dagger$. In what follows, let 
\begin{equation}\label{eq:Hneq}
\mathscr H_\tneq^\per=	\mathscr H_{2,\ast,\tneq}^\per \eqdef\{\psi:[n]_\tneq^2\to \R\}\comma\qquad 
\mathscr H_\tneq=	\mathscr H_{2,\ast,\tneq}\eqdef\set{\psi \in \mathscr H_\tneq^\per: \text{symmetric}}\fstop
\end{equation}

\begin{proposition}[Intertwining at $\tau=0$]\label{pr:intertwining-zero}
	The map
	$\Lambda:
		\mathscr H_\tneq
		\to
		\mathscr R$,
	given by
	\begin{equation}\label{eq:Lambda-two-dagger}
		\Lambda\psi(\theta)
		\eqdef
			\frac12
		\sum_{(x,y)\in[n]^2_\tneq}
		\pi_x\pi_y\,
		\psi(x,y)
		\tonde{\theta_x-\theta_y}^2
		\comma
		\qquad \theta\in\R^n\comma \psi \in \mathscr H_\tneq\comma
	\end{equation}
	is a linear isomorphism and satisfies
	\begin{equation}\label{eq:intertwining-zero}
		\mathscr L_0\Lambda
		=
		\Lambda L_2^\dagger
		\qquad\text{on}\ \mathscr H_\tneq\fstop
	\end{equation}
\end{proposition}
\begin{proof}
	For $\theta\in\R^n$, set
	\begin{equation}
		d_\theta(x,y)
		\eqdef
		\tonde{\theta_x-\theta_y}^2
		\comma
		\qquad (x,y)\in[n]^2_\tneq\fstop
	\end{equation}
	For every block $B\subseteq[n]$, the definitions of the two limiting updates give
	\begin{equation}\label{eq:block-duality-zero}
		\sum_{z\in B}\pi_{z|B}\,
		d_{K_B^{e_z}\theta}(x,y)
		=
		\Pi_{B,2}^\dagger d_\theta(x,y)
		\comma
		\qquad (x,y)\in[n]^2_\tneq\fstop
	\end{equation}
	Indeed, if exactly one of $x,y$ belongs to $B$, both sides average the corresponding coordinate according to $\pi_B$; if both belong to $B$, both sides vanish.

	Moreover, $\Pi_{B,2}^\dagger$ is self-adjoint in
	$L^2([n]^2_\tneq,\pi\otimes\pi)$. For instance, if $x,z\in B$ and $y\notin B$, this follows from
$
		\pi_x\pi_y\,\pi_{z|B}
		=
		\pi_z\pi_y\,\pi_{x|B}$; the other cases are analogous. Hence, by \eqref{eq:block-duality-zero},
	\begin{equation}		\label{eq:block-intertwining-zero}
		\sum_{z\in B}\pi_{z|B}\,
		\Lambda\psi(K_B^{e_z}\theta)
		=
		\frac12
		\ttscalar{\psi}{\Pi_{B,2}^\dagger d_\theta}_{\pi\otimes\pi}
		=
		\frac12
		\ttscalar{\Pi_{B,2}^\dagger\psi}{d_\theta}_{\pi\otimes\pi}
		=
		\Lambda\tttonde{\Pi_{B,2}^\dagger\psi}(\theta)\fstop
	\end{equation}
	Subtracting $\Lambda\psi(\theta)$, multiplying by $w_B$, and summing over $B$ proves \eqref{eq:intertwining-zero}.
	
	It remains to prove that $\Lambda:\mathscr H_\tneq\to \mathscr R$ is an isomorphism. Its image consists of homogeneous quadratic polynomials invariant under translations $\theta\mapsto\theta+s\mathbf1$, and is therefore contained in $\mathscr R$ (Proposition \ref{pr:Rk-star}). If $\Lambda\psi=0$, comparison of the coefficient of $\theta_x\theta_y$, for $x\neq y$, gives
	\begin{equation}
		-2\pi_x\pi_y\psi(x,y)=0\fstop
	\end{equation}
	Thus, $\psi=0$, so $\Lambda$ is injective. Since
$
		\dim\mathscr H_\tneq
		=
		\binom n2
		=
		\dim\mathscr R$,
	it is also surjective.
\end{proof}

Since every $\Pi_{B,2}^\dagger$ is self-adjoint in
$L^2([n]^2_\tneq,\pi\otimes\pi)$, the same holds for $L_2^\dagger$. The intertwining in Proposition \ref{pr:intertwining-zero} and \eqref{eq:lambda-zero} therefore imply
\begin{equation}\label{eq:lambda-zero-killed-symmetric}
	\lambda_0
	=
	\min\spec
	(
	-L_2^\dagger|_{\mathscr H_\tneq}
	)\fstop
\end{equation}
In fact, the restriction to symmetric functions may be removed: recalling \eqref{eq:Hneq},
\begin{equation}\label{eq:lambda-zero-killed}
	\lambda_0
	=
	\min\spec(-L_2^\dagger|_{\mathscr H_\tneq^\per})\fstop
\end{equation}
Indeed, transposition $(x,y)\mapsto(y,x)$ commutes with $L_2^\dagger$. Consider a communicating class attaining the smallest principal Dirichlet eigenvalue. If the class is invariant under transposition, uniqueness of its positive principal eigenfunction makes this eigenfunction symmetric. Otherwise, the class and its transpose are distinct, and the sum of their positive principal eigenfunctions is now a symmetric eigenfunction. Thus, the smallest Dirichlet eigenvalue is always attained in $\mathscr H_\tneq$. 

This observation will be crucial below; see, e.g., \eqref{eq:lambda-0-lambda}. Indeed, we shall mainly consider  antisymmetric $L_2^\dagger$-eigenfunctions of the form
\begin{equation}
	\psi(x,y)
	=
	\phi(x)-\phi(y)\comma
	\qquad
	(x,y)\in[n]_\tneq^2\comma
\end{equation}
as
introduced in \eqref{eq:Dphi}, which inherit their eigenvalues from one-particle eigenfunctions $\phi$. A central part of the analysis will consist in characterizing the geometries for which such a $\psi$ has a fixed sign on each communicating class of the killed two-particle dynamics; in that case, $|\psi|$ is a symmetric, nonnegative eigenfunction with the same eigenvalue.

We may now complete the limiting comparison announced in Section \ref{sec:comparison-intro}.

\begin{proposition}\label{pr:first-limit}
	For all positive site weights $\alpha$ and nonnegative block weights $w$,
	\begin{equation}\label{eq:first-limit}
	\lambda_0=	\lim_{\tau\to0}
		\gap_{2,\ast}(\tau\alpha,w)
		\le
		\gap_1(\alpha,w)\fstop
	\end{equation}
\end{proposition}

\begin{proof}
	Let $0\neq \phi\in \mathscr H_{1,\ast}$ be a nonconstant one-particle eigenfunction satisfying
	$
		L_1\phi
		=
		-\lambda\phi$, for some $\lambda \ge 0$.
	Define $\cD\phi\in \mathscr H_\tneq^\per$ as the anti-symmetric function given by
	\begin{equation}\label{eq:Dphi}
		\cD\phi(x,y)
		\eqdef
		\phi(x)-\phi(y)
		\comma
		\qquad (x,y)\in[n]^2_\tneq\fstop
	\end{equation}
	For every block $B\subseteq[n]$, a direct verification gives, for every $\phi\in \mathscr H_1$,
	\begin{equation}
		\Pi_{B,2}^\dagger \cD\phi
		=
		\cD\Pi_{B,1}\phi\fstop
	\end{equation}
	Indeed, when both particles belong to $B$, both sides vanish, since $\Pi_{B,1}\phi$ is constant on $B$. Consequently,
	\begin{equation}\label{eq:eigen-asym}
		L_2^\dagger \cD\phi
		=
		\cD L_1\phi
		=
		-\lambda\cD\phi\fstop
	\end{equation}
	Since $\phi$ is nonconstant, $\cD\phi\neq0$. Hence, \eqref{eq:eigen-asym} shows that $\cD\phi\in \mathscr H_\tneq^\per$ is an anti-symmetric eigenfunction of $L_2^\dagger$ with eigenvalue $\lambda\ge 0$.  By \eqref{eq:lambda-zero-killed-symmetric}--\eqref{eq:lambda-zero-killed}, 
	\begin{equation}\label{eq:lambda-0-lambda}
		\lambda_0
		\le
		\lambda\fstop
	\end{equation}
	The claim now follows from
	$
		\lambda_0
		=
		\lim_{\tau\to0}
		\gap_{2,\ast}(\tau\alpha,w)
	$,
	as established in \eqref{eq:lambda-zero}, by taking $\lambda = \gap_1(\alpha,w)$.
\end{proof}

\subsubsection{Geometries corresponding to $\tau_c(\alpha,w)=0$}

The preceding section identifies $\lambda_0$ with the principal Dirichlet eigenvalue of the killed two-particle system. We now characterize equality in \eqref{eq:first-limit}, which, by strict monotonicity, is precisely the regime $\tau_c(\alpha,w)=0$.

\begin{proposition}[Characterization of $\tau_c(\alpha,w)=0$]\label{pr:segment-zero}
	Assume that $w$ induces a minimal hypergraph. Then, for every positive $\alpha$, the following are equivalent:
	\begin{enumerate}[(i)]
		\item $\tau_c(\alpha,w)=0$;
		\item $\lambda_0=\gap_1(\alpha,w)$;
		\item after a relabeling of $[n]$, every $B\subseteq[n]$ with $w_B>0$ and $|B|\ge2$ is an interval.
	\end{enumerate}
	Whenever these conditions hold,
	\begin{equation}\label{eq:lambda-zero-gap-one}
		\lambda_0
		=
		\gap_1(\alpha,w)
		<
		\lambda_\tau
		\comma
		\qquad \tau>0\fstop
	\end{equation}
\end{proposition}

\begin{proof}
	By Proposition \ref{pr:first-limit} and the strict monotonicity in Proposition \ref{pr:monotonicity},
	\begin{equation}\label{eq:tau-zero-lambda-zero}
		\tau_c(\alpha,w)=0
		\quad\Longleftrightarrow\quad
		\lambda_0=\gap_1(\alpha,w)\fstop
	\end{equation}

	We next prove that equality in \eqref{eq:tau-zero-lambda-zero} forces the interval geometry. Let $\phi$ be a nonconstant one-particle eigenfunction associated to $\lambda=\gap_1(\alpha,w)$, and let
	$a_1<a_2<\cdots<a_m$
	be its distinct values, with level sets
	\begin{equation}\label{eq:gap-level-sets-zero}
		N_j
		\eqdef
		\tset{x\in[n]:\phi(x)=a_j}
		\comma
		\qquad j=1,\ldots,m\fstop
	\end{equation}
	Since $\phi$ is nonconstant, $m\ge2$. By the intertwining used in the proof of Proposition \ref{pr:first-limit},
	\begin{equation}
		L_2^\dagger \cD\phi
		=
		-\gap_1(\alpha,w)\cD\phi
		=
		-\lambda_0\cD\phi\fstop
	\end{equation}
	Therefore, by the Perron--Frobenius discussion following \eqref{eq:lambda-zero-killed}, the restriction of $\cD\phi$ to each communicating class of the killed chain is either identically zero or has a strict sign.
	
	Fix $B\subseteq[n]$ with $w_B>0$. Suppose first that $B$ meets $N_i$ and $N_k$, with $i<k$, but misses some $N_j$, $i<j<k$. Choose
	\begin{equation}
		x\in B\cap N_i
		\comma\qquad
		y\in N_j\setminus B
		\comma\qquad
		z\in B\cap N_k\fstop
	\end{equation}
	The states $(x,y)$ and $(z,y)$ communicate through $B$-updates of the $L_2^\dagger$-dynamics, whereas
	\begin{equation}
		\cD\phi(x,y)=a_i-a_j<0
		\comma\qquad
		\cD\phi(z,y)=a_k-a_j>0\comma
	\end{equation}
	contradicting the preceding sign property. Hence, the level sets met by $B$ have consecutive indices.
	
	Suppose next that $B$ meets at least two level sets but does not contain one of them. Then, for some $i\neq j$, we may choose
	\begin{equation}
		x\in B\cap N_i
		\comma\qquad
		y\in N_i\setminus B
		\comma\qquad
		z\in B\cap N_j\fstop
	\end{equation}
	Again, $(x,y)$ and $(z,y)$ communicate, but
	\begin{equation}
		\cD\phi(x,y)=0
		\comma\qquad
		\cD\phi(z,y)=a_j-a_i\neq0\comma
	\end{equation}
	which is impossible on a communicating class on which $\cD\phi$ is either identically zero or strictly signed. Thus, every positive-weight block is either contained in one level set or is a union of consecutive complete level sets. The partition
	\begin{equation}
		[n]=N_1\sqcup\cdots\sqcup N_m
	\end{equation}
	is therefore $w$-compatible. Minimality and $m\ge2$ force $m=n$, so every $N_j$ is a singleton. Relabeling the sites according to the increasing values of $\phi$ proves the interval condition.
	
	Conversely, assume that the sites are labeled so that every positive-weight
	block of size at least two is an interval. Since every $\Pi_{B,1}$ preserves
	nondecreasing functions, the one-particle semigroup is monotone. Hence, by
	\cite[Lemma 22.17]{levin2017markov}, there exists a nonconstant
	nondecreasing eigenfunction $\phi$ such that
	\begin{equation}\label{eq:monotone-gap-eigenfunction}
		L_1\phi=-\gap_1(\alpha,w)\phi\fstop
	\end{equation}
	We claim that $\phi$ is strictly increasing.
	
	Suppose that $\phi(i)=\phi(i+1)$ for some $i=1,\ldots, n-1$. Taking the difference
	of \eqref{eq:monotone-gap-eigenfunction} at $i+1$ and $i$ gives
	\begin{equation}\label{eq:flat-gradient}
		0
		=
		L_1\phi(i+1)-L_1\phi(i)
	=
		\sum_{\substack{B\ni i+1\\B\not\ni i}}
		w_B\bigl(\pi_B(\phi)-\phi(i+1)\bigr)
		+
		\sum_{\substack{B\ni i\\B\not\ni i+1}}
		w_B\bigl(\phi(i)-\pi_B(\phi)\bigr)\fstop
	\end{equation}
	Every term on the right-hand side is nonnegative. Indeed, an interval containing
	$i+1$ but not $i$ lies to the right of $i$, while an interval containing $i$ but
	not $i+1$ lies to its left. Consequently, equality in \eqref{eq:flat-gradient}
	implies that $\phi$ is constant on every positive-weight block starting at
	$i+1$ or ending at $i$.
	Let
	$
		[n]=N_1\sqcup\cdots\sqcup N_m
	$
	be the partition into the level sets of $\phi$, ordered according to their
	values. Since $\phi$ is nondecreasing, the sets $N_j$ are intervals. We claim
	that every positive-weight block $B$ is either contained in one $N_j$, or is a
	union of consecutive level sets.
	Indeed, suppose that $B$ meets more than one level set and starts strictly
	inside some $N_j$. If $a=\min B$, then $\phi(a-1)=\phi(a)$, while $B$ starts
	at $a$ and contains points where $\phi>\phi(a)$. This contradicts the conclusion
	drawn from \eqref{eq:flat-gradient}. The analogous argument at $\max B$ shows
	that $B$ cannot end strictly inside a level set. The claim follows.
	
	Thus, the partition $(N_1,\ldots,N_m)$ is $w$-compatible. Since $\phi$ is
	nonconstant, $m\ge2$, while minimality excludes $2\le m<n$. Therefore
	$m=n$, and
	\begin{equation}\label{eq:strictly-increasing-gap-eigenfunction}
		\phi(1)<\phi(2)<\cdots<\phi(n)\fstop
	\end{equation}
	Before killing, the two-particle dynamics cannot change the order of the two coordinates: if an interval $B$ contains exactly one coordinate of $(x,y)$, relocating that coordinate inside $B$ cannot move it across the other one. Hence, every communicating class is contained in one of the two sets
	\begin{equation}
		\tset{(x,y)\in[n]^2_\tneq:x<y}
		\qquad\text{or}\qquad
		\tset{(x,y)\in[n]^2_\tneq:x>y}\fstop
	\end{equation}
	By \eqref{eq:strictly-increasing-gap-eigenfunction}, $\cD\phi$ has a strict sign on every communicating class. Since it is a $\gap_1(\alpha,w)$-eigenfunction of $-L_2^\dagger$, Perron--Frobenius theory shows that $\gap_1(\alpha,w)$ is the principal Dirichlet eigenvalue on every class. Recalling \eqref{eq:lambda-zero-killed-symmetric}--\eqref{eq:lambda-zero-killed}, we obtain
	\begin{equation}
		\lambda_0=\gap_1(\alpha,w)\fstop
	\end{equation}
	The equivalence \eqref{eq:tau-zero-lambda-zero} and strict monotonicity then give \eqref{eq:lambda-zero-gap-one}.
\end{proof}

\subsection{Proof of Theorems \ref{th:gap-1-2-quantitative}, \ref{th:phase-transition}, \ref{th:dichotomy-graph} and \ref{th:dichotomy-hypergraph}}\label{sec:comparison-final}

We collect the results established above.

\begin{proof}[Proof of Theorem \ref{th:gap-1-2-quantitative}]
	The lower bound was proved in Corollary \ref{cor:gap-lower-infty}, while the upper bound and its extension to general hypergraphs were established in Proposition \ref{pr:lambda-infty-upper} and Remark \ref{rem:lambda-infty-upper-general}. Sharpness of the two constants was proved in Remarks \ref{rem:gap-lower-sharp} and \ref{rem:lambda-infty-sharp}.
\end{proof}

\begin{proof}[Proof of Theorem \ref{th:phase-transition}]
	For minimal hypergraphs, the claim follows from Proposition \ref{pr:monotonicity} and Remark \ref{rem:monotonicity-gap-1}. The general case follows from Propositions \ref{pr:reduction}--\ref{pr:reduction-gap}.
\end{proof}

\begin{proof}[Proof of Theorems \ref{th:dichotomy-graph} and \ref{th:dichotomy-hypergraph}]
	For minimal hypergraphs, the characterizations of $\tau_c(\alpha,w)=\infty$ and $\tau_c(\alpha,w)=0$ were established in Propositions \ref{pr:mean-field-infty} and \ref{pr:segment-zero}, respectively.
	
	We now remove the minimality assumption, retaining only \eqref{eq:w-n}--\eqref{eq:connected-hypergraph}. By Proposition~\ref{pr:reduction-gap}, passing to any $w$-compatible partition preserves $\gap_1$ and $\gap_{2,\ast}$ at every scale $\tau>0$, and hence also $\lambda_0$, $\lambda_\infty$, and $\tau_c$. Moreover, the quotient hypergraph associated with the coarsest partition is minimal (Proposition \ref{pr:reduction}). Applying Proposition \ref{pr:segment-zero} to this quotient shows that $\tau_c(\alpha,w)=0$ if and only if, after ordering the quotient sites, every positive-weight quotient block is an interval. This is precisely the segment-like condition in Theorem \ref{th:dichotomy-hypergraph}\ref{it:hypergraph-dichotomy-segment}. Likewise, Proposition \ref{pr:mean-field-infty} yields the corresponding characterization of $\tau_c(\alpha,w)=\infty$. This proves Theorem \ref{th:dichotomy-hypergraph}.
	
	Theorem \ref{th:dichotomy-graph} then follows from the observation following the statement of its hypergraph counterpart.
\end{proof}

\section{Boundary-driven ${\rm KMP}$}\label{sec:KMP-reservoirs}
In this section, we prove Theorem \ref{th:gap-res}. The argument is based on a suitable decomposition of polynomials and on a killed-particle representation of their highest-degree components. In particular, it does not use the hidden model associated with the boundary-driven dynamics. That representation has instead played a central role in the description of the corresponding nonequilibrium steady state
\cite{de_masi_ferrari_gabrielli_hidden_2023,giardina_redig_tol_intertwining_2024}.

\subsection{Polynomials and killed particles}\label{sec:bd-particles}Recall from \eqref{eq:coeff-space} and \eqref{eq:L-k-sym} that
$\mathscr H_k^\per$ denotes the space of functions $\psi:[n]^k\to\R$, endowed with the
$L^2(\mu_k)$ scalar product, whereas $\mathscr H_k\subseteq\mathscr H_k^\per$ is the
subspace of symmetric functions. The generator $L_k$ acts on
$\mathscr H_k^\per$ and describes the labeled $k$-particle dynamics, while
$L_k^\sym=L_k|_{\mathscr H_k}$ describes its unlabeled counterpart.
The same hat notation as in \eqref{eq:hat-tilde-diagonal} will be used for
polynomials on $\Omega_{\rm bd}$: for $\psi \in \mathscr H_k$,
\begin{equation}\label{eq:hat-bd}
	\textstyle
	\widehat\psi(\eta)
	\eqdef
	\sum_{x_1,\ldots,x_k\in[n]}
	\psi(x_1,\ldots,x_k)\,
	\eta_{x_1}\cdots\eta_{x_k}
	\comma
	\qquad
	\eta\in\Omega_{\rm bd}\fstop
\end{equation}
Since $\Omega_{\rm bd}$ has nonempty interior, the map
in \eqref{eq:hat-bd} is a linear isomorphism between $\mathscr H_k$ and
the homogeneous polynomials of degree $k$ on $\Omega_{\rm bd}$.

Recall from Section \ref{sec:reservoirs-intro} that $\mathscr Q_k$ denotes the space of polynomials of degree at most
$k$ on $\Omega_{\rm bd}$, with $\mathscr Q_{-1}\eqdef{0}$. The generator
$\cL_{\rm bd}$ preserves each $\mathscr Q_k$, but its action is no longer
homogeneous: a reservoir update produces terms of degree strictly smaller than
the original one. In Proposition \ref{pr:bd-top-degree} below we identify the action of $\cL_{\rm bd}$ on the
highest-degree component of a polynomial.

First, we introduce some notation. Recall \eqref{eq:gen-res}--\eqref{eq:gen-bd} and \eqref{eq:N-x-bd}, and define, for $k\ge 1$, the multiplication operator $L_{\partial,k}$ on $\mathscr H_k^\per$ by
\begin{equation}\label{eq:L-partial-k-bd}
\textstyle	L_{\partial,k}\psi(x_1,\ldots,x_k)
	\eqdef
	-\psi(x_1,\ldots,x_k)
	\sum_{x\in[n]}\omega_x\,
	\tttonde{1-\mathbf E_x\big[U_x^{\mathfrak n_x(x_1,\ldots,x_k)}\big]}
	\comma
\end{equation}
for $x_1,\ldots, x_k \in [n]$.
Here, $\omega_x$, $U_x$ and $\mathbf E_x$ are as in
\eqref{eq:xi-x-u-x}--\eqref{eq:gen-res}:   $\mathbf E_x$ denotes expectation with respect to the
reservoir variables $(U_x,\xi_x)$ associated with an update at $x$, while $\omega_x\ge 0$ denotes the rate of such updates.
The corresponding boundary-driven particle generator is
\begin{equation}\label{eq:L-bd-k}
	L_{{\rm bd},k}
	\eqdef
	L_k+L_{\partial,k}
	\comma
	\qquad
	L_{{\rm bd},k}^{\sym}
	\eqdef
	L_{{\rm bd},k}|_{\mathscr H_k}
	\fstop
\end{equation}
Since $L_k$ is self-adjoint on $L^2(\mu_k)=L^2([n]^k,\mu_k)$ and $L_{\partial,k}$ is a nonpositive multiplication operator, $L_{{\rm bd},k}$ is self-adjoint and sub-Markovian. Moreover, it commutes with permutations of the labels, so that $\mathscr H_k$ is invariant.

The operator in \eqref{eq:L-partial-k-bd} has the following probabilistic interpretation. At a reservoir update at $x$, conditionally on $U_x$, each particle at $x$ survives independently with probability $U_x$; the process is killed as soon as at least one particle does not survive. Thus, $\mathbf E_x[U_x^{\mathfrak n_x}]$ is precisely the probability that all particles currently at $x$ survive the update.

\begin{proposition}[Top-degree intertwining]\label{pr:bd-top-degree}
	For every $k\ge1$ and every $\psi\in\mathscr H_k$,
	\begin{equation}\label{eq:bd-top-degree}
		\cL_{\rm bd}\widehat\psi
		-
		\widehat{L_{{\rm bd},k}^{\sym}\psi}
		\in
		\mathscr Q_{k-1}
		\fstop
	\end{equation}
	In particular, the action induced by $\cL_{\rm bd}$ on the quotient space \begin{equation}\mathscr Q_k/\mathscr Q_{k-1}=\set{[f]_k\eqdef f+\mathscr Q_{k-1}: f\in \mathscr Q_k}\end{equation} is similar to $L_{{\rm bd},k}^{\sym}$ and does not depend on $\rho$.
\end{proposition}

\begin{proof}
	The conservative part satisfies the exact intertwining
	$\cL\widehat\psi=\widehat{L_k^{\sym}\psi}$ by Proposition \ref{pr:particle-system-intertwining}. It remains to identify the highest-degree part of the reservoir generator.
	
	Fix $x_1,\ldots,x_k\in[n]$ and write $\mathfrak n_x=\mathfrak n_x(x_1,\ldots,x_k)$. For the monomial $\eta_{x_1}\cdots\eta_{x_k}$, a reservoir update at $x$ gives
	\begin{equation}	\label{eq:reservoir-monomial-top}
		\mathbf E_x\bigg[
		\prod_{i=1}^k
		\eta^{x,U_x,\xi_x}_{x_i}
		\bigg]
		=
		\mathbf E_x[
		U_x^{\mathfrak n_x}(\eta_x+\xi_x)^{\mathfrak n_x}
		]
		\prod_{y\neq x}\eta_y^{\mathfrak n_y}
		=
		\mathbf E_x[U_x^{\mathfrak n_x}]\,
		\eta_{x_1}\cdots\eta_{x_k}
		+
		q_x(\eta)
		\comma
	\end{equation}
	where $q_x\in\mathscr Q_{k-1}$. Indeed, every term in the binomial expansion of $(\eta_x+\xi_x)^{\mathfrak n_x}$ containing a positive power of $\xi_x$ has degree at most $k-1$ in $\eta$ and, moreover, $\xi_x\sim {\rm Gamma}(\beta_x,\rho_x)$ has finite moments of all orders. Multiplying \eqref{eq:reservoir-monomial-top} by the coefficients of $\psi$, summing over $x_1,\ldots,x_k$, and then over the reservoir sites gives
	\begin{equation}
		\cL_\partial\widehat\psi
		-
		\widehat{L_{\partial,k}\psi}
		\in
		\mathscr Q_{k-1}
		\fstop
	\end{equation}
	Together with the conservative intertwining, this proves \eqref{eq:bd-top-degree}. The operator $L_{{\rm bd},k}$ depends on the reservoir variables only through the moments of $U_x$, and hence is independent of $\rho$.
\end{proof}

Every polynomial in $\mathscr Q_k$ decomposes uniquely into homogeneous
components. By \eqref{eq:hat-bd}, the homogeneous component of degree $j$ is
identified with a unique element of $\mathscr H_j$. Equivalently, the hat map
identifies the quotient $\mathscr Q_j/\mathscr Q_{j-1}$ with
$\mathscr H_j$. Proposition \ref{pr:bd-top-degree} therefore gives the
following spectral decomposition.

\begin{proposition}[Polynomial spectrum]\label{pr:bd-polynomial-spectrum}
	For every $k\ge0$, counting algebraic multiplicities,
	\begin{equation}\label{eq:bd-polynomial-spectrum}
		\spec\tttonde{-\cL_{\rm bd}|_{\mathscr Q_k}}
		=
		\{0\}
		\cup
		\bigcup_{j=1}^k
		\spec\tttonde{-L_{{\rm bd},j}^{\sym}}
		\subseteq[0,\infty)
		\fstop
	\end{equation}
	Consequently, the polynomial spectrum of $-\cL_{\rm bd}$ is real, nonnegative, and independent of $\rho$.
\end{proposition}

\begin{proof}
	Choose a basis of $\mathscr Q_k$ adapted to the filtration
	\begin{equation}
		\mathscr Q_0
		\subseteq
		\mathscr Q_1
		\subseteq\cdots\subseteq
		\mathscr Q_k
		\fstop
	\end{equation}
	By Proposition \ref{pr:bd-top-degree}, the matrix of $\cL_{\rm bd}|_{\mathscr Q_k}$ is block triangular, with diagonal blocks $0$ and $L_{{\rm bd},j}^{\sym}$, $1\le j\le k$. This proves \eqref{eq:bd-polynomial-spectrum}.
	
	Each $L_{{\rm bd},j}^{\sym}$ is the restriction of the self-adjoint, nonpositive operator $L_{{\rm bd},j}$ to an invariant subspace. Its spectrum is therefore real and nonpositive. The final claim follows from the last statement of Proposition \ref{pr:bd-top-degree}.
\end{proof}

\begin{remark}[The reversible case]\label{rem:bd-reversible-spectrum} The above characterization of the polynomial spectrum of $\cL_{\rm bd}$ is valid in the general, not necessarily reversible, case.
	In the special case in which $\rho_x\equiv\rho_\ast$ on $\supp(\omega)$, the product Gamma measure in \eqref{eq:nu-rhoa} is reversible. Then, self-adjointness and \eqref{eq:Qk-invariance-bd} imply that the orthogonal complement of $\mathscr Q_{k-1}$ in $\mathscr Q_k$ is invariant under $\cL_{\rm bd}$. Since polynomials are dense in $L^2(\nu)$, the generator admits a complete orthonormal family of polynomial eigenfunctions, and its $L^2(\nu)$-spectrum coincides with the closure of the polynomial spectrum in \eqref{eq:bd-polynomial-spectrum}. In particular, the polynomial gap in \eqref{eq:def-polynomial-gap-bd} is the usual $L^2(\nu)$-spectral gap.
\end{remark}

\subsection{One-particle domination}\label{sec:bd-gap}
We next compare the principal eigenvalues of the killed particle generators. The comparison is most transparent through their common graphical construction.

\begin{lemma}[Killing-time comparison]\label{lem:bd-killing-comparison}
	For every $k\ge1$,
	\begin{equation}\label{eq:bd-principal-comparison}
		\min\spec\tonde{-L_{{\rm bd},k}}
		\ge
		\min\spec\tonde{-L_{{\rm bd},1}}
		\fstop
	\end{equation}
\end{lemma}

\begin{proof}
	Let $T_k$ denote the killing time of the labeled $k$-particle process generated by $L_{{\rm bd},k}$. Couple simultaneously the systems with different numbers of particles as follows. At each bulk update, use the same block and the same Dirichlet vector, and move all particles in that block conditionally independently given this vector. At a reservoir update at $x$, use the same $U_x$ and let each particle at $x$ survive independently with probability $U_x$.
	
	For $i=1,\ldots,k$, let $T_k^{(i)}$ be the killing time of the $i$-th label under this construction. Each label has the one-particle dynamics generated by $L_{{\rm bd},1}$, and
	\begin{equation}\label{eq:Tk-min-T1}
		T_k
		\eqdef
		\min_{1\le j\le k}T_k^{(j)}
		\le
		T_k^{(i)}
		\qquad\text{a.s.}\comma
	\end{equation}
	where the inequality holds for every $i=1,\ldots,k$. Hence, for all $t\ge0$,
	\begin{equation}\label{eq:survival-comparison-bd}
		\max_{x_1,\ldots,x_k\in[n]}
		\P_{x_1,\ldots,x_k}(T_k>t) \le \max_{x_1,\ldots,x_k\in [n]} \P_{x_1,\ldots, x_k}(T_k^{(i)}>t)
		=
		\max_{x\in[n]}
		\P_x(T_1>t)
		\comma
	\end{equation}
	where $\P_{x_1,\ldots,x_k}$ stands for the law of the $L_{{\rm bd},k}$ process started from
	$(x_1,\ldots,x_k)\in[n]^k$, and $\P_x$ for the law of the one-particle
	process started from $x\in[n]$. The equality in \eqref{eq:survival-comparison-bd} follows because, under $\P_{x_1,\ldots,x_k}$, the marginal law of the $i$-th label is exactly that of the one-particle process started from $x_i$.
	
	For any finite sub-Markov generator $A$, positivity of its semigroup gives
	\begin{equation}\label{eq:}
		\|e^{tA}\|_{\infty\to\infty}
		= \max_z\,(e^{tA}\mathbf 1)(z)=
		\max_z\P_z(T>t)\comma\qquad t \ge 0
		\comma
	\end{equation}
	where $T\ge 0$ denotes the random killing time of the corresponding continuous-time chain.
	Hence, the above identity and Gelfand's formula yield
	\begin{equation}\label{eq:survival-spectral-bound}
		\lim_{t\to\infty}
		\frac1t
		\log\max_z\P_z(T>t)
		=
		\max\spec(A)
		=
		-\min\spec(-A)
		\fstop
	\end{equation}
	Applying \eqref{eq:survival-spectral-bound} to both sides of \eqref{eq:survival-comparison-bd} proves \eqref{eq:bd-principal-comparison}.
\end{proof}

\begin{proof}[Proof of Theorem \ref{th:gap-res}]
	Since $\mathscr H_k\subseteq\mathscr H_k^\per$ and $L_{{\rm bd},k}$ is self-adjoint,
	\begin{equation}\label{eq:bd-sym-labeled-comparison}
		\min\spec\tttonde{-L_{{\rm bd},k}^{\sym}}
		\ge
		\min\spec\tttonde{-L_{{\rm bd},k}}
		\ge
		\min\spec\tttonde{-L_{{\rm bd},1}}
		\comma
		\qquad k\ge1
		\fstop
	\end{equation}
	Here, the second inequality is Lemma \ref{lem:bd-killing-comparison}, while the first is the variational characterization of the smallest eigenvalue. Since $\mathscr H_1=\mathscr H_1^\per$, equality holds throughout \eqref{eq:bd-sym-labeled-comparison} when $k=1$.
	
	Proposition \ref{pr:bd-polynomial-spectrum} now gives, for every $k\ge1$,
	\begin{equation}
		\gap\tttonde{\cL_{\rm bd}|_{\mathscr Q_k}}
		=
		\min_{1\le j\le k}
		\min\spec\tttonde{-L_{{\rm bd},j}^{\sym}}
		=
		\min\spec\tttonde{-L_{{\rm bd},1}}
		=
		\gap\tttonde{\cL_{\rm bd}|_{\mathscr Q_1}}
		\fstop
	\end{equation}
	Taking the infimum over $k\ge1$ proves \eqref{eq:gap-res-intro}. Independence of $\rho$ follows from Proposition~\ref{pr:bd-polynomial-spectrum}.
	
	If this common gap is positive, let $\phi\in\mathscr H_1\setminus\{0\}$ satisfy
	$L_{{\rm bd},1}\phi=-\gap(\cL_{\rm bd})\,\phi$. Proposition~\ref{pr:bd-top-degree} gives, for some constant $c\in\R$,
	\begin{equation}
		\cL_{\rm bd}\widehat\phi
		=
		-\gap(\cL_{\rm bd})\,\widehat\phi+c
		\fstop
	\end{equation}
	Therefore, $\widehat\phi-c/\gap(\cL_{\rm bd})$ is a degree-one eigenfunction associated with the gap.
\end{proof}

\begin{remark}[Further reservoir mechanisms]\label{rem:general-reservoir}
	The family in \eqref{eq:xi-x-u-x} interpolates between several reservoir mechanisms considered in the literature. The choice
	\begin{equation}
		\beta_x=\alpha_x
	\end{equation}
	corresponds to the boundary update used, for instance, in
	\cite{de_masi_ferrari_gabrielli_hidden_2023,giardina_redig_tol_intertwining_2024}. For finite $\beta_x$, the factor $U_x$ retains part of the energy present at $x$ before the update. At the opposite extreme,
	\begin{equation}
		U_x(\eta_x+\xi_x)\
		 \Longrightarrow\
		\xi_x'\sim{\rm Gamma}(\alpha_x,\rho_x)\comma\quad  \text{as}\  \beta_x\to\infty\comma
	\end{equation}
	with $\xi_x'$ being independent of $\eta_x$, and one recovers the full-refresh mechanism
	\cite{kipnis_heat_1982}; see also \cite{carinci_duality_2013-1}. In particular, the original KMP reservoirs correspond, up to the choice of the boundary rates and temperatures, to $\alpha_x=1$ and $\beta_x=\infty$: at each boundary update, the energy is replaced by an exponential random variable with mean $\rho_x$. Conversely, as $\beta_x\to0$, one has $U_x\to1$ and $\xi_x\to0$ in probability, so that the reservoir action becomes void. One may also allow $\rho_x=0$, interpreting $\xi_x=0$ almost surely, in which case the update only removes energy through the multiplication $\eta_x\longmapsto U_x\eta_x$.

	More generally, our proof does not rely on the beta--gamma form of the variables in \eqref{eq:xi-x-u-x}. It only uses that, at an update of the reservoir attached to $x$, the energy is replaced by
	\begin{equation}
		\eta_x\longmapsto U_x(\eta_x+\xi_x)\comma
	\end{equation}
	where $(U_x,\xi_x)\in[0,1]\times[0,\infty)$ is sampled independently of the current configuration and has finite moments of every order. The variables $U_x$ and $\xi_x$ need not be independent. The same top-degree computation gives \eqref{eq:L-partial-k-bd}, and the killing-time argument is unchanged. The full-refresh update is covered directly by setting the coefficient of every positive-degree term equal to zero. Analogous block-reservoir mechanisms may be treated in the same way, at the cost of heavier notation.
\end{remark}

\section{Stochastic exchange models}\label{sec:SEM}

The main goal of this section is to prove Theorem~\ref{th:SEM} from
Section~\ref{sec:SEM-intro}. We begin by giving a rigorous construction of
stochastic exchange models together with their associated hidden dynamics.
The core of the proof of Theorem~\ref{th:SEM} occupies
Sections~\ref{sec:SEM-particles}--\ref{sec:SEM-proof}, whereas examples of stochastic exchange models are
discussed in Section~\ref{sec:SEM-examples}.

\subsection{Construction and first properties}\label{sec:SEM-basic}
Fix $n\ge 2$.
Throughout, let
\begin{equation}
	\mathfrak M
	\eqdef
	\{
	M=(M_{xy})_{x,y\in[n]}\in[0,1]^{n\times n}:
	M\mathbf1=\mathbf1
	\}
	\subseteq\R^{n\times n}
\end{equation}
denote the compact space of row-stochastic $n\times n$ matrices, and let
$\Upsilon$ be a Borel measure on $\mathfrak M$, possibly with
$\Upsilon(\mathfrak M)=\infty$, satisfying
\eqref{eq:SEM-assumption-intro}. 
By row-stochasticity, this assumption equivalently reads
\begin{equation}\label{eq:SEM-assumption}
	\textstyle
	\int\Upsilon(\dd M)\,\sum_{x\in [n]}\tonde{1-M_{xx}}<\infty\fstop
\end{equation}

Recall from Section~\ref{sec:SEM-intro} that the stochastic exchange model
associated with $\Upsilon$ is built from the updates, with infinitesimal rate $\Upsilon(\dd M)$,
\begin{equation}\label{eq:SEM-updates}
	\eta\longmapsto\eta M
	\comma\qquad
	\theta\longmapsto M\theta\comma
\end{equation}
where the first update acts on energy configurations
$\eta\in\Omega$, regarded as row vectors, while the second defines the
corresponding hidden dynamics on column vectors $\theta\in\R^n$.
Row-stochasticity ensures that $\eta M\in\Omega$ whenever
$\eta\in\Omega$, as well as
\begin{equation}\label{eq:SEM-cube-invariance}
	M[a,b]^n\subseteq[a,b]^n
	\comma\qquad
	-\infty<a<b<\infty\comma
\end{equation}
since every coordinate of $M\theta$ is a convex combination of the
coordinates of $\theta$. In particular, although the hidden dynamics is
not conservative in general, every cube $[a,b]^n$ is invariant and every
flat configuration is fixed.

In the following proposition, we rigorously construct the stochastic exchange model and its hidden model starting from their infinitesimal descriptions. We show that the sole integrability condition \eqref{eq:SEM-assumption} is sufficient both to make the corresponding generators well defined on the corresponding polynomial spaces and, through standard Hille--Yosida theory, to construct the associated Feller processes; see, e.g., \cite{liggett_interacting_2005}. In particular, this construction allows $\Upsilon(\mathfrak M)$ to be infinite. If instead $\Upsilon(\mathfrak M)<\infty$, as happens, for instance, for the KMP model, the conclusion is immediate, since the operators in \eqref{eq:gen-SEM} and \eqref{eq:gen-SEM-hidden} below extend uniquely to bounded Markov generators on the corresponding spaces of continuous functions.

Before stating the proposition, recall from Section~\ref{sec:polynomials} that $\mathscr P_k$ denotes the space of polynomial functions on $\Omega$ of degree at most $k$, while $\mathscr R_k$ denotes the corresponding space of degree-$k$ homogeneous polynomials for the hidden model. Set, for $-\infty<a<b<\infty$,
\begin{equation}\label{eq:P-R-SEM}
	\textstyle
	\mathscr P
	\eqdef
	\bigcup_{k\ge0}\mathscr P_k
	\comma\qquad
	\mathscr R|_{[a,b]^n}
	\eqdef
	\bigoplus_{k\ge0}
	\mathscr R_k|_{[a,b]^n}
	\comma
\end{equation}
where $\mathscr R_k|_{[a,b]^n}$ denotes the restriction to $[a,b]^n$ of the polynomials in $\mathscr R_k$. For a compact set $\Xi$, we write $\cC(\Xi)$ for the Banach space of real-valued continuous functions on $\Xi$, endowed with the uniform norm. Below, we shall take either $\Xi=\Omega$ or $\Xi=[a,b]^n$.

\begin{proposition}[Construction and Feller property]
	\label{pr:SEM-Feller}
	Assume \eqref{eq:SEM-assumption}.
	Then, for every $f\in\mathscr P$ and
	$g\in\mathscr R|_{[a,b]^n}$,  $-\infty<a<b<\infty$, the integrals
	\begin{align}
		\cL^\tups f(\eta)
		&\eqdef
		\int\Upsilon(\dd M)
		\tonde{
			f(\eta M)-f(\eta)
		}
		\comma\qquad
		\eta\in\Omega\comma
		\label{eq:gen-SEM}
		\\
		\mathscr L^\tups g(\theta)
		&\eqdef
		\int\Upsilon(\dd M)
		\tonde{
			g(M\theta)-g(\theta)
		}
		\comma\qquad
		\theta\in[a,b]^n\comma
		\label{eq:gen-SEM-hidden}
	\end{align}
	are well defined and finite. Moreover, for every $k\ge0$,
	\begin{equation}\label{eq:SEM-poly-preservation}
		\cL^\tups\mathscr P_k
		\subseteq
		\mathscr P_k
		\comma\qquad
		\mathscr L^\tups
		\mathscr R_k|_{[a,b]^n}
		\subseteq
		\mathscr R_k|_{[a,b]^n}\fstop
	\end{equation}
	As a consequence,
	$(\cL^\tups,\mathscr P)$
	(resp.\ $(\mathscr L^\tups,\mathscr R|_{[a,b]^n})$)
	is a Markov pregenerator on $\cC(\Omega)$
	(resp.\ $\cC([a,b]^n)$), whose closure generates a unique
	Feller process on $\Omega$ (resp.\ $[a,b]^n$).
\end{proposition}
Since the proof of Proposition \ref{pr:SEM-Feller} is rather standard, we postpone it to Appendix \ref{app:SEM-Feller}.

By Proposition~\ref{pr:SEM-Feller}, the polynomial spaces $\mathscr P$ and
$\mathscr R|_{[a,b]^n}$ in \eqref{eq:P-R-SEM} are cores for the
respective generators. Since we shall primarily study their restrictions to these polynomial cores, with a slight abuse of notation, we use $\cL^\tups$ and $\mathscr L^\tups$ for both the pregenerators and
their closures.

\begin{remark}[Invariant measures]\label{rem:SEM-invariant-measures}
	Both $\Omega$ and $[a,b]^n$ are compact. Hence, by the Feller property
	and, e.g., \cite[Proposition~I.1.8(f)]{liggett_interacting_2005}, both
	processes admit invariant probability measures. Henceforth,
	$\mu^\tups$ denotes an arbitrary invariant probability measure of
	the stochastic exchange model on $\Omega$.
	
	Under the sole assumption \eqref{eq:SEM-assumption},
	$\mu^\tups$ need not be unique. For instance, if $\Upsilon$ is
	supported on permutation matrices, the dynamics only permutes the
	coordinates of configurations $\eta\in \Omega$, and many invariant probability measures coexist.
Proposition~\ref{pr:SEM-uniqueness} below provides several equivalent characterizations
of uniqueness, all involving only the second-order moments of the entries
of $M\in\mathfrak M$ with respect to $\Upsilon$.

In contrast, for the hidden process, invariant probability measures are
never unique:  row-stochasticity gives $M(c\mathbf1)=c\mathbf1$ for every
$M\in\mathfrak M$ and $c\in [a,b]$, thus, $\delta_{c\mathbf 1}$ is invariant.	
\end{remark}

\subsection{Particle systems, polynomials, and intertwinings}
\label{sec:SEM-particles}
This section develops, for general stochastic exchange models, the polynomial characterizations of the eigenvalue problem established for the ${\rm KMP}$ model in Section~\ref{sec:prelim-particle-systems}. This amounts to introducing the natural labeled-particle dynamics, identifying the corresponding linearly isomorphic polynomial spaces, and establishing the associated intertwining relations.
Most of the objects introduced in
Section~\ref{sec:prelim-particle-systems} have natural analogues in the
present setting, and the corresponding statements carry over verbatim. The main difference is that the particle system need not possess a canonical strictly positive reversible measure, which was used there to define the tilde maps \eqref{eq:tilde-tensor}--\eqref{eq:hat-tilde-diagonal} and characterize the spaces $\mathscr H_{k,\ast}$ and $\mathscr P_{k,\ast}$ (Proposition \ref{pr:Hk-star}). For the algebraic study of the eigenvalue problem, it suffices to replace that measure, at each level $k$, by an arbitrary strictly positive symmetric probability measure on $[n]^k$, say $\varpi_k$. Since $\varpi_k$ need be neither invariant nor
reversible for the particle dynamics, the corresponding statements
require minor modifications, which we record below and summarize in Proposition \ref{pr:SEM-pure-degree-isospectrality}. In this section, we retain most of the notation (e.g., $\mathscr H_k, \mathscr P_k, \mathscr R_k$, and their tensor analogues) introduced in
Section~\ref{sec:prelim-particle-systems}.
	
For $M\in\mathfrak M$ and $k\ge1$, let $M^{\otimes k}$ be the
Markov operator on $[n]^k$ under which the $k$ labels move independently according to the rows of $M$, and set
\begin{equation}\label{eq:L-k-Gamma}
	\textstyle
	L_k^\tups
	\eqdef
	\int\Upsilon(\dd M)\tonde{M^{\otimes k}-I}\comma
\end{equation}
as an operator acting on the coefficient space $\mathscr H_k^\per$ in \eqref{eq:coeff-space}.  By \eqref{eq:SEM-assumption}, this is a well-defined Markov generator, as, for every
$\psi\in\mathscr H_k^\per$ and $x_1,\ldots,x_k\in [n]$,
\begin{equation}\label{eq:L-k-Upsilon-integrability}
	\textstyle	\big|(M^{\otimes k}\psi)(x_1,\ldots, x_k)-\psi(x_1,\ldots, x_k)\big|
	\le
	2\max_{y_1,\ldots,y_k\in [n]}|\psi(y_1,\ldots, y_k)|
	\sum_{i=1}^k\tonde{1-M_{x_i x_i}}\fstop
\end{equation}
In particular, 	 $L_1$ is the generator of the continuous-time random walk on $[n]$ with rates
\begin{equation}\textstyle
	r_{xy}^\tups\eqdef \int \Upsilon(\dd M)\,M_{xy}\comma\qquad x,y\in [n]\comma x\neq y\fstop
\end{equation}
 More generally, for $k\ge 1$, $L_k^\tups$ is the generator of $k$ labeled particles, whose rate to jump from $(x_1,\ldots, x_k)\in [n]^k$ to another state $(y_1,\ldots, y_k)\in [n]^k$ reads
\begin{equation}
	\textstyle r^\tups(x_1,\ldots, x_k; y_1,\ldots, y_k)\eqdef \int \Upsilon(\dd M)\, M_{x_1y_1}\cdots M_{x_ky_k}\fstop
\end{equation}
As $M_{x_1y_1}\cdots M_{x_ky_k}=M_{x_{\varsigma(1)}y_{\varsigma(1)}}\cdots M_{x_{\varsigma(k)}y_{\varsigma(k)}}$, $\varsigma \in \mathfrak S_k$,  $L_k^\tups$ leaves $\mathscr H_k$ in \eqref{eq:coeff-space-sym} invariant. 

For every $k\ge1$, fix a strictly positive probability measure
$\varpi_k$ on $[n]^k$, invariant under permutations of the labels.  The
choice of $\varpi_k$ is immaterial.  One may take, for instance, the
measure $\mu_k$ from \eqref{eq:mu-k} for given positive site weights $\alpha=(\alpha_x)_{x\in [n]}$, the flat measure
$\varpi_k\equiv n^{-k}$, or, whenever it has full support, the probability measure
\begin{equation}\label{eq:mu-k-Gamma}
	\mu_k^\tups(x_1,\ldots,x_k)
	\eqdef \mu^\tups(\eta_{x_1}\cdots \eta_{x_k})
\comma\qquad x_1,\ldots, x_k\in [n]\comma
\end{equation}
induced by the $\Upsilon$-stochastic exchange model invariant measure $\mu^\tups$ fixed in
Remark~\ref{rem:SEM-invariant-measures}. It is immediate to check that $\mu_k^\tups$ is $L_k^\tups$-invariant. For the purpose of proving Theorem \ref{th:SEM}, no invariance or consistency
assumption on the family $(\varpi_k)_{k\ge1}$ will be needed.

Let $L_k^{\tups,\tvarpi}$ be the adjoint of $L_k^\tups$
in $L^2(\varpi_k)$.  If $D_{\varpi_k}$ denotes the diagonal matrix with
entries $\varpi_k=(\varpi_k(x_1,\ldots, x_k))_{(x_1,\ldots, x_k )\in [n]^k}$, then, viewed as $n^k\times n^k$ matrices,	
\begin{equation}\label{eq:SEM-varpi-adjoint}
	L_k^{\tups,\tvarpi}
	=
	D_{\varpi_k}^{-1}(L_k^\tups)^{\mathsf T}D_{\varpi_k}\fstop
\end{equation}
Thus, $L_k^{\tups,\tvarpi}$ is similar to the transpose of
$L_k^\tups$ and, since every finite matrix is similar to its
transpose, also to $L_k^\tups$ itself.  Its off-diagonal entries are
nonnegative, but it is a
Markov generator if and only if $\varpi_k$ is invariant for
$L_k^\tups$; it is the time-reversal generator only in that case.

The spaces $\mathscr P_k$, $\mathscr H_k$, and $\mathscr R_k$ and the hat maps in \eqref{eq:hat-tensor} and \eqref{eq:hat-tilde-diagonal} are
independent of $\varpi_k$.  In contrast, the tilde maps \eqref{eq:tilde-tensor}--\eqref{eq:hat-tilde-diagonal}
depend on the reference measure.  Accordingly, for
$\psi\in\mathscr H_k$, set
\begin{equation}\label{eq:tilde-psi-varpi-Gamma}
	\widetilde\psi^{\tvarpi}(\theta)
	\eqdef
	\sum_{x_1,\ldots,x_k\in[n]}
	\varpi_k(x_1,\ldots,x_k)\,
	\psi(x_1,\ldots,x_k)\,
	\theta_{x_1}\cdots\theta_{x_k}\comma\qquad \theta \in \R^n\fstop
\end{equation}

The arguments of Propositions~\ref{pr:polynomial-identifications} and
\ref{pr:particle-system-intertwining} apply verbatim, with the sole change that the hidden-model intertwining involves the adjoint in
$L^2(\varpi_k)$. Indeed, strict positivity and label symmetry of $\varpi_k$ imply that
\begin{equation}\label{eq:SEM-basic-isomorphisms}
	\mathscr H_k
	\xrightarrow[\hspace*{7mm}]{\ \psi\longmapsto\widehat\psi\ }
	\mathscr P_k
	\comma\qquad
	\mathscr H_k
	\xrightarrow[\hspace*{7mm}]{\ \psi\longmapsto\widetilde\psi^\tvarpi\ }
	\mathscr R_k
\end{equation}
are linear isomorphisms.  
Furthermore, for every $\psi \in \mathscr H_k$,
\begin{equation}\label{eq:SEM-intertwinings-varpi}
	\cL^\tups\widehat\psi
	=
	\widehat{L_k^\tups\psi}
	\comma\qquad
	\mathscr L^\tups\widetilde\psi^{\,\varpi}
	=
	\widetilde{L_k^{\tups,\tvarpi}\psi}^{\!\tvarpi}
\fstop
\end{equation}
Consequently, all four finite-dimensional operators in
\begin{equation}\label{eq:SEM-full-isospectrality}
	\cL^\tups|_{\mathscr P_k}
	\comma\qquad
	L_k^\tups|_{\mathscr H_k}
	\comma\qquad
	L_k^{\tups,\tvarpi}|_{\mathscr H_k}
	\comma\qquad
	\mathscr L^\tups|_{\mathscr R_k}
\end{equation}
are similar.  In particular, counting algebraic multiplicities,
\begin{equation}\label{eq:SEM-energy-particle-spectrum}
	\spec(\cL^\tups|_{\mathscr P_k})
	=
	\spec(L_k^\tups|_{\mathscr H_k})
	=
	\spec(\mathscr L^\tups|_{\mathscr R_k})\fstop
\end{equation}

We finally record the corresponding modification of the finer decompositions from Section~\ref{sec:polynomials}. Recall \eqref{eq:particle-removal}, \eqref{eq:particle-removal-addition-sym}, and	 \eqref{eq:particle-removal-addition-relations}, and let
\begin{equation}
	\mathfrak a_k(\mathscr H_{k-1})
	\subseteq
	\mathscr H_k
	\comma\qquad
	\mathscr H_{k,\ast}^{\varpi}
	\eqdef
	\mathfrak a_k(\mathscr H_{k-1})^{\perp_{\varpi_k}}\fstop
\end{equation}
The same tensorization argument as in
\eqref{eq:particle-consistency} gives
\begin{equation}
	L_k^\tups\mathfrak a_k
	=
	\mathfrak a_kL_{k-1}^\tups\fstop
\end{equation} Consequently, $\mathfrak a_k(\mathscr H_{k-1})$ is
$L_k^\tups$-invariant and, by adjointness, its orthogonal complement
$\mathscr H_{k,\ast}^{\varpi}$ is
invariant for the adjoint operator $L_k^{\tups,\tvarpi}$:
\begin{equation}\label{eq:SEM-invariance}
	L_k^\tups \mathfrak a_k(\mathscr H_{k-1})\subseteq \mathfrak a_k(\mathscr H_{k-1})\comma\qquad L_k^{\tups,\tvarpi}\mathscr H_{k,\ast}^\tvarpi\subseteq \mathscr H_{k,\ast}^\varpi\fstop
\end{equation} This is the only relevant
change from Section~\ref{sec:polynomials}: unless $L_k^\tups$ is
self-adjoint in $L^2(\varpi_k)=L^2([n]^k,\varpi_k)$, the space
$\mathscr H_{k,\ast}^{\varpi}$ need not be invariant under
$L_k^\tups$ itself.

The image of $\mathscr H_{k,\ast}^{\varpi}$ under the $\varpi$-tilde map in \eqref{eq:tilde-psi-varpi-Gamma} is nevertheless the
same $\varpi$-independent space as in \eqref{eq:R-k-ast}:
\begin{equation}\label{eq:SEM-R-k-star-varpi}
	\ttset{\widetilde\psi^\varpi\in \mathscr R_k:
	\psi\in\mathscr H_{k,\ast}^\varpi
}
	=
	\mathscr R_{k,\ast}
	\fstop
\end{equation}
Indeed, by arguing as in the proof of Proposition \ref{pr:Rk-star}, under the $\varpi$-tilde map in \eqref{eq:tilde-psi-varpi-Gamma}, $\varpi$-orthogonality to
$\mathfrak a_k(\mathscr H_{k-1})$ is equivalent to
$
	\sum_{x\in[n]}\partial_{\theta_x}
	\widetilde\psi^\varpi
	=0$, 
that is, to invariance under translations in the direction $\mathbf1\in \R^n$.  In view of \eqref{eq:SEM-R-k-star-varpi},
 the second intertwining in \eqref{eq:SEM-intertwinings-varpi} and the second inclusion in \eqref{eq:SEM-invariance} together give
\begin{equation}
	\mathscr L^\tups\mathscr R_{k,\ast}
	\subseteq
	\mathscr R_{k,\ast}\fstop
\end{equation}

On the $\eta$-variable side, the hat map in \eqref{eq:hat-tilde-diagonal} sends
$\mathfrak a_k(\mathscr H_{k-1})$ onto $\mathscr P_{k-1}$ and hence
identifies  $\widehat{\mathscr H_{k,\ast}^{\varpi}}$ with a vector-space
complement of $\mathscr P_{k-1}$ in $\mathscr P_k$. This complement need
not be $\cL^\tups$-invariant, since
$\mathscr H_{k,\ast}^{\varpi}$ is invariant under
$L_k^{\tups,\tvarpi}$ rather than under $L_k^\tups$; see \eqref{eq:SEM-intertwinings-varpi} and \eqref{eq:SEM-invariance}. We therefore
consider the quotient $\mathscr P_k/\mathscr P_{k-1}$, whose elements are
the equivalence classes
\begin{equation}
	[f]_k
	\eqdef
	f+\mathscr P_{k-1}
	\comma\qquad
	f\in\mathscr P_k\fstop
\end{equation}
Since both $\mathscr P_k$ and $\mathscr P_{k-1}$ are $\cL^\tups$-invariant,
$\cL^\tups$ induces a well-defined operator
on $\mathscr P_k/\mathscr P_{k-1}$, describing the degree-$k$ action of $\cL^\tups$
independently of any choice of complement. A similar construction was
used for the boundary-driven ${\rm KMP}$ model in
Section~\ref{sec:bd-particles}; see
Proposition~\ref{pr:bd-top-degree}.

The preceding space isomorphisms and operator identifications yield the
following refinement of \eqref{eq:SEM-energy-particle-spectrum}, which
separates the spectral contributions of the successive quotients
$\mathscr P_k/\mathscr P_{k-1}$. Recall that, here,
$\gap(\cL^\tups|_{\mathscr P_k})$ is the minimum real part of the
spectrum of $-\cL^\tups|_{\mathscr P_k}$ after removing one copy of
the zero eigenvalue corresponding to constants, with algebraic
multiplicities counted.

\begin{proposition}
	\label{pr:SEM-pure-degree-isospectrality}
	For every $k\ge1$, counting algebraic multiplicities,
	\begin{equation}\label{eq:SEM-layer-spectrum}
	\spec(
	\cL^\tups\
	\text{on}\
	\mathscr P_k/\mathscr P_{k-1}
	)
	=
	\spec(
	L_k^{\tups,\tvarpi}
	|_{\mathscr H_{k,\ast}^{\varpi}}
	)
	=
	\spec(
	\mathscr L^\tups|_{\mathscr R_{k,\ast}}
	)\fstop
\end{equation}
	Consequently, 
	\begin{equation}\label{eq:SEM-spectrum-filtration}
		\spec\tttonde{
			\cL^\tups|_{\mathscr P_k}
		}
		=
		\{0\}
		\bigcup_{j=1}^k
		\spec\tttonde{
			\mathscr L^\tups
			|_{\mathscr R_{j,\ast}}
		}
	\end{equation}
as multisets, and
\begin{equation}\label{eq:SEM-gap-via-hidden}
	\gap(\cL^\tups|_{\mathscr P_k})
	=
	\min_{1\le j\le k}\gap(\mathscr L^\tups|_{\mathscr R_{j,\ast}})\comma
\end{equation}
where
\begin{equation}\label{eq:SEM-gap-hidden}
	\gap(\mathscr L^\tups|_{\mathscr R_{k,\ast}})
	\eqdef
	\min\ttset{
		{\rm Re}(\lambda):
		\lambda\in
		\spec(-\mathscr L^\tups|_{\mathscr R_{k,\ast}})
	}\fstop
\end{equation}
\end{proposition}

In light of Proposition~\ref{pr:SEM-pure-degree-isospectrality},
Theorem~\ref{th:SEM} admits an equivalent interpretation in terms of
 unlabeled particle systems, namely, 
 \begin{equation}\label{eq:SEM-gap-particle-interpretation}
 	\gap(\cL^\tups)
 	=
 	\inf_{k\ge1}
 	\gap\tttonde{
 		L_k^\tups|_{\mathscr H_k}
 	}
 	=
 	\gap\tttonde{
 		L_2^\tups|_{\mathscr H_2}
 	}\comma
 \end{equation}
where, consistently with Theorem~\ref{th:SEM} and the definition in \eqref{eq:SEM-gap-hidden}, all the gaps are understood as the smallest real parts of the eigenvalues of the negatives of the corresponding generators (possibly excluding the eigenvalue associated to constant functions), with $\gap(\cL^\tups)$ defined by restricting $\cL^\tups$ to polynomial functions. A labeled-particle strengthening, bearing the same relation to
Theorem~\ref{th:SEM} as Theorem~\ref{th:disc-ell-KMP} does to
Theorem~\ref{th:main-KMP}, also remains valid:
\begin{equation}\label{eq:SEM-gap-labeled-unlabeled}
	\gap(L_k^\tups)
	=
	\gap\tttonde{
		L_k^\tups|_{\mathscr H_k}
	}
	\comma\qquad k\ge1\fstop
\end{equation}
Indeed, the preceding identifications and the proofs in the next section admit analogous synchronized tensor versions.
To keep the presentation concise and avoid introducing an additional
layer of tensor notation, we leave the details of this refinement to
the reader.

\subsection{Proof of Theorem \ref{th:SEM}}\label{sec:SEM-proof}

In view of Proposition~\ref{pr:SEM-pure-degree-isospectrality} and, in particular, \eqref{eq:SEM-gap-via-hidden}--\eqref{eq:SEM-gap-hidden}, $\gap(\cL^\tups)$ in \eqref{eq:SEM-gap-intro} is determined by its hidden-model counterpart, namely, $\inf_{k\ge 1}\gap(\mathscr L^\tups|_{\mathscr R_{k,\ast}})$. 
The analysis of the latter lies at the core of the proof of the spectral-gap identities in Theorems  \ref{th:main-KMP}--\ref{th:disc-ell-KMP}. To prove the identity in \eqref{eq:SEM-gap-identity}, the core of the argument from
Section~\ref{sec:proof-main-KMP} carries over essentially unchanged in this more general setting, after allowing for possibly complex eigenvalues and
eigenfunctions.
\begin{proof}[Proof of Theorem \ref{th:SEM}]
	Fix any
	strictly positive probability measure $\pi$ on $[n]$; for instance,
	one may take $\pi\eqdef\varpi_1$, where $\varpi_1$ is the reference
	measure fixed in Section~\ref{sec:SEM-particles}. As in
	\eqref{eq:var-pi-intro}, set, for every $\theta \in \R^n$,
	\begin{equation}\label{eq:var-pi-SEM}
		\textstyle
		\pi(\theta)
		\eqdef
		\sum_{x\in[n]}\pi_x\,\theta_x
		\comma\qquad
		\var_\pi(\theta)
		\eqdef
		\sum_{x\in[n]}
		\pi_x\tonde{\theta_x-\pi(\theta)}^2\fstop
	\end{equation}
As in Section \ref{sec:proof-main-KMP}, fix $k\ge2$, an even integer $2\le\ell\le k$, and $-\infty<a<b<\infty$. Then, by Lemmas \ref{lem:var-power}--\ref{lem:G-Var} and \eqref{eq:var-a-b}, one has
\begin{equation}\label{eq:SEM-var-domination}
	\var_\pi^{\ell/2}\in\mathscr R_{\ell,\ast}
	\comma\qquad
	|g(\theta)|
	\le
	C\var_\pi(\theta)^{k/2}
	\le
	C\tonde{b-a}^{k-\ell}\var_\pi(\theta)^{\ell/2}\comma
\end{equation}
for every $g$ in the complexification of $\mathscr R_{k,\ast}$, for every $\theta \in [a,b]^n$, and some $C=C(g,k,\pi)>0$. Further, let $\lambda\in \C$ be
an eigenvalue of
$-\mathscr L^\tups|_{\mathscr R_{k,\ast}}$, and let $g\neq0$ be a
corresponding eigenfunction in the complexification of
$\mathscr R_{k,\ast}$. Choosing $\theta\in[a,b]^n$ with
$g(\theta)\neq0$, positivity of the hidden-model semigroup $(e^{t\mathscr L^\tups})_{t\ge 0}$ and
\eqref{eq:SEM-var-domination} give, for every $t>0$,
\begin{align}\label{eq:SEM-hidden-gap-comparison}
	\begin{aligned}
	e^{-{\rm Re}(\lambda)t}|g(\theta)|
	=
	|(e^{t\mathscr L^\tups}g)(\theta)|
	&\le
	e^{t\mathscr L^\tups}|g|(\theta)\\
	&
	\le
C\tonde{b-a}^{k-\ell}e^{t\mathscr L^\tups}\var_\pi^{\ell/2}(\theta)\\
&	\le
	C\tonde{b-a}^k
	\ttnorm{
		e^{t\mathscr L^\tups|_{\mathscr R_{\ell,\ast}}}
	}_{a,b}
	\comma
	\end{aligned}
\end{align}
where, as in \eqref{eq:tensor-semigroup-norm}, $\norm{\emparg}_{a,b}:\mathscr R_{\ell,\ast}\to [0,\infty)$ denotes the uniform norm over $\theta \in[a,b]^n$.
Hence, 
\begin{equation}
	-{\rm Re}(\lambda)
	\le
	\lim_{t\to\infty}\frac1t
	\log
	\ttnorm{
		e^{t\mathscr L^\tups|_{\mathscr R_{\ell,\ast}}}
	}_{a,b}
=
	\max\ttset{
		{\rm Re}(\sigma):
		\sigma\in
		\spec(
		\mathscr L^\tups|_{\mathscr R_{\ell,\ast}}
		)
	}
	=
	-\gap(
		\mathscr L^\tups|_{\mathscr R_{\ell,\ast}})\comma
\end{equation}
where the first identity comes from the fact that $\mathscr R_{\ell,\ast}$ is finite dimensional and Gelfand's formula, whereas the second identity follows from the definition in \eqref{eq:SEM-gap-hidden}.
Taking the minimum over the eigenvalues $\lambda\in \C$ of $-\mathscr L^\tups|_{\mathscr R_{\ell,\ast}}$ yields, for $k\ge 2$,
\begin{equation}\label{eq:SEM-higher-degree-ordering}
	\gap(\mathscr L^\tups|_{\mathscr R_{\ell,\ast}})\le \gap(\mathscr L^\tups|_{\mathscr R_{k,\ast}})
	\comma\qquad
	\text{for every}\ 2\le\ell\le k\ \text{with $\ell\ge 2$ even}\fstop
\end{equation}
In particular, choosing $\ell=2$ and using
\eqref{eq:SEM-gap-via-hidden} gives
\begin{equation}\label{eq:SEM-degree-two-conclusion}
	\gap(\cL^\tups)
	=
	\min\ttset{
		\gap\tttonde{
			\mathscr L^\tups|_{\mathscr R_{1,\ast}}},\gap\tttonde{
			\mathscr L^\tups|_{\mathscr R_{2,\ast}}
	}}
	=
	\gap(\cL^\tups|_{\mathscr P_2})\fstop
\end{equation}
This proves the identity in \eqref{eq:SEM-gap-identity}.

The existence of an invariant measure for $\cL^\tups$ was established in Section \ref{sec:SEM-basic}, Remark \ref{rem:SEM-invariant-measures}; the claim on its uniqueness is the equivalence of the assertions \ref{it:uniq1} and \ref{it:uniq2} of Proposition~\ref{pr:SEM-uniqueness}.	
\end{proof} 

\begin{remark}\label{rem:SEM-higher-degree}
	The ordering in \eqref{eq:SEM-higher-degree-ordering}, together with
	the isospectrality in \eqref{eq:SEM-layer-spectrum}, yields the
	$\cL^\tups$-analogue of the ${\rm KMP}$
	inequalities \eqref{eq:gap-even-odd} from
	Remark~\ref{rem:new-ordering}.
\end{remark}

\begin{remark} As observed in Remark~\ref{rem:even-degree-cones}, if
	$k\ge2$ is even, the cone $\mathscr C_{k,\ast}$ of nonnegative
	functions in $\mathscr R_{k,\ast}$ is proper, has nonempty interior,
	and is preserved by the hidden-model semigroup. Hence, the
	Perron--Frobenius theorem shows that
	$
		\gap\tttonde{
			\mathscr L^\tups|_{\mathscr R_{k,\ast}}
		}
$
	is a real, nonnegative eigenvalue of
	$-\mathscr L^\tups|_{\mathscr R_{k,\ast}}$, admitting a nonzero 
	nonnegative eigenfunction. By the aforementioned isospectralities, it is
	also an eigenvalue of $-\cL^\tups|_{\mathscr P_k}$ and
	$-L_k^\tups|_{\mathscr H_k}$. In particular, this applies to
	$k=2$.
	Consequently, if
	\begin{equation}\label{eq:SEM-gap-1-2}
		\gap\tttonde{
			\mathscr L^\tups|_{\mathscr R_{2,\ast}}
		}
		\le
		\gap\tttonde{
			\mathscr L^\tups|_{\mathscr R_{1,\ast}}
		}\comma
	\end{equation}
	then the identity in \eqref{eq:SEM-degree-two-conclusion} shows that $\gap(\cL^\tups)$ is a
	real eigenvalue. Without \eqref{eq:SEM-gap-1-2}, however, $\gap(\cL^\tups)$ need not itself be an eigenvalue, as the following elementary example shows.
	
	Let $n=3$ and $\Upsilon=\delta_{M^\star}$, where
\begin{equation}
	M^\star
	\eqdef
	\frac16
	\begin{pmatrix}
		3 & 2 & 1\\
		1 & 3 & 2\\
		2 & 1 & 3
	\end{pmatrix}
\end{equation}
is row-stochastic.
	The eigenvalues of $M^\star$ are $1,z,\bar z\in \C$, where
	\begin{equation}
		z=\frac14+\frac{\sqrt3}{12}\,{\rm i}
		\comma\qquad
		|z|^2=\frac1{12}\fstop
	\end{equation}
	Hence, 
	\begin{equation}
		\spec\tttonde{
			-\mathscr L^\tups|_{\mathscr R_{1,\ast}}
		}
		=
		\ttset{1-z,1-\bar z}\comma
		\qquad
		\gap\tttonde{
			\mathscr L^\tups|_{\mathscr R_{1,\ast}}
		}
		=
		\frac34\comma
	\end{equation}
	although neither eigenvalue equals $3/4$. Moreover, 
	\begin{equation}
		\spec\tttonde{
			-\mathscr L^\tups|_{\mathscr R_{2,\ast}}
		}
		=
		\ttset{
			1-z^2,
			1-|z|^2,
			1-\bar z^2
		}\comma
	\end{equation}
	and, therefore,
	\begin{equation}
		\gap\tttonde{
			\mathscr L^\tups|_{\mathscr R_{2,\ast}}
		}
		=
		\frac{11}{12}
		>
		\frac34
		=
		\gap(\cL^\tups)\fstop
	\end{equation}
	Thus, in this example, \eqref{eq:SEM-gap-1-2} fails and $\gap(\cL^\tups)$ is the real
	part of a nonreal eigenvalue. 
\end{remark}

It remains to prove the uniqueness statement in Theorem~\ref{th:SEM}. 

\begin{proposition}[Uniqueness of the invariant measure]
	\label{pr:SEM-uniqueness}
	Assume that the Borel measure $\Upsilon$ on $\mathfrak M$ satisfies
	\eqref{eq:SEM-assumption}. Then, the following statements are
	equivalent:
	\begin{enumerate}[(i)]
		\item\label{it:uniq1}
		The $\Upsilon$-stochastic exchange model admits a unique invariant
		probability measure.
		
		\item\label{it:uniq2}
		The polynomial gap in \eqref{eq:SEM-gap-intro} satisfies
		\begin{equation}
			\gap(\cL^\tups)>0\fstop
		\end{equation}
		
		\item\label{it:uniq3}
		The two-particle generator $L_2^\tups$ admits a unique invariant
		probability measure.
		
		\item\label{it:uniq4}
		The hidden-model gap in \eqref{eq:SEM-gap-hidden} satisfies
		\begin{equation}
			\gap\tttonde{
				\mathscr L^\tups
				|_{\mathscr R_{2,\ast}}
			}>0\fstop
		\end{equation}
		
		\item\label{it:uniq5}
		For every strictly positive probability measure $\pi$ on $[n]$ and some $t>0$, 
		\begin{equation}\label{eq:SEM-strict-variance-contraction}
			e^{t\mathscr L^\tups}\var_\pi(\theta)
			<
			\var_\pi(\theta)
			\comma\qquad
			\text{for every nonconstant }\theta\in\R^n\comma
		\end{equation}
		where $\var_\pi$ was defined in \eqref{eq:var-pi-SEM}.
	\end{enumerate}
\end{proposition}
\begin{remark}[A variant of condition \ref{it:uniq5}]
	If $L_1^\tups$ admits a strictly positive invariant probability
	measure $\pi$, then, for this choice of $\pi$, \textquotedblleft some $t>0$\textquotedblright\ in
	\eqref{eq:SEM-strict-variance-contraction} may be replaced by
	\textquotedblleft every $t>0$\textquotedblright. Indeed, $L_1$-invariance of $\pi$ and Jensen's inequality
	give
	\begin{equation}
		e^{t\mathscr L^\tups}\var_\pi
		\le
		\var_\pi
		\comma\qquad t\ge0\fstop
	\end{equation}
	Hence, for every $\theta\in\R^n$, the map
	$t\mapsto
	e^{t\mathscr L^\tups}\var_\pi(\theta)
	$
	is nonincreasing and analytic. Equality at any positive time would
	force equality at all times, contradicting \eqref{eq:SEM-strict-variance-contraction}.
\end{remark}

\begin{proof}[Proof of Proposition \ref{pr:SEM-uniqueness}]
	We first prove the equivalence of
	\ref{it:uniq2} and \ref{it:uniq3}. By
	\eqref{eq:SEM-gap-particle-interpretation} and \eqref{eq:SEM-gap-labeled-unlabeled},
	\begin{equation}\label{eq:SEM-gap-L2-labeled}
		\gap(\cL^\tups)
		=
		\gap\tttonde{
			L_2^\tups|_{\mathscr H_2}
		}
		=
		\gap(L_2^\tups)\fstop
	\end{equation}
Recall that, for a finite-state Markov generator, zero is always an eigenvalue of the corresponding generator, and its
	multiplicity equals the number of closed communicating classes, while
	all the remaining eigenvalues of the generator have strictly
	negative real part. Hence, $L_2^\tups$ admits a unique invariant
	probability measure if and only if $\gap(L_2^\tups)>0$. Together
	with \eqref{eq:SEM-gap-L2-labeled}, this proves
	\ref{it:uniq2} $\Longleftrightarrow$ \ref{it:uniq3}.
	
	To compare \ref{it:uniq3} and \ref{it:uniq4}, let $1\le n_1\le n$ and
	$1\le n_2\le n(n+1)/2$ denote the numbers of closed communicating classes of,
	respectively, the one-particle chain $L_1^\tups = L_1^\tups|_{\mathscr H_1}$ and the unlabeled two-particle chain $L_2^\tups|_{\mathscr H_2}$, respectively.  Proposition~\ref{pr:SEM-pure-degree-isospectrality}
	and the spectral filtration in \eqref{eq:SEM-spectrum-filtration} give
	\begin{equation}\label{eq:SEM-kernel-dimension}
		\dim\ker\tttonde{
			\mathscr L^\tups
			|_{\mathscr R_{2,\ast}}
		}
		=
		n_2-n_1\comma
	\end{equation}
	yielding, together with $n_2\ge n_1(n_1+1)/2$, 
	\begin{equation}\label{eq:ker-L-R2}
		\ker\tttonde{
			\mathscr L^\tups
			|_{\mathscr R_{2,\ast}}
		}
		=
		\{0\}
		\quad\Longleftrightarrow\quad
		n_2=1\fstop
	\end{equation}
	On the one hand, since all nonzero eigenvalues under consideration have strictly
	negative real part, the left-hand side of \eqref{eq:ker-L-R2} is equivalent to
	\ref{it:uniq4}. On the other hand, $n_2=1$ is equivalent to
	positivity of $\gap(L_2^\tups|_{\mathscr H_2})$ in \eqref{eq:SEM-gap-L2-labeled}, thus,
	to \ref{it:uniq3}. This proves
	\ref{it:uniq3} $\Longleftrightarrow$ \ref{it:uniq4}.
	
	We next prove
	\ref{it:uniq4} $\Longleftrightarrow$ \ref{it:uniq5}. Recall from \eqref{eq:cone-C-2-star} the definition of the proper 
	cone $\mathscr C\subset \mathscr R_{2,\ast}$.	
	For any strictly positive probability measure $\pi$,
	$\var_\pi\in\interior(\mathscr C)$, while
	$e^{t\mathscr L^\tups}\mathscr C\subseteq\mathscr C$,
	$t\ge0$.	Assume \ref{it:uniq4}. Then, the hidden-model semigroup on
	$\mathscr R_{2,\ast}$ converges exponentially to zero; in particular,  
	$e^{t\mathscr L^\tups}\var_\pi
		\to 0
		$ in $\mathscr R_{2,\ast}$ as $t\to \infty$.
	Since $\var_\pi\in\interior(\mathscr C)$, for all sufficiently large
	$t$,
	\begin{equation}\label{eq:var-var-int}
		\var_\pi-e^{t\mathscr L^\tups}\var_\pi
		\in\interior(\mathscr C)\comma
	\end{equation}
	which is precisely \eqref{eq:SEM-strict-variance-contraction}. 	Conversely, suppose that
	\eqref{eq:SEM-strict-variance-contraction} (or, equivalently, \eqref{eq:var-var-int}) holds for some strictly
	positive $\pi$ and some $t>0$. 	Since
	$\var_\pi-e^{t\mathscr L^\tups}\var_\pi$ is strictly positive
	on the compact set
	$\ttset{
			\theta\in\R^n:
			\pi(\theta)=0, 
			\var_\pi(\theta)=1
		}$, 
	homogeneity and translation invariance yield some $\varepsilon>0$
	such that
	\begin{equation}\label{eq:SEM-uniform-variance-contraction}
		e^{t\mathscr L^\tups}\var_\pi
		\le
		(1-\varepsilon)\var_\pi\comma
	\end{equation}
	which, by positivity and iteration, further gives
	\begin{equation}\label{eq:SEM-iterated-variance-contraction}
		e^{\ell t\mathscr L^\tups}\var_\pi
		\le
		(1-\varepsilon)^\ell\var_\pi
		\comma\qquad \ell\ge1\fstop
	\end{equation}
	By Gelfand's formula, \eqref{eq:SEM-iterated-variance-contraction} gives $\gap(\mathscr L^\tups|_{\mathscr R_{2,\ast}})\ge -\frac1t\log(1-\eps)>0$, thus, proving \ref{it:uniq4}.
	
		It remains to prove
	\ref{it:uniq1} $\Longleftrightarrow$ \ref{it:uniq2}. Suppose that
	$\gap(\cL^\tups)=0$. By \eqref{eq:SEM-degree-two-conclusion},
	$
		\gap\tttonde{\cL^\tups|_{\mathscr P_2}}=0$.
	Moreover, by \eqref{eq:SEM-basic-isomorphisms}--\eqref{eq:SEM-energy-particle-spectrum},
	there
	exists a nonconstant real $f\in\mathscr P_2$ such that
	$\cL^\tups f=0$. 	Hence $e^{t\cL^\tups}f=f$ for every $t\ge0$. Define
	\begin{equation}
		\Omega_+
		\eqdef
		\ttset{\eta\in\Omega:f(\eta)=\max_\Omega f}
		\comma\qquad
		\Omega_-
		\eqdef
		\ttset{\eta\in\Omega:f(\eta)=\min_\Omega f}\fstop
	\end{equation}
	Since $\Omega$ is compact and $f$ is continuous and nonconstant,
	$\Omega_+$ and $\Omega_-$ are nonempty, compact, and disjoint.
	Moreover, if $\eta\in\Omega_+$, then
	\begin{equation}
		\max_\Omega f
		=
		e^{t\cL^\tups}f(\eta)
		=
		\mathbb E_\eta[f(\eta_t)]
		\le
		\max_\Omega f\comma
	\end{equation}
	so $\eta_t\in\Omega_+$ almost surely; the same estimate for $-f$ shows that also $\Omega_-$ is
	invariant. 	Hence, each set supports an invariant probability measure. Since
	$\Omega_+$ and $\Omega_-$ are disjoint, these measures are distinct.
	This proves 
	\ref{it:uniq1} $\Longrightarrow$ \ref{it:uniq2}.
	
	Conversely, let
$\mu^{\tups,1}\neq\mu^{\tups,2}$ be invariant probability
measures. By density of $\mathscr P$ in $\cC(\Omega)$, their
restrictions differ on some $\mathscr P_k$. Set
\begin{equation}
	\ell_i
	\eqdef
	\mu^{\tups,i}|_{\mathscr P_k}
	\in\mathscr P_k'
	\comma\qquad i\in\{1,2\}\comma
\end{equation}
where $\mathscr P_k'$ is the algebraic dual of $\mathscr P_k$. Then
$\ell_1\neq\ell_2$, while
$\ell_1(\mathbf1)=\ell_2(\mathbf1)=1$. Hence, they cannot be
nontrivial scalar multiples of one another and are therefore linearly
independent. Further, for the algebraic adjoint
	$
		\tttonde{\cL^\tups|_{\mathscr P_k}}'
		:
		\mathscr P_k'
		\to
		\mathscr P_k'$, 
invariance gives
\begin{equation}
	\tttonde{\cL^\tups|_{\mathscr P_k}}'\ell_1=	\tttonde{\cL^\tups|_{\mathscr P_k}}'\ell_2=0\fstop
\end{equation}
	Thus, zero has geometric, and hence algebraic, multiplicity at least
	two for the adjoint. Since a finite-dimensional operator and its
	adjoint have the same characteristic polynomial, zero has algebraic
	multiplicity at least two also for
	$\cL^\tups|_{\mathscr P_k}$. Therefore,
	$
		\gap\tttonde{\cL^\tups|_{\mathscr P_k}}=0$, 
	and consequently $\gap(\cL^\tups)=0$.
	 This proves \ref{it:uniq2} $\Longrightarrow$ \ref{it:uniq1}.
\end{proof}

\subsection{Examples}\label{sec:SEM-examples} As already mentioned, the ${\rm KMP}$ model with site and block weights $\alpha$ and $w$
corresponds to the stochastic exchange model with the update measure $\Upsilon=\Upsilon^\tKMP$ in \eqref{eq:Upsilon-KMP}. We next discuss
three further classes of stochastic exchange models: averaging-like processes,
immediate exchange models, and the harmonic process. Some of these models with all update blocks of size two were recently studied
in \cite{kim_quattropani_sau_spectral_2025,casanova2025partially}. 	Throughout this section,
we fix $n\ge2$.

\subsubsection{Averaging processes and generalized exchanges}
\label{sec:SEM-AVG}

Fix a probability vector $\pi=(\pi_x)_{x\in[n]}$ with $\pi_x>0$ for every
$x$, and nonnegative block weights $w=(w_B)_{B\subseteq[n]}$. Recall from
\eqref{eq:K_B^U2} that $K_B^\pi$ is the row-stochastic matrix which
redistributes the total mass in $B$ deterministically according to the
conditional probability vector $(\pi_x/\pi(B))_{x\in B}$ and leaves all
coordinates outside $B$ unchanged. The corresponding block averaging process
is obtained from
\begin{equation}\label{eq:Upsilon-AVG}
\textstyle \Upsilon=	\Upsilon^{\tavg}
	\eqdef
	\sum_{B\subseteq[n]}
	w_B\,\delta_{K_B^\pi}\fstop
\end{equation}
This is precisely the update measure underlying the limiting hidden generator
$\mathscr L_\infty$ in \eqref{eq:hidden-generator-infty}; its support is the
set $\supp(\Upsilon_\infty)$ introduced in \eqref{eq:supp-Upsilon-infty}. 
For pairwise blocks, the update on the coordinates $\{x,y\}$ reads
\begin{equation}
	(\eta_x,\eta_y)
	\begin{pmatrix}
		\frac{\pi_x}{\pi_x+\pi_y} &\frac{\pi_y}{\pi_x+\pi_y}\\[.2cm]
		\frac{\pi_x}{\pi_x+\pi_y} &\frac{\pi_y}{\pi_x+\pi_y}
	\end{pmatrix}
	=
	(
	\tfrac{\pi_x}{\pi_x+\pi_y}(\eta_x+\eta_y),
	\tfrac{\pi_y}{\pi_x+\pi_y}(\eta_x+\eta_y)
	)\fstop
\end{equation}
This recovers the weighted averaging process considered in, e.g., 
\cite{aldous_lecture_2012,quattropani2021mixing}; the corresponding
mean-field block dynamics was studied in \cite{caputo_repeated_2024}.
The probability measure $\delta_\pi$ is invariant and, under the usual
connectivity assumption \eqref{eq:connected-hypergraph} on the hypergraph induced by $w$, it is the unique
invariant probability measure. 

In fact, averaging processes satisfy the stronger one-particle spectral gap
identity. For pairwise blocks, this was proved in
\cite[Theorem~2.1]{quattropani2021mixing}; in the present block setting, it
also follows by combining \eqref{eq:gap-lower-infty} with
Theorem~\ref{th:SEM}. Indeed, writing
$\cL^{\tavg}\eqdef\cL^{\Upsilon^{\tavg}}$, one has
\begin{equation}\label{eq:AVG-gap-identity}
	\gap(\cL^{\tavg})
	=
	\gap(\cL^{\tavg}|_{\mathscr P_1})\fstop
\end{equation}

The preceding example belongs to a broader class. Suppose that $\Upsilon$ is
a finite measure on $\mathfrak M$ and that there exists a strictly positive
probability vector $\pi\in\Omega$ such that
\begin{equation}\label{eq:SEM-pi-invariance}
	\pi M=\pi
	\comma\qquad
	\text{for $\Upsilon$-a.e.\ $M$}\fstop
\end{equation}
Thus, $\pi$ is invariant for $\Upsilon$-almost every $M$. Consequently,
$\delta_\pi$ is invariant for the stochastic exchange model, while
$\pi^{\otimes k}$ is invariant for $L_k^\tups$ for every $k\ge1$, since
$\pi^{\otimes k}M^{\otimes k}=\pi^{\otimes k}$ for $\Upsilon$-almost every
$M$.
If, more strongly,
\begin{equation}\label{eq:SEM-pi-detailed-balance}
	D_\pi M=M^{\mathsf T}D_\pi
	\comma\qquad
	\text{for $\Upsilon$-a.e.\ $M$}\comma
\end{equation}
where $D_\pi$ is the diagonal matrix with diagonal $\pi$, then $\pi$ is
reversible for $\Upsilon$-almost every $M$. Indeed, this condition is
equivalent to the detailed-balance equation $\pi_xM_{xy}=\pi_yM_{yx}$, $x,y\in[n]$, and makes
$M^{\otimes k}$ and $L_k^\tups$ self-adjoint on $L^2(\pi^{\otimes k})$. In particular, \eqref{eq:SEM-pi-detailed-balance} guarantees that $\spec(\cL^\tups|_{\mathscr P})$ is real.

While \eqref{eq:SEM-pi-invariance} and \eqref{eq:SEM-pi-detailed-balance} are equivalent for size-two block updates, 
for updates involving three or more sites,
\eqref{eq:SEM-pi-invariance} is strictly weaker than
\eqref{eq:SEM-pi-detailed-balance}. For example, the cyclic permutation
matrix
\begin{equation}
	\begin{pmatrix}
		0&1&0\\
		0&0&1\\
		1&0&0
	\end{pmatrix}
\end{equation}
preserves the uniform probability vector $\pi\equiv 1/n$, but is not reversible with respect
to it.

\begin{remark}[Permutations and the interchange process]\label{rem:IP}
	Permutation updates provide an interesting degenerate example of stochastic exchange models. Suppose that
	$\Upsilon$ is supported on $n\times n$ permutation matrices,  elements
	of $\mathfrak M$. Starting from a configuration $\eta$ with pairwise distinct
	coordinates, and regarding the values $(\eta_x)_{x\in[n]}$ as distinct labels,
	the stochastic exchange model projects onto a block-interchange process;
	see \cite{caputo_proof_2010,bristiel_caputo_entropy_2021,alon2026aldous}.
	
	On the full state space $\Omega$, however, permutation updates preserve the
	unordered collection of coordinates of $\eta$. Consequently, the dynamics
	decomposes into infinitely many invariant subsets and admits infinitely many
	invariant probability measures. In particular, by
	Proposition~\ref{pr:SEM-uniqueness}, its spectral gap vanishes. Thus, although
	Theorem~\ref{th:SEM} applies to permutation updates, the ambient state space
	$\Omega$ is too large to capture the nontrivial spectral-gap identities of the
	corresponding interchange-process projections, such as
	\cite[Theorem~1.1]{caputo_proof_2010} in the graph case and its proposed extensions to hypergraphs
	\cite[Conjecture~1.7]{bristiel_caputo_entropy_2021}.
\end{remark}

\subsubsection{Immediate exchange model}\label{sec:IEM}
We next give a block version of the \textit{immediate exchange model}
\cite{heinsalu_patriarca_kinetic_2014,katriel_immediate_2015,redig_generalized_2017}. As for the ${\rm KMP}$ model, fix positive parameters
$\alpha=(\alpha_x)_{x\in[n]}$ and nonnegative block weights $w=(w_B)_{B\subseteq[n]}$. For every block $B\subseteq[n]$ with $w_B>0$, let
$
A^B=(a^B_{xy})_{x,y\in B}
$
be a symmetric matrix with strictly positive entries satisfying
\begin{equation}\label{eq:IEM-row-sums}\textstyle
	\sum_{y\in B}a^B_{xy}
	=
	\alpha_x
	\comma\qquad x\in B\fstop
\end{equation}
For every $x\in B$, independently sample
\begin{equation}
	U^{(x)}
	=
	(U^{(x)}_y)_{y\in B}
	\sim
	{\rm Dir}\big((a^B_{xy})_{y\in B}\big)
\end{equation}
and let $J_B$ be the stochastic matrix whose $x$-th row equals
$U^{(x)}$ on $B$ for $x\in B$, while $J_B$ agrees with the identity
outside $B$. The immediate exchange model corresponds to the choice
\begin{equation}\label{eq:Upsilon-IEM-block}
	\textstyle
	\Upsilon^{\tiem}
	=
	\sum_{B\subseteq[n]}
	w_B\,{\rm Law}(J_B)\fstop
\end{equation}
For $B=\{x,y\}$, $x\neq y$, after writing
\begin{equation}
	A^B
	=
	\begin{pmatrix}
		\alpha_{xx} & \alpha_{xy}\\
		\alpha_{xy} & \alpha_{yy}
	\end{pmatrix}
	\comma\qquad
	\alpha_{xx}+\alpha_{xy}=\alpha_x
	\comma
	\alpha_{yy}+\alpha_{yx}=\alpha_y\comma \alpha_{xy}=\alpha_{yx}\comma
\end{equation}
and sampling independent random variables
$
	U\sim{\rm Beta}(\alpha_{xx},\alpha_{xy})
$,
$	V\sim{\rm Beta}(\alpha_{yy},\alpha_{yx})$,		
the update restricted to the coordinates $\{x,y\}$ takes the matrix form
\begin{equation}\label{eq:IEM-pair-matrix-update}
	(\eta_x,\eta_y)
	\longmapsto
	(\eta_x,\eta_y)
	\begin{pmatrix}
		U & 1-U\\
		1-V & V
	\end{pmatrix}
	=
	\big(
	U\eta_x+(1-V)\eta_y,
	(1-U)\eta_x+V\eta_y
	\big)\fstop
\end{equation}
Thus, in contrast with the ${\rm KMP}$ update in \eqref{eq:KMP-update-segment-matrix}, the rows in
the updated block are independent rather than identical. 

We claim that $\mu={\rm Dir}(\alpha)$ is reversible for every single-block
update and, consequently, for
$\cL^{\tiem}\eqdef\cL^{\Upsilon^{\tiem}}$. This follows once again from the standard
Gamma representation of the Dirichlet distribution, and symmetry of $A^B$. We leave the details to the reader.
The same symmetry also makes the one-particle dynamics reversible. Indeed,
for distinct $x,y\in[n]$, set
\begin{equation}\label{eq:IEM-txy}
	\textstyle
	t_{xy}
	\eqdef
	\sum_{B\supseteq\{x,y\}}w_B\,a^B_{xy}
	=
	t_{yx}\comma
\end{equation}
and $t_{xx}\eqdef 0$.
Then, the one-particle rate from $x$ to $y$ is $t_{xy}/\alpha_x$, and its
reversible probability measure is $\pi=\alpha/\alpha_0$, with
$\alpha_0\eqdef\alpha([n])$.

The following corollary combines Theorem~\ref{th:SEM} with reversibility
with respect to $\mu={\rm Dir}(\alpha)$ and density of polynomials in
$L^2(\mu)$, whereas the comparison estimate \eqref{eq:IEM-gap-comparison} below extends
\cite[Theorem~1.3]{kim_quattropani_sau_spectral_2025} from graphs to
hypergraphs; we omit the details. Here, the polynomial spaces
$\mathscr P_{k,\ast}$ are those defined in \eqref{eq:P-k-star}.

\begin{corollary}[Immediate exchange model]\label{cor:IEM}
	For the block immediate exchange model,
	$\gap(\cL^{\tiem})$ coincides with the ordinary
	$L^2(\mu)$ spectral gap of $\cL^\tiem$, and satisfies
	\begin{equation}\label{eq:IEM-degree-two-gap}
		\gap(\cL^{\tiem})
		=
		\gap(\cL^{\tiem}|_{\mathscr P_2})
		=
		\min\{
		\gap(\cL^{\tiem}|_{\mathscr P_{1,\ast}}),
		\gap(\cL^{\tiem}|_{\mathscr P_{2,\ast}})
		\}\fstop
	\end{equation}
	Further, 
	\begin{equation}\label{eq:IEM-gap-comparison}
		\gap(\cL^{\tiem}|_{\mathscr P_{2,\ast}})
		\ge
		\gamma_{\tiem}\,
		\gap(\cL^{\tiem}|_{\mathscr P_{1,\ast}})\comma
	\end{equation}
	where (recall \eqref{eq:IEM-txy})
	\begin{equation}\label{eq:gamma-IEM-block}
		\gamma_{\tiem}
		\eqdef
		\left(1+\frac1{\alpha_0}\right)
		\min_{x,y\,:\, t_{xy}>0}
		\frac1{t_{xy}}
		\sum_{B\supseteq\{x,y\}}w_B
		\sum_{z\in B}
		\frac{a^B_{xz}a^B_{zy}}{\alpha_z+1}\fstop
	\end{equation}
	If, additionally, for every $B\subseteq [n]$ with $w_B>0$,
	\begin{equation}\label{eq:IEM-condition}
		\sum_{z\in B}\frac{a_{xz}^Ba_{zy}^B}{\alpha_z+1}\ge a_{xy}^B\comma\qquad x, y \in B\comma x\neq y\comma
	\end{equation}
	then $\gamma_\tiem \ge 1$ and, thus, $\gap(\cL^\tiem)=\gap(\cL^\tiem|_{\mathscr P_{1,\ast}})$.
\end{corollary}
\begin{remark}
	When $B=\{x,y\}$, by \eqref{eq:IEM-row-sums} and the symmetry $a_{xy}^B=a_{yx}^B$, \eqref{eq:IEM-condition} becomes
	\begin{equation}
	a_{xx}^Ba_{yy}^B\ge \tttonde{1+a_{xy}^B}^2\fstop
	\end{equation}
\end{remark}

\subsubsection{Harmonic process}\label{sec:HP}

The \textit{harmonic process}
\cite{frassek2021exact,franceschini_frassek_giardina_integrable_2023} provides an
important example for which the update measure is not finite. Fix symmetric
conductances $b_{xy}=b_{yx}\ge0$, with $b_{xx}=0$, and positive
parameters $\alpha=(\alpha_x)_{x\in[n]}$. For $x\neq y$ and $u\in(0,1)$,
let $H^{xy,u}$ be the stochastic matrix equal to the identity except
for its $x$-th row, which is given by
\begin{equation}
	e_xH^{xy,u}
	=
	u e_x+(1-u)e_y\fstop
\end{equation}
Thus, restricted to the coordinates $\{x,y\}$, the update reads
\begin{equation}
	(\eta_x,\eta_y)
	\begin{pmatrix}
		u & 1-u\\
		0 & 1
	\end{pmatrix}
	=
	\big(u\eta_x,(1-u)\eta_x+\eta_y\big)\fstop
\end{equation}
The harmonic process is built from the update measure
\begin{equation}\label{eq:Upsilon-HP}
	\Upsilon^{\thp}(\dd M)
	=
	\sum_{x,y\in [n]}
	b_{xy}
	\int_0^1
	\frac{u^{\alpha_x-1}}{1-u}
	\,\delta_{H^{xy,u}}(\dd M)\,\dd u\fstop
\end{equation}
The measure $\Upsilon^{\thp}$ has infinite total mass. Nevertheless,
it satisfies \eqref{eq:SEM-assumption}, since the only nonzero
off-diagonal entry of $H^{xy,u}$ is $(H^{xy,u})_{xy}=1-u$ and
\begin{equation}
	\int_0^1
	\frac{u^{\alpha_x-1}}{1-u}
	\tonde{1-u}
	\,\dd u
	=
	\frac1{\alpha_x}
	<\infty\fstop
\end{equation}
The harmonic process is reversible with respect to $\mu={\rm Dir}(\alpha)$;
see \cite{franceschini_frassek_giardina_integrable_2023}.

As similarly done for the immediate exchange model in Corollary \ref{cor:IEM}, we collect the main findings on the harmonic process' spectral gap in the following corollary.
 The bound in \eqref{eq:HP-gap-comparison} follows from 
\cite[Theorem~1.2]{kim_quattropani_sau_spectral_2025}.
 Write
$\cL^{\thp}\eqdef\cL^{\Upsilon^{\thp}}$.

\begin{corollary}[Harmonic process]\label{cor:HP} For the harmonic process, 
	 $\gap(\cL^{\thp})$ coincides with the ordinary
	$L^2(\mu)$ spectral gap of $\cL^{\thp}$ and satisfies
	\begin{equation}\label{eq:HP-degree-two-gap}
		\gap(\cL^{\thp})
		=
		\gap(\cL^{\thp}|_{\mathscr P_2})
		=
		\min\{
		\gap(\cL^{\thp}|_{\mathscr P_{1,\ast}}),
		\gap(\cL^{\thp}|_{\mathscr P_{2,\ast}})
		\}\fstop
	\end{equation}
	Further,
	\begin{equation}\label{eq:HP-gap-comparison}
		\gap(\cL^{\thp}|_{\mathscr P_{2,\ast}})
		\ge
		\gamma_{\thp}\,
		\gap(\cL^{\thp}|_{\mathscr P_{1,\ast}})\comma
	\end{equation}
	where
	\begin{equation}\label{eq:gamma-HP}
		\gamma_{\thp}
		\eqdef
		\left(1+\frac1{\alpha_0}\right)
		\min_{x,y\,:\, b_{xy}>0}
		\left\{
		\frac{\alpha_x}{\alpha_x+1}
		+
		\frac{\alpha_y}{\alpha_y+1}
		\right\}\fstop
	\end{equation}
	In particular, if
	\begin{equation}
		\min_{x,y\,:\,b_{xy}>0}\alpha_x\alpha_y\ge1\comma
	\end{equation}
	then $\gamma_{\thp}\ge1$ and, consequently,
	$
		\gap(\cL^{\thp})
		=
		\gap(\cL^{\thp}|_{\mathscr P_{1,\ast}})$.
\end{corollary}

\appendix

\section{Refining Lemma \ref{lem:full-range}}\label{app:modules-etc}

This appendix proves the refinement of Lemma \ref{lem:full-range} announced in Remark \ref{rem:full-range-refinements}, and crucial in the proof of item \ref{it:hypergraph-dichotomy-mean-field} of Theorems \ref{th:dichotomy-graph} and \ref{th:dichotomy-hypergraph}; see Proposition \ref{pr:mean-field-infty}. The main result of this section is Proposition \ref{pr:full-range-infty}. Its proof requires some technical lemmas, inspired by the works \cite{chein1981partitive,bui2012tree}, and whose complete proofs are postponed to the end of this appendix. The weights $\alpha$ and $w$ are fixed throughout, while $\pi=\alpha/\alpha_0$ is as in \eqref{eq:var-pi-intro}.

Before stating the result, recall
 that $\supp(\Upsilon_\infty)$ in \eqref{eq:supp-Upsilon-infty} contains only update matrices of the form $K_B^\pi$, one for each $B\subseteq[n]$ with $w_B>0$. For such a block, 
\begin{equation}\label{eq:KB-pi-convex-combination}
	K_B^\pi
	=
	\sum_{x\in B}
	\frac{\pi_x}{\pi(B)}
	K_B^{e_x}\comma\quad \text{with}\  \pi(B)= \sum_{y\in B}\pi_y\fstop
\end{equation}
Consequently, every product of matrices from $\supp(\Upsilon_\infty)$ is a linear combination of products of matrices in $\supp(\Upsilon_0)$ from \eqref{eq:supp-Upsilon-0}. Hence, with the notation in \eqref{eq:full-range-Upsilon-0} and \eqref{eq:full-range-Upsilon-infty}, 
\begin{equation}\label{eq:range-inclusion-infty-zero-app}
	\cS^{\tups_\infty}(\theta)
	\subseteq
	\cS^{\tups_0}(\theta)\comma\qquad \theta \in \R^n\fstop
\end{equation}
As observed in Remark \ref{rem:full-range-refinements}, this smaller range already indicates why recovering a full-range property for $\cS^{\tups_\infty}(\theta)$ is a delicate task.

Moreover, the updates preserve the $\pi$-mean and fix the constant configurations, that is, 
\begin{equation}
	\pi(K_B^\pi\theta)
	=
	\pi(\theta)
	\comma
	\qquad
	K_B^\pi\mathbf1
	=
	\mathbf1\fstop
\end{equation}
In particular, 
\begin{equation}\label{eq:H}
\mathscr H_{1,\ast}=
	\big\{
		v\in\R^n:
		\pi(v)=0
		\big\}\comma
\end{equation}
 is invariant under every $K_B^\pi$. Thus, for every $0\neq \theta\in \mathscr H_{1,\ast}$,
$
	\cS^{\tups_\infty}(\theta)
	\subseteq
	\mathscr H_{1,\ast}
	\neq
	\R^n$.
Therefore, $\cS^{\tups_\infty}(\theta)=\R^n$ (cf.\ \eqref{eq:full-range-stronger}) cannot hold uniformly over all nonconstant $\theta$. The appropriate full-range statement is instead the analogue of \eqref{eq:full-range}, that is, \eqref{eq:full-range-Upsilon-infty-appendix} below.

\begin{proposition}\label{pr:full-range-infty}
	Assume that the hypergraph induced by $w$ is minimal; see Section \ref{sec:strong-positivity} and Proposition \ref{pr:strong-positivity-hypergraph}. Then, for every nonconstant $\theta\in\R^n$,
	\begin{equation}\label{eq:full-range-Upsilon-infty-appendix}
		\cS^{\tups_\infty}(\theta)+\R\mathbf1
		=
		\R^n\fstop
	\end{equation}
	Moreover, the hidden model introduced in Section \ref{sec:tau-infty}, whose generator $\mathscr L_\infty$ is given in \eqref{eq:hidden-generator-infty}, is $\mathscr C$-strongly positive on $\mathscr R=\mathscr R_{2,\ast}$: for every $t>0$,
	\begin{equation}\label{eq:strong-positivity-infty}
		e^{t\mathscr L_\infty}
		\tonde{\mathscr C\setminus\{0\}}
		\subseteq
		\interior(\mathscr C)\fstop
	\end{equation}
\end{proposition}

Before presenting the proof of Proposition~\ref{pr:full-range-infty}, we introduce some combinatorial tools related to the theory of partitive families and modular decomposition; see \cite{chein1981partitive} and \cite{bui2012tree} for a modern account. In fact, Lemma~\ref{lem:modules} below, which constitutes a key ingredient in the proof of Proposition~\ref{pr:full-range-infty}, could alternatively be obtained as a direct consequence of the decomposition theorem for partitive families; see, e.g., \cite[Theorem~2]{bui2012tree}. Nevertheless, we include an elementary direct proof in order to keep the argument self-contained.

For a symmetric matrix $Q=(Q_{xy})_{x,y\in[n]}$, a set $A\subseteq[n]$ is called a $Q$-module if every index outside $A$ sees all indices in $A$ in the same way, namely,
\begin{equation}\label{eq:Q-module}
	z\notin A\comma x,y\in A
	\qquad\Longrightarrow\qquad
	Q_{zx}=Q_{zy}.
\end{equation}
Then,
\begin{equation}\label{eq:module-family}
	\mathcal M(Q)
	\eqdef
	\set{A\subseteq[n]:A\text{ is a $Q$-module}}
\end{equation}
denotes the family of all $Q$-modules.

We first record the only closure property of modules that will be used. We say that two sets $A,D\subseteq[n]$ \textit{overlap} if $A\cap D$, $A\setminus D$, and $D\setminus A$ are all nonempty.

\begin{lemma}[Closure of modules]\label{lem:module-closure}
	If $A,D\in\cM(Q)$ overlap, then
	\begin{equation}\label{eq:module-closure}
		A\cap D\comma
		A\cup D\comma
		A\setminus D\comma
		D\setminus A\comma
		A\mathbin{\triangle}D
		\in\cM(Q)\fstop
	\end{equation}
\end{lemma}

\begin{proof}
	Let $A,D\in\cM(Q)$ overlap. If $x,y\in A\cap D$ and $z\notin A\cap D$, then $z\notin A$ or $z\notin D$, and the module property of the corresponding set gives $Q_{zx}=Q_{zy}$. Hence, $A\cap D\in\cM(Q)$.
	
	Next, let $x,y\in A\cup D$ and $z\notin A\cup D$. If $x,y$ both belong to $A$ or both to $D$, the conclusion is immediate. Otherwise, after exchanging $x$ and $y$, we may assume that $x\in A\setminus D$ and $y\in D\setminus A$. Choosing $u\in A\cap D$, we obtain
	\begin{equation}
		Q_{zx}=Q_{zu}=Q_{zy}\fstop
	\end{equation}
	Thus, $A\cup D\in\cM(Q)$.
	
	Consider now $A\setminus D$. Let $x,y\in A\setminus D$ and $z\notin A\setminus D$. If $z\notin A$, the module property of $A$ applies. Otherwise, $z\in A\cap D$. Choose $u\in D\setminus A$. Since $x,y\notin D$, $u\notin A$, and $Q$ is symmetric,
	\begin{equation}
		Q_{zx}
		=
		Q_{xz}
		=
		Q_{xu}
		=
		Q_{ux}
		=
		Q_{uy}
		=
		Q_{yu}
		=
		Q_{yz}
		=
		Q_{zy}\fstop
	\end{equation}
	Hence, $A\setminus D\in\cM(Q)$, and the same argument with $A$ and $D$ exchanged gives $D\setminus A\in\cM(Q)$.
	
	Finally, let $x,y\in A\mathbin{\triangle}D$ and $z\notin A\mathbin{\triangle}D$. The conclusion follows from the difference case if $x,y$ belong to the same one of the two sets $A\setminus D$ and $D\setminus A$. Otherwise, assume that $x\in A\setminus D$ and $y\in D\setminus A$. If $z\notin A\cup D$, choose $u\in A\cap D$ and write $Q_{zx}=Q_{zu}=Q_{zy}$. If $z\in A\cap D$, the module properties of $D$ and $A$ give
	\begin{equation}
		Q_{xz}=Q_{xy}
		\comma
		Q_{yz}=Q_{yx}\fstop
	\end{equation}
	By symmetry, $Q_{zx}=Q_{zy}$. Therefore, $A\mathbin{\triangle}D\in\cM(Q)$.
\end{proof}

\begin{lemma}[Triviality of modules]\label{lem:modules}
	Let $Q=(Q_{xy})_{x,y\in[n]}$ be symmetric, and suppose that every $B\subseteq[n]$ with $w_B>0$ is a $Q$-module, see \eqref{eq:Q-module}.
	If the hypergraph induced by $w$ is minimal, then there exists $c\in\R$ such that
	\begin{equation}\label{eq:Q-offdiag-constant}
		Q_{xy}=c
		\comma
		\qquad x\neq y\fstop
	\end{equation}
\end{lemma}

\begin{proof}
	Fix the family $\cM(Q)$ in \eqref{eq:module-family}. Call a module $A\in\cM(Q)$ \textit{strong} if it overlaps no member of $\cM(Q)$.
	
	\smallskip
	\noindent
	\emph{Step 1: the only nonempty proper strong modules are the singletons.}
	Strong modules form a laminar family: two intersecting strong modules are comparable by inclusion. Every singleton is strong and, since $[n]$ is finite, is contained in a maximal proper strong module. Let
	\begin{equation}
		N_1,\ldots,N_m
	\end{equation}
	be the distinct maximal proper strong modules. They are pairwise disjoint and cover $[n]$, and hence
	\begin{equation}\label{eq:strong-module-partition}
		[n]=N_1\sqcup\cdots\sqcup N_m\fstop
	\end{equation}
	This partition is compatible with every module in $\cM(Q)$. Indeed, if $A\in\cM(Q)$ meets $N_i$, then, since $N_i$ is strong, either $A\subseteq N_i$ or $N_i\subseteq A$. Thus, either $A$ is contained in one atom of \eqref{eq:strong-module-partition}, or it is the union of all atoms that it meets.
	
	In particular, \eqref{eq:strong-module-partition} is $w$-compatible, because every block $B$ with $w_B>0$ belongs to $\cM(Q)$. Minimality excludes $1<m<n$, while $m\neq1$ because every $N_i$ is proper. Therefore, $m=n$, and every $N_i$ is a singleton. Since every proper strong module is contained in a maximal proper one, the claim follows.
	
	\smallskip
	\noindent
	\emph{Step 2: there exists a two-element module.}
	Minimality, together with $w_{[n]}=0$, ensures that some block $B$ with $w_B>0$ satisfies
	\begin{equation}
		2\le \abs{B}\le n-1\fstop
	\end{equation}
	Indeed, if every positive-weight block were a singleton, the two-atom partition $[n]=\{1\}\sqcup([n]\setminus\{1\})$ would be $w$-compatible.
	
	Choose $A\in\cM(Q)$ with $2\le\abs{A}\le n-1$ and minimal cardinality. By Step 1, $A$ is not strong, so there exists $D\in\cM(Q)$ overlapping $A$. Lemma \ref{lem:module-closure} gives
	\begin{equation}
		A\cap D\in\cM(Q)
		\comma
		A\setminus D\in\cM(Q)\fstop
	\end{equation}
	Both sets are nonempty proper subsets of $A$. By the minimality of $\abs{A}$, they are singletons, and hence $\abs{A}=2$.
	
	\smallskip
	\noindent
	\emph{Step 3: all two-element sets are modules.}
	Define a relation on $[n]$ by
	\begin{equation}\label{eq:module-equivalence}
		x\sim y
		\qquad\Longleftrightarrow\qquad
		x=y
		\quad\text{or}\quad
		\{x,y\}\in\cM(Q)\fstop
	\end{equation}
	This is an equivalence relation. Only transitivity requires verification: if $x,y,z$ are pairwise distinct and $\{x,y\},\{y,z\}\in\cM(Q)$, these two modules overlap, and their symmetric difference $\{x,z\}$ belongs to $\cM(Q)$ by Lemma \ref{lem:module-closure}.
	
	Every equivalence class $C$ belongs to $\cM(Q)$. Indeed, this is clear for singletons. Otherwise, fix $x_0\in C$ and enumerate $C=\{x_0,x_1,\ldots,x_r\}$. Starting from $\{x_0,x_1\}\in\cM(Q)$, repeatedly apply the union part of Lemma \ref{lem:module-closure} to the overlapping modules already constructed and $\{x_0,x_k\}$, $k=2,\ldots,r$.
	
	We claim that every equivalence class $C$ is strong. Suppose otherwise, and choose $D\in\cM(Q)$ of minimal cardinality among the modules overlapping $C$. Fix
	\begin{equation}\label{eq:module-overlap-points}
		x\in D\cap C
		\comma
		y\in C\setminus D
		\comma
		z\in D\setminus C\fstop
	\end{equation}
	Since $x\sim y$, one has $\{x,y\}\in\cM(Q)$.
	
	First, $D\cap C=\{x\}$. Otherwise, choosing $x'\in(D\cap C)\setminus\{x\}$, the modules $D$ and $\{x,y\}$ overlap. Lemma \ref{lem:module-closure} then gives
	\begin{equation}
		D\setminus\{x,y\}=D\setminus\{x\}\in\cM(Q)\fstop
	\end{equation}
	This set still overlaps $C$: it contains $x'\in C$, misses $y\in C$, and contains $z\notin C$. Its cardinality is smaller than that of $D$, a contradiction.
	
	Set
	$
		O\eqdef D\setminus\{x\}=D\setminus C$.
	Again by Lemma \ref{lem:module-closure}, $O\in\cM(Q)$. We show that $O$ is a singleton. Suppose that $\abs{O}\ge2$. Then $O$ is a nontrivial proper module and, by Step 1, is not strong. Choose $H\in\cM(Q)$ overlapping $O$.
	
	If $x\notin H$, then $D$ and $H$ overlap, and Lemma \ref{lem:module-closure} gives $D\setminus H\in\cM(Q)$. This set contains $x\in C$, misses $y\in C$, and contains a point of $O\setminus H\subseteq[n]\setminus C$; hence it overlaps $C$. Since $D\cap H\neq\emp$, it is smaller than $D$, a contradiction.
	
	If $x\in H$ and $H\subseteq D$, then $H$ itself overlaps $C$: it contains $x\in C$ and a point of $H\cap O\subseteq[n]\setminus C$, while it misses $y\in C$. Moreover, $O\setminus H\neq\emp$, so $\abs{H}<\abs{D}$, again a contradiction.
	
	Finally, if $x\in H$ and $H\not\subseteq D$, then $D$ and $H$ overlap, and Lemma \ref{lem:module-closure} gives $D\cap H\in\cM(Q)$. This set contains $x\in C$ and a point of $O\cap H\subseteq[n]\setminus C$, while it misses $y\in C$; hence it overlaps $C$. Since $O\setminus H\neq\emp$, it is smaller than $D$, which is impossible.
	
	Consequently, $O=\{z\}$ for some $z\notin C$. But then $D=\{x,z\}\in\cM(Q)$, so $x\sim z$, contradicting $x\in C$ and $z\notin C$. Thus, every equivalence class is strong.
	
	By Step 2, at least one equivalence class contains two distinct points. Since this class is strong, Step 1 forces it to be $[n]$. Therefore,
	\begin{equation}\label{eq:all-pairs-modules}
		\{x,y\}\in\cM(Q)
		\comma
		\qquad x\neq y\fstop
	\end{equation}
	
	We may now conclude. Fix $x\in[n]$. For distinct $y,z\in[n]\setminus\{x\}$, the module property of $\{y,z\}$ gives
	\begin{equation}
		Q_{xy}=Q_{xz}\fstop
	\end{equation}
	Hence, there exists $c_x\in\R$ such that $Q_{xy}=c_x$ for every $y\neq x$. If $x\neq y$, symmetry yields
	\begin{equation}
		c_x=Q_{xy}=Q_{yx}=c_y\fstop
	\end{equation}
	Thus, all the $c_x$ coincide, proving \eqref{eq:Q-offdiag-constant}.
\end{proof}

We are finally in shape to present the proof of Proposition \ref{pr:full-range-infty}.
\begin{proof}[Proof of Proposition \ref{pr:full-range-infty}]
	Decompose $\R^n=\mathscr H_{1,\ast}\oplus \R\mathbf1$, where $\mathscr H_{1,\ast}\subseteq \R^n$ is given in \eqref{eq:H}.
	For every $B\subseteq[n]$ with $w_B>0$, $K_B^\pi$ is the $L^2(\pi)$-orthogonal projection onto the configurations which are constant on $B$. Hence, it is self-adjoint, fixes $\mathbf1$, and leaves $\mathscr H_{1,\ast}$ invariant.
	
	We show that every nonzero subspace $\mathscr W\subseteq \mathscr H_{1,\ast}$ invariant under all $K_B^\pi$, $w_B>0$, equals $\mathscr H_{1,\ast}$. Suppose otherwise, and let $P$ be the $L^2(\pi)$-orthogonal projection onto $\mathscr W\oplus\R\mathbf1$. Since $K_B^\pi$
	is self-adjoint and leaves $\mathscr W\oplus \R\mathbf1$ invariant, it also leaves its orthogonal complement invariant. Hence,
	$
	PK_B^\pi=K_B^\pi P
	$.
	Write $P=(P_{xy})_{x,y\in[n]}$, and set
	\begin{equation}
		Q_{xy}
		\eqdef
		\frac{P_{xy}}{\pi_y}\comma\qquad x,y \in [n]\fstop
	\end{equation}
	The matrix $Q=(Q_{xy})_{x,y\in [n]}$ is symmetric. Moreover, since $P$ commutes with $K_B^\pi$, it also commutes with $I-K_B^\pi$ and therefore leaves $\im(I-K_B^\pi)$ invariant. This range consists precisely of the vectors supported on $B$ with zero $\pi$-mean. Thus, for $x,y\in B$, the vector $
	v
	\eqdef e_x/\pi_x- e_y/\pi_y	$
	belongs to $\im(I-K_B^\pi)$, and so does $Pv$. In particular, $Pv$ vanishes at $z\notin B$, yielding
	$
	0
	=
	(Pv)_z
	=
	{P_{zx}}/{\pi_x}
	-
	{P_{zy}}/{\pi_y}
	=
	Q_{zx}-Q_{zy}$.
	Hence, for every $B$ with $w_B>0$,	\begin{equation}\label{eq:Q-module-app}
		Q_{zx}=Q_{zy}
		\comma
		\qquad
		z\notin B\comma x,y\in B\fstop
	\end{equation}
	By Lemma \ref{lem:modules}, symmetry of $Q$, \eqref{eq:Q-module-app}, and our assumption of hypergraph minimality,
	\begin{equation}\label{eq:Q-offdiag-app}
		Q_{xy}=c\comma\qquad x \neq y\comma 
	\end{equation}
	for some $c\in\mathbb R$.
	Since $P\mathbf1=\mathbf1$, \eqref{eq:Q-offdiag-app} yields
	\begin{equation}\label{eq:Q-diag-app}
		\pi_xQ_{xx}
		=
		1-c+c\pi_x
		\comma
		\qquad x\in[n]\fstop
	\end{equation}
	Using $P^2=P$ in an off-diagonal entry $x\neq y$ and \eqref{eq:Q-offdiag-app}--\eqref{eq:Q-diag-app}, we obtain
	\begin{equation}\textstyle
		c
		=
		Q_{xy}
		=
		\sum_{z\in[n]}\pi_zQ_{xz}Q_{zy}=
		c(1-c+c\pi_x)
		+
		c(1-c+c\pi_y)
		+
		c^2(1-\pi_x-\pi_y)
		=
		2c-c^2\fstop
	\end{equation}
	Thus $c\in\{0,1\}$. If $c=0$, then $P=I$; if $c=1$, then $P_{xy}=\pi_y$ for all $x,y\in [n]$, so $P$ is the projection onto $\R\mathbf1$. Both alternatives contradict $0\neq \mathscr W\subsetneq \mathscr H_{1,\ast}$. Hence $\mathscr W=\mathscr H_{1,\ast}$.
	
	Finally, fix a nonconstant $\theta \in \R^n$, and let $\bar \theta \eqdef \theta-\pi(\theta)\mathbf1$.
	Since $\theta$ is nonconstant, $0\neq \bar\theta\in \mathscr H_{1,\ast}$. The linear span
	\begin{equation}
		\cS^{\tups_\infty}(\bar \theta)\eqdef\spanop\big\{K_\ell\cdots K_1\bar\theta: \ell\ge 0\,,\, K_1,\ldots,K_\ell \in \supp(\Upsilon_\infty)\big\}
	\end{equation}
	is a nonzero common invariant subspace of $\mathscr H_{1,\ast}$, and therefore equals $\mathscr H_{1,\ast}$. As every $K_B^\pi$ fixes $\mathbf1$,
	\begin{equation}
		K_B^\pi\theta
		=
		K_B^\pi\bar \theta+\pi(\theta)\mathbf1\fstop
	\end{equation}
	Hence,  $\cS^{\tups_\infty}(\theta)+\R\mathbf1
	=
	\cS^{\tups_\infty}(\bar\theta)+\R\mathbf1
	=
	\mathscr H_{1,\ast}+\R\mathbf1
	=
	\R^n$, namely, \eqref{eq:full-range-Upsilon-infty-appendix}, as desired.

	It remains to prove \eqref{eq:strong-positivity-infty}. Fix $t>0$ and $0\neq g\in\mathscr C$. By Markovianity,
	$e^{t\mathscr L_\infty}g\in\mathscr C$.
	Suppose that this function does not belong to $\interior(\mathscr C)$. By \eqref{eq:cone-interior}, there exists a nonconstant $\theta\in\R^n$ such that
	$\tttonde{e^{t\mathscr L_\infty}g}(\theta)=0$.
	The propagation-of-zeros argument in Lemma \ref{lem:propagation-zeros}, applied to the update support $\Upsilon_\infty$ from Remark \ref{rem:full-range-refinements}, shows that $g$ vanishes on $\cS^{\tups_\infty}(\theta)$. Moreover, by definition of $\mathscr R=\mathscr R_{2,\ast}$, every $g\in\mathscr R$ vanishes on $\R\mathbf1$, and the zero set of the nonnegative quadratic form $g$ is a linear subspace. Hence, $g$ vanishes on $\cS^{\tups_\infty}(\theta)+\R\mathbf1$. Thus, by \eqref{eq:full-range-Upsilon-infty-appendix},  $g=0$, a contradiction. This proves \eqref{eq:strong-positivity-infty}.
\end{proof}

\section{Feller property of stochastic exchange models}\label{app:SEM-Feller}
In this section we present the proof Proposition \ref{pr:SEM-Feller}, which essentially amounts to verify the standard criteria in \cite{liggett_interacting_2005}.
\begin{proof}[Proof of Proposition \ref{pr:SEM-Feller}]
	For every $\eta\in\Omega$ and $M\in\mathfrak M$, using that the rows of $M$ are probability vectors,
	\begin{equation}\label{eq:SEM-increment-bound}\textstyle
		\norm{\eta M-\eta}_1
		\eqdef
		\sum_{x\in[n]}|(\eta M)_x-\eta_x|
		\le
		\sum_{x\in[n]}\norm{M_{x\emparg}-e_x}_1
		=
		2\sum_{x\in[n]}\tonde{1-M_{xx}}\fstop
	\end{equation}
	Now fix $k\ge0$ and $f\in\mathscr P_k$. Since $f$ is Lipschitz on the compact set $\Omega$, 	\eqref{eq:SEM-increment-bound} and \eqref{eq:SEM-assumption}
	show that the integral in \eqref{eq:gen-SEM} is well defined.
	Moreover, for every $M\in\mathfrak M$, the map
	$\eta\mapsto f(\eta M)-f(\eta)$ belongs to $\mathscr P_k$, with
	$M$-dependent coefficients which, by $0\le M_{xy}\le 1$ and the preceding bound, are all
	$\Upsilon$-integrable under \eqref{eq:SEM-assumption}. Thus, the
	integral in \eqref{eq:gen-SEM} is well defined and belongs to
	$\mathscr P_k$, proving the first inclusion in
	\eqref{eq:SEM-poly-preservation}.

	Similarly, for $\theta\in[a,b]^n$,
	\begin{equation}\label{eq:SEM-hidden-increment-bound}\textstyle	
		\norm{M\theta-\theta}_1 =	 \sum_{x\in [n]}|\sum_{y\neq x}M_{xy}\tonde{\theta_y-\theta_x}|
		\le
		(b-a)\sum_{x\in[n]}\tonde{1-M_{xx}}\fstop
	\end{equation}
	Thus, the same argument yields the well definition of
	$\mathscr L^\tups$ on $\mathscr R|_{[a,b]^n}$, and
	the second inclusion in \eqref{eq:SEM-poly-preservation}.
	
	Given these facts, the construction follows from the standard theory of Markov pregenerators. We treat only $\cL^\tups$, the argument for $\mathscr L^\tups$ being identical.
	
	According to \cite[Definition~I.2.1]{liggett_interacting_2005},
	$(\cL^\tups,\mathscr P)$ is a Markov pregenerator on $\cC(\Omega)$. Indeed, $\mathscr P$ given in \eqref{eq:P-R-SEM} is dense in $\cC(\Omega)$, contains the constants,
	$\cL^\tups\mathbf 1=0$, and $\cL^\tups$ satisfies the positive
	minimum principle: if $f\in\mathscr P$ and $\eta\in\Omega$ satisfy
	$f(\eta)=\min_{\Omega}f$, then
	\begin{equation}
		\textstyle
		\cL^\tups f(\eta)
		=
		\int\Upsilon(\dd M)\,
		\tonde{f(\eta M)-f(\eta)}
		\ge0\comma
	\end{equation}
	since $\eta M\in\Omega$ for every $M\in\mathfrak M$. By \cite[ Proposition I.2.2]{liggett_interacting_2005}, this yields condition (c) in \cite[Definition I.2.1]{liggett_interacting_2005}.
	Moreover, for every $\lambda>0$, 	$(I-\lambda\cL^\tups)\mathscr P=\mathscr P$.
	Indeed, by \eqref{eq:SEM-poly-preservation}, $I-\lambda\cL^\tups$
	leaves each $\mathscr P_k$ invariant, and it is
	injective there by the positive minimum principle; hence, by finite dimensionality of the space $\mathscr P_k$, it is bijective.
	
	Therefore, by \cite[Propositions~I.2.5--6]{liggett_interacting_2005},
	the closure $\bar \cL^\tups$ is a Markov pregenerator and
	$I-\lambda\bar \cL^\tups$ has closed range. Since this range
	contains the dense subspace $\mathscr P$, it equals $\cC(\Omega)$.
	Thus, $\bar \cL^\tups$ is a Markov generator \cite[Definition~I.2.7]{liggett_interacting_2005} and, by
	\cite[Theorem~I.2.9]{liggett_interacting_2005}, generates a unique
	Feller process on $\Omega$.
\end{proof}

\subsection*{Acknowledgments} 
This work was supported in part by the Italian
Ministry of Foreign Affairs and International Cooperation, grant number BR26GR05. While
this work was written, the authors were associated to INdAM (Istituto Nazionale di Alta Matematica
“Francesco Severi”) and the group GNAMPA. MQ and FS acknowledge partial support from the GNAMPA-INdAM project ``\emph{Stochastic exchange models: from kinetic theory to opinion dynamics}''. The authors acknowledge the use of ChatGPT (GPT-5.6 Sol) as an auxiliary tool for exploratory work on proofs, identifying potentially relevant  literature, and improving the typesetting and the presentation of the manuscript; responsibility for all mathematical content rests entirely with the authors.

\subsection*{Notation guide}
For the reader's convenience, we collect the main recurring notation in order of appearance. All objects are defined at the indicated locations.
Superscripts indicate restrictions to a block, as in $\alpha^B$ and
$\eta^B$, whereas parentheses denote the corresponding total masses, as in
$\alpha(B)=\sum_{x\in B}\alpha_x$ and
$\eta(B)=\sum_{x\in B}\eta_x$. Conditional quantities are denoted by
$\mu_B(f)$ and
$\pi_{x|B}=\pi_B(\mathbf1_x)=\pi(\mathbf1_x\mid B)=\pi_x/\pi(B)$.
For a linear operator $A$ defined on a vector space of functions,  $\gap(A)$ always denotes the smallest (real part of an) eigenvalue of $-A$, possibly excluding the eigenvalue associated to constant functions.

\begingroup
\small
\setlength{\tabcolsep}{4pt}
\setlength{\extrarowheight}{1.5pt}
\renewcommand{\arraystretch}{1.06}

\begin{longtable}{@{}
		>{\centering\arraybackslash}m{.20\textwidth}|
		>{\centering\arraybackslash}m{.56\textwidth}|
		>{\centering\arraybackslash}m{.16\textwidth}
		@{}}
	\textbf{Symbol} & \textbf{Brief description} & \textbf{Ref.}\\
	\hline
	\endfirsthead
	\textbf{Symbol} & \textbf{Brief description} & \textbf{Ref.}\\
	\hline
	\endhead
	$[n]$ & Site set & Sec.~\ref{sec:KMP-intro}\\
	$\Omega$ & Energy simplex & Sec.~\ref{sec:KMP-intro}\\
	$\alpha$ & Site weights & \eqref{eq:B_up}\\
	$w$ & Block rates & \eqref{eq:B_up}\\
	$\eta^B$ & Coordinates in $B$ & \eqref{eq:B_up}\\
	$\alpha^B$ & Weights in $B$ & \eqref{eq:B_up}\\
	$\eta(B)$ & Energy in $B$ & \eqref{eq:B_up}\\
	$\mu$ & Dirichlet law & \eqref{eq:gen-KMP}\\
	$\mu_B$ & Block conditional expectation & \eqref{eq:gen-KMP}\\
	$\cL$ & ${\rm KMP}$ generator & \eqref{eq:gen-KMP}\\
	$\gap(\cL)$ & Spectral gap & \eqref{eq:gapac}\\
	$\mathscr P_k$ & Polynomials of degree at most $k$ & \eqref{eq:poly-invariance}\\
	$\widehat\psi$ & Energy polynomial & \eqref{eq:intro-widehat-psi}\\
	$L_k$ & Labeled particle generator & \eqref{eq:intro-intertwining}\\
	$r_{xy}(\alpha,w)$ & One-particle jump rates & \eqref{eq:RW-rates}\\
	$\alpha(B)$ & Total weight of $B$ & \eqref{eq:RW-rates}\\
	$\mu_k$ & $k$-particle law & \eqref{eq:mu-k}\\
	$\alpha_0$ & Total site weight & \eqref{eq:mu-k}\\
	$\mathfrak n_x$ & Particles at site $x$ & \eqref{eq:mu-k}, \eqref{eq:N-x-bd}\\
	$\mathscr H_k$ & Symmetric coefficient space & \eqref{eq:coeff-space-sym}\\
	$L_k^\sym$ & Unlabeled particle generator & \eqref{eq:L-k-sym}\\
	$\gap_k(\alpha,w)$ & Gap on $\mathscr P_k$ & \eqref{eq:def-gap-k-alpha-w}\\
	$\mathscr P_{k,\ast}$ & Pure degree-$k$ component & \eqref{eq:P-k-star}\\
	$\gap_{k,\ast}(\alpha,w)$ & Gap on $\mathscr P_{k,\ast}$ & \eqref{eq:gap-inf-star-intro}\\
	$\alpha_{w,{\rm min}}$ & Minimal updated-block weight & \eqref{eq:gap-1-2-quantitative}\\
	$\tau_c(\alpha,w)$ & Transition threshold & Thm.~\ref{th:phase-transition}\\
	$N_1,\ldots,N_m$ & Compatible/coarsest partition & \eqref{eq:partition}, \eqref{eq:partition-coarsest-compatible}\\
	$\widetilde w$ & Inter-atom block weights & \eqref{eq:weights-tilde}\\
	$p$ & Sphere exponents & \eqref{eq:gen-cone}\\
	$\mathbb S_p$ & Inhomogeneous sphere & \eqref{eq:gen-cone}\\
	$\kappa$ & Cone measure & \eqref{eq:gen-cone}\\
	$\cG$ & Gibbs-sampler generator & \eqref{eq:gen-cone}\\
	$\omega$ & Reservoir rates & \eqref{eq:gen-res}\\
	$\beta$ & Reservoir shapes & \eqref{eq:xi-x-u-x}\\
	$\rho$ & Reservoir scales & \eqref{eq:xi-x-u-x}\\
	$U_x$ & Retained-energy fraction & \eqref{eq:xi-x-u-x}\\
	$\xi_x$ & Injected energy & \eqref{eq:xi-x-u-x}\\
	$\mathbf E_x$ & Reservoir expectation & \eqref{eq:xi-x-u-x}\\
	$\Omega_{\rm bd}$ & Boundary state space & \eqref{eq:gen-res}\\
	$\cL_\partial$ & Reservoir generator & \eqref{eq:gen-res}\\
	$\cL_{\rm bd}$ & Boundary-driven generator & \eqref{eq:gen-bd}\\
	$\nu$ & Boundary invariant law & \eqref{eq:nu-rhoa}\\
	$\mathscr Q_k$ & Boundary polynomials of degree at most $k$ & \eqref{eq:Qk-invariance-bd}\\
	$\mathscr Q$ & All boundary polynomials & \eqref{eq:def-polynomial-gap-bd}\\
	$\Upsilon$ & Update-matrix measure & \eqref{eq:gen-SEM-intro}\\
	$\cL^\tups$ & Exchange-model generator & \eqref{eq:gen-SEM-intro}\\
	$K_B^U$ & ${\rm KMP}$ update matrix & \eqref{eq:K_B^U}\\
	$\theta$ & Hidden profile & \eqref{eq:hidden-update-segment}\\
	$\widetilde\psi$ & Hidden polynomial & \eqref{eq:intro-tilde-psi}\\
	$\mathscr L$ & Hidden generator & \eqref{eq:intro-forward-backward-intertwining}\\
	$\pi$ & Reference probability vector, usually $\alpha/\alpha_0$ & \eqref{eq:var-pi-intro}\\
	$\pi(\theta)$ & Weighted mean & \eqref{eq:var-pi-intro}\\
	$\var_\pi$ & Spatial variance & \eqref{eq:var-pi-intro}\\
	$\delta(\alpha,w)$ & Variance-contraction rate & \eqref{eq:def-delta}\\
	$\mathscr H_k^\per$ & Labeled coefficient space & \eqref{eq:coeff-space}\\
	$\mathbf E_B$ & Block-update expectation & \eqref{eq:gen-KMP2}\\
	$\Pi_{B,k}$ & Particle heat-bath operator & \eqref{eq:Pi-B-k-matrix}\\
	$\cL_k^\per$ & Tensor energy generator & \eqref{eq:gen-KMP-tensor}\\
	$\mathscr L_k^\per$ & Tensor hidden generator & \eqref{eq:gen-hidden-tensor}\\
	$\widehat\psi^\per$ & Tensor energy polynomial & \eqref{eq:hat-tensor}\\
	$\widetilde\psi^\per$ & Tensor hidden polynomial & \eqref{eq:tilde-tensor}\\
	$\mathscr R_k$ & Homogeneous hidden polynomials & \eqref{eq:R-k}\\
	$\mathscr P_k^\per$ & Separately linear polynomials on $\Omega^k$ & \eqref{eq:def-P-k-tensor}\\
	$\mathscr R_k^\per$ & Multilinear hidden polynomials on $(\R^n)^k$ & \eqref{eq:def-R-k-tensor}\\
	$\mathfrak a_{k,i}$ & Particle-removal operator & \eqref{eq:particle-removal}\\
	$\mathfrak b_{k-1,i}$ & Particle-addition operator & \eqref{eq:particle-addition}\\
	$\mathfrak a_k$ & Symmetric removal operator & \eqref{eq:decomp-particle}\\
	$\mathfrak b_{k-1}$ & Symmetric addition operator & \eqref{eq:H-k-star-symmetric}\\
	$\mathscr H_{k,\ast}$ & New symmetric component & \eqref{eq:H-k-star-symmetric}\\
	$\mathscr H_{k,\ast}^\per$ & New labeled component & \eqref{eq:H-k-star-symmetric}\\
	$\mathscr R_{k,\ast}$ & New symmetric hidden component & \eqref{eq:R-k-ast}\\
	$\mathscr R_{k,\ast}^\per$ & New labeled hidden component & \eqref{eq:R-k-ast-per}\\
	$\gap_{k,\ast}^\per(\alpha,w)$ & Labeled pure-degree gap & \eqref{eq:gap-k-ast-per}\\
	$\sigma_J$ & Walsh monomial & \eqref{eq:Walsh-decomposition-Gibbs}\\
	$\mathscr R$ & Quadratic hidden space & \eqref{eq:R=R_2-ast}\\
	$\mathscr C$ & Nonnegative quadratic cone & \eqref{eq:cone-C-2-star}\\
	$\mathscr R'$ & Algebraic dual of $\mathscr R$ & \eqref{eq:dual-cone-C-2-star}\\
	$\mathscr C'$ & Dual cone & \eqref{eq:dual-cone-C-2-star}\\
	$g_\ast$ & Right Perron eigenfunction & \eqref{eq:PF-right-left}\\
	$h_\ast$ & Left Perron eigenfunction & \eqref{eq:PF-right-left}\\
	$\bar\alpha$ & Quotient site weights & \eqref{eq:weights-quotient}\\
	$\bar w$ & Quotient block weights & \eqref{eq:weights-quotient}\\
	$\lambda_\tau$ & Scaled two-particle gap & \eqref{eq:lambda-tau}\\
	$g_\tau$ & Right Perron eigenfunction & \eqref{eq:lambda-tau}\\
	$h_\tau$ & Left Perron eigenfunction & \eqref{eq:lambda-tau}\\
	$\pi_{x|B}$ & Conditional site weights & \eqref{eq:var-pi-conditional-pi}\\
	$\pi_B$ & Conditional weighted mean & \eqref{eq:var-pi-conditional-pi}\\
	$\var_{\pi_B}$ & Conditional variance & \eqref{eq:var-pi-conditional-var}\\
	$\mathscr L_\infty$ & Hidden-model generator at $\tau=\infty$ & \eqref{eq:hidden-generator-infty}\\
	$\lambda_\infty$ & Quadratic gap at $\tau=\infty$ & \eqref{eq:lambda-infty}\\
	$\cE_1$ & One-particle Dirichlet form & \eqref{eq:dirichlet-form-1}\\
	$\mathscr L_0$ & Hidden-model generator at $\tau=0$ & \eqref{eq:hidden-generator-zero}\\
	$\lambda_0$ & Quadratic gap at $\tau=0$ & \eqref{eq:lambda-zero}\\
	$L_2^\dagger$ & Killed two-particle generator & \eqref{eq:L-two-dagger}\\
	$[n]^2_\tneq$ & Off-diagonal two-particle states & \eqref{eq:Hneq}\\
	$L_{\partial,k}$ & Reservoir killing operator & \eqref{eq:L-partial-k-bd}\\
	$L_{{\rm bd},k}$ & Labeled killed generator & \eqref{eq:L-bd-k}\\
	$L_{{\rm bd},k}^\sym$ & Symmetric killed generator & \eqref{eq:L-bd-k}\\
	$T_k$ & Killing time & Lem.~\ref{lem:bd-killing-comparison}\\
	$\mathfrak M$ & Row-stochastic matrices & Sec.~\ref{sec:SEM-basic}\\
	$\mathscr L^\tups$ & Hidden exchange-model generator & \eqref{eq:gen-SEM-hidden}\\
	$\mathscr P$ & Polynomial core on $\Omega$ & \eqref{eq:P-R-SEM}\\
	$\mathscr R|_{[a,b]^n}$ & Hidden polynomial core on $[a,b]^n$ & \eqref{eq:P-R-SEM}\\
	$\mu^\tups$ & Exchange-model invariant law & Rem.~\ref{rem:SEM-invariant-measures}\\
	$M^{\otimes k}$ & Independent $k$-label update & \eqref{eq:L-k-Gamma}\\
	$L_k^\tups$ & Exchange-model particle generator & \eqref{eq:L-k-Gamma}\\
	$r_{xy}^\tups$ & One-particle exchange rates & \eqref{eq:L-k-Gamma}\\
	$\mu_k^\tups$ & Invariant $k$-particle moment law & \eqref{eq:mu-k-Gamma}\\
	$\varpi_k$ & Symmetric reference law on $[n]^k$ & Sec.~\ref{sec:SEM-particles}\\
	$L_k^{\tups,\tvarpi}$ & $\varpi_k$-adjoint particle operator & \eqref{eq:SEM-varpi-adjoint}\\
	$\widetilde\psi^\tvarpi$ & Reference-measure hidden polynomial & \eqref{eq:tilde-psi-varpi-Gamma}\\
	$\mathscr H_{k,\ast}^\varpi$ & $\varpi_k$-orthogonal pure component & \eqref{eq:SEM-invariance}\\
	$\mathscr P_k/\mathscr P_{k-1}$ & Pure degree-$k$ polynomial quotient & \eqref{eq:SEM-layer-spectrum}\\
	$K_B^\pi$ & Deterministic $\pi$-averaging matrix & \eqref{eq:K_B^U2}\\
	$\Upsilon^{\tavg}$ & Averaging-process update measure & \eqref{eq:Upsilon-AVG}\\
	$\cL^{\tavg}$ & Averaging-process generator & \eqref{eq:AVG-gap-identity}\\
	$A^B$ & Block-IEM parameter matrix & \eqref{eq:IEM-row-sums}\\
	$J_B$ & Block-IEM update matrix & \eqref{eq:Upsilon-IEM-block}\\
	$\Upsilon^{\tiem}$ & Block-IEM update measure & \eqref{eq:Upsilon-IEM-block}\\
	$\cL^{\tiem}$ & Block-IEM generator & \eqref{eq:IEM-degree-two-gap}\\
	$t_{xy}$ & Effective IEM conductances & \eqref{eq:IEM-txy}\\
	$\gamma_{\tiem}$ & IEM degree-two comparison factor & \eqref{eq:gamma-IEM-block}\\
	$b_{xy}$ & Harmonic-process conductances & \eqref{eq:Upsilon-HP}\\
	$H^{xy,u}$ & Harmonic-process transfer matrix & \eqref{eq:Upsilon-HP}\\
	$\Upsilon^{\thp}$ & Harmonic-process update measure & \eqref{eq:Upsilon-HP}\\
	$\cL^{\thp}$ & Harmonic-process generator & \eqref{eq:HP-degree-two-gap}\\
	$\gamma_{\thp}$ & Harmonic-process comparison factor & \eqref{eq:gamma-HP}\\
\end{longtable}
\endgroup

%\bibliographystyle{alpha}
%\bibliography{bookshelf}

\end{document}